%% file: main-arxiv.tex
\documentclass{article}

\usepackage[english]{babel}

\usepackage[letterpaper,top=2cm,bottom=2cm,left=3cm,right=3cm,marginparwidth=1.75cm]{geometry}

\input{config}

\title{Stochastic Saddle Avoidance Beyond Unit Excitation and Smoothness: A Pathwise Lyapunov--Perron Framework}
\author{Junwen Qiu\thanks{Industrial Systems Engineering and Management, National University of Singapore, Singapore.
  } \and 
        Bohao Ma\thanks{School of Data Science, The Chinese University of Hong Kong, Shenzhen, China. \\ Email: \{jwqiu@nus.edu.sg, bohaoma@link.cuhk.edu.cn, andremilzarek@cuhk.edu.cn,
  junyuz@nus.edu.sg\} } \and Andre Milzarek${}^\dagger$ \and Junyu Zhang${}^*$ }
\begin{document}
\maketitle

\begin{abstract}
Unit excitation (UE) is a common assumption in stochastic saddle avoidance: the stochastic error must have a uniformly positive component along every direction, in expectation. This condition gives a direct way to rule out convergence to strict saddles, but it also oversimplifies the actual noise structure, and does not match many stochastic optimization regimes. In overparameterized or interpolation models, the noise may vanish near stationarity. In finite-sum problems, the stochastic gradient noise may lie in a low-dimensional, data-dependent subspace. In these (common) scenarios, UE is naturally not satisfied.
In this paper, we prove an abstract almost sure avoidance theorem for stochastic recursions without UE\@. The theorem replaces UE-type requirements by verifiable pathwise conditions. In applications, these conditions follow, e.g., from local smoothness and finite-moment assumptions under standard i.i.d.\ sampling, or from the finite-sum structure under without-replacement sampling. Since the stochastically sampled maps generally do not share a fixed point, the celebrated center-stable manifold argument used in deterministic analyses is not directly applicable. Instead, we use a path-dependent change of variables together with a pathwise Lyapunov--Perron-based proof strategy. As applications, we obtain strict saddle avoidance for stochastic mirror descent (including SGD) and for random reshuffling. For nonsmooth composite objectives, we prove avoidance results for a proximal-type stochastic gradient method. Combining these insights with suitable iterate convergence guarantees allows us to establish convergence to local minimizers of the original objective function.
\end{abstract}
\section{Introduction}

First-order methods are among the primary workhorses for modern large-scale nonconvex optimization. Under standard geometric assumptions, such as the Kurdyka-{\L}ojasiewicz property, many first-order methods are known to have convergent trajectories whose limits are stationary \cite{absil2005convergence,AttBol09,AttBolRedSou10,AttBolSva13}. However, this does not imply convergence to local minimizers because local maximizers and saddle points can also be stationary points. One common setting in which this gap can be closed is the strict saddle property, under which every stationary point is either a local minimizer or a strict saddle --- a point having inherent negative curvature directions. This property appears in many structured nonconvex landscapes, including the eigenvector problem and PCA \cite{SunQuWright2015NotScary}, orthogonal tensor decomposition \cite{GeHuangJinYuan2015}, dictionary learning \cite{SunQuWright2017DictionaryI}, phase retrieval \cite{SunQuWright2016PhaseRetrieval}, and low-rank matrix recovery \cite{BhojanapalliNeyshaburSrebro2016LowRank}. For such problems, once the iterates are known to converge to stationary points, convergence to a local minimizer reduces to ruling out convergence to strict saddles. This motivates the study of strict saddle avoidance properties.

Lee et al.\ \cite{LeePPSJR2019} study first-order deterministic methods through the iterative update $z^{k+1}=T(z^k)$. At a strict saddle $z^*$, the derivative ${\rm D}T(z^*)$ has at least one unstable direction. The center-stable manifold theorem \cite{Vanderbauwhede1989} then implies that the set of initial points converging to strict saddles has measure zero. The framework in \cite{LeePPSJR2019} covers gradient descent, manifold gradient descent, mirror descent, and coordinate descent methods with constant step-sizes. The subsequent works \cite{panageas2019first,mucsat2026non} extend such avoidance guarantees to the setting where diminishing step-sizes are used. Davis and Drusvyatskiy \cite{DavisDrusvyatskiy2022} establish saddle avoidance for proximal-type methods applied to nonsmooth problems via the active manifold theory.

Modern optimization problems can often be formulated as large-scale finite-sum problems or online problems with streaming data, where deterministic methods with full gradient computation can be expensive or unavailable. Stochastic algorithms therefore update the iterate using noisy or incomplete first-order information \cite{robbins1951stochastic,bottou2018optimization}. Classical results of Pemantle and Bena\"{\i}m prove almost sure non-convergence to unstable points for stochastic approximation methods \cite{pemantle1990nonconvergence,benaim1996dynamical,benaim1999dynamics}, and this line of analysis remains the broader basis of the modern stochastic saddle avoidance theory \cite{mertikopoulos2020almost,BianchiHachemSchechtman2023,liuYuan2023almostSureSaddleAvoidance,DavisDrusvyatskiyJiang2025}. 

Nevertheless, there are three significant gaps between this line of theory and the stochastic algorithms used in current optimization practice: (i) The usual excitation assumptions, such as UE, are much stronger than what practical stochastic oracles provide; (ii) For nonsmooth stochastic problems, existing saddle avoidance results rely on UE, and in addition, 
work on a tilted (perturbed) objective function that differs from the original problem; (iii) Existing analyses mostly rely on fresh samples that are independent of the past and hence generate martingale-difference noises, whereas the practically used without-replacement schemes such as random reshuffling produce dependent and conditionally biased stochastic gradients. Next, we expand upon the mentioned limitations and provide further details.

\paragraph{Limitation 1: Unit excitation.} A common approach in stochastic saddle avoidance theory is to impose nondegenerate noise conditions that keep the stochastic errors ``excited'' along escaping directions \cite{pemantle1990nonconvergence,benaim1996dynamical,benaim1999dynamics,mertikopoulos2020almost,BianchiHachemSchechtman2023,DavisDrusvyatskiyJiang2025}. Let $\Exp_k[\cdot]$ denote the conditional expectation with respect to the randomness generated before iteration $k$. The \emph{unit excitation} condition\footnote{Some works instead assume a uniform noise condition, for instance $e^k \sim \mathrm{Unif}(B_r(0))$ for some fixed $r>0$.} requires that there exists a constant $b>0$ such that the stochastic errors $\{e^k\}_{k\in\N}$ satisfy 
\begin{equation}\label{eq:assumption UE}
{\inf}_{k\in\N} \;\Exp_k\big[\max\{0,\langle e^k,v\rangle\}\big]\ge b
\qquad \text{for every unit vector } v\in\Rd. \tag{UE}
\end{equation}  
Near a strict saddle, \eqref{eq:assumption UE} ensures excitation along unstable directions. This mechanism underlies many stochastic avoidance proofs,
but it does not match several standard regimes in stochastic optimization, making the \eqref{eq:assumption UE} condition an oversimplification. 

First, \eqref{eq:assumption UE} excludes vanishing stochastic noise. Since
$\max\{0,\langle e^k,v\rangle\}\le \|e^k\|$, the condition cannot hold when
$\Exp_k[\|e^k\|]$ becomes arbitrarily small. This is incompatible with
variance reduction methods such as SVRG and SAGA, whose gradient estimators
become increasingly accurate near stationary points
\cite{johnson2013accelerating,xiao2014proximal,defazio2014saga}. The same
issue also appears when applying classical SGD in an interpolation setting \cite{schmidt2013fast,vaswani2019fast,gower2021sgd}. In particular, for an expected loss $f(x)=\Exp_\xi[f(x;\xi)]$, interpolation is said to hold at a stationary point $x^*$ if $\nabla f(x^*;\xi)=0$ for almost every $\xi$. Hence, the stochastic noise $e^k$ vanishes as $x^k\to x^*$, contradicting \eqref{eq:assumption UE}.
Second, \eqref{eq:assumption UE} requires nondegeneracy in all directions. Even when the stochastic noise does not vanish, finite-sum sampling may not excite all directions. Consider the problem $f(x):=\frac1n\sum_{i=1}^n f_i(x)$ and SGD with index sampling. In this case, the stochastic error is given by $\nabla f_i(x)-\nabla f(x)$, $i \in \{1,\dots,n\}$. These errors lie in the span of the component gradients at $x$, i.e.,
\[
    \cS(x):=\operatorname{span}\{\nabla f_i(x)-\nabla f(x):i=1,\ldots,n\}.
\]
Hence, the noise associated with any mini-batch lies in $\cS(x)$, with dimension $\dim \cS(x)\le n-1$. When $n \leq d$, there are directions in
$\Rd$ orthogonal to $\cS(x)$, invalidating \eqref{eq:assumption UE}. Meanwhile, overparameterization is common in modern machine
learning. For example, the ImageNet benchmark contains about
$10^6$ training images
\cite{deng2009imagenet}, whereas standard ViT models have more than
$10^7$ trainable parameters \cite{dosovitskiy2020image}.

\paragraph{Limitation 2: Nonsmooth stochastic objectives.}
Classical strict saddle theory naturally relies on Hessian information, but this route is not directly available for nonsmooth objectives. In the deterministic setting, Davis and Drusvyatskiy \cite{DavisDrusvyatskiy2022} addressed this difficulty using active manifold theory and showed that proximal-type methods avoid \emph{active strict saddle points}\footnote{Here ``active'' refers to the active manifold identified by the method, and the saddle structure is defined relative to that manifold.}. Combined with iterate convergence guarantees \cite{AttBolSva13}, this yields convergence to a local minimizer under the strict saddle property.

The stochastic nonsmooth case has an additional difficulty: the iterate convergence theory. This is because when an algorithm itself is not convergent, then ruling out convergence to strict saddles becomes less informative. For algorithms such as the standard proximal stochastic gradient method, current convergence theory often gives only subsequential stationarity \cite{LiMil22}, i.e., every accumulation point is stationary. However, convergence of the whole iterate sequence remains unknown. Assumptions such as isolated stationary points can ``upgrade'' subsequential stationarity to iterate convergence \cite{BianchiHachemSchechtman2023}, but they require global information that is often hard to verify. A different approach is to perturb the objective by a generic linear tilt\footnote{A tilted function is a linear perturbation of the form $f_v(x):=f(x)-\langle v,x\rangle$ for some $v\in\Rd$.}. For a generic tilt, the stationary structure becomes simpler: the tilted objective has finitely many stationary points and satisfies a strict saddle property. Davis et al.\ \cite{DavisDrusvyatskiyJiang2025} used this idea to prove convergence of stochastic nonsmooth algorithms to local minimizers of tilted objectives.
This is a powerful way to bypass the convergence difficulty, but the conclusion is for a perturbed objective that is not merely a theoretical intermediate object: it changes the optimization problem and hence the algorithm actually being run. However, this leaves a complementary question: Can a stochastic nonsmooth method converge to a local minimizer of the original objective, without relying on unit excitation, a generic tilt, or related nondegeneracy assumptions?

\paragraph{Limitation 3: Without-replacement sampling.}
A third gap comes from the sampling scheme. Existing analyses mostly rely on fresh samples that are independent of the past, which allows them to utilize the martingale-difference structure of the noises. However, this sampling model does not always reflect common practice. In neural network training, one often reshuffles the dataset at the beginning of each epoch, then processes each sample once. This without-replacement scheme is common in practice and is supported by data-loading pipelines in PyTorch \cite{paszke2019pytorch} and TensorFlow \cite{abadi2016tensorflow}. Furthermore, it has motivated a substantial body of work on convergence properties for the random reshuffling method \cite{gurbuzbalaban2015random,nagaraj2019sgd,safran2019good,rajput2020closing,mishchenko2020random,nguyen2020unified}. However, the stochastic gradients of random reshuffling are dependent and are generally biased estimators of the full gradient. Therefore, classical frameworks cannot explain the saddle avoidance of random reshuffling, even if \eqref{eq:assumption UE} is satisfied.
To our knowledge, the closest result is the local saddle escape theorem by Beneventano \cite{Beneventano2024TrajectoriesSGDWithoutReplacement}. It was shown that the bias accumulated by without-replacement sampling can push the trajectory away from a given saddle, provided a nonzero projection condition holds. For $f(x)=\frac1n \sum_{i=1}^n f_i(x)$, the condition takes the form:
\begin{equation}\label{eq:rr-nonzero-projection}
    \frac1n{\sum}_{i=1}^n
    \langle v,\nabla^2 f_i(x^*)\nabla f_i(x^*)\rangle \ne 0.
\end{equation}
Here, $x^*$ is the saddle point and $v$ is an eigenvector corresponding to a negative eigenvalue of
$\nabla^2 f(x^*)$. On the one hand, this condition is difficult to verify
because it depends on sample-wise second-order information at the saddle. On the other hand, it can fail for structural reasons. For example, under an
interpolation condition, i.e., $\nabla f_i(x^*)=0$ for all $i\in \{1,\dots,n\}$, \eqref{eq:rr-nonzero-projection} fails. Interpolation is typically satisfied for over-parameterized models
\cite{schmidt2013fast,vaswani2019fast,gower2021sgd}.

\subsection{Our solutions and contributions}
We study the following stochastic recursion driven by sampled update fields:
\begin{equation}\label{eq:intro-stochastic-recursion}
    z^{k+1}=z^k-\alpha_kG(z^k;\xi^k),
\end{equation}
where $z^k$ is the state variable, $\xi^k$ denotes the oracle randomness, and
$G(\cdot;\xi^k)$ is the sampled update field. The associated mean field is
denoted by $F:\Rd\to\Rd$. In this case, the ``saddle points'' correspond to the unstable zeros of $F$, the points $z^*$ at which $F(z^*)=0$ while
${\rm D}F(z^*)$ has an eigenvalue with negative real part. With suitable
choices of the state variable and oracle, this formulation covers a large variety of stochastic
algorithms, such as stochastic mirror descent, proximal-type
stochastic gradient methods, and random reshuffling.

The difficulty in removing the long-standing excitation assumption is that
\eqref{eq:assumption UE} provides the key escape mechanism in existing stochastic
avoidance proofs. Near a strict saddle, it guarantees a nontrivial noise
projection onto any unstable direction, so the iterates are repeatedly perturbed and move away from the saddle. Without this directional excitation, avoidance needs to be
established via a completely different mechanism. One possible way is to show that the set of initial states whose trajectories
remain trapped near an unstable zero lies in a Lebesgue null set.

This is the measure-zero basin principle used in deterministic saddle
avoidance. In the work of Lee et al.\ \cite{LeePPSJR2019}, the authors study a map
$T$ that fixes the saddle, $T(z^*)=z^*$, and the unstable directions are from ${\rm D}T(z^*)$. The center-stable manifold theorem then implies that the basin of $z^*$ has Lebesgue measure zero. The results in \cite{LeePPSJR2019} are stated for constant step-sizes and were later extended to diminishing step-sizes \cite{panageas2019first,mucsat2026non}. These deterministic variants preserve the same fixed-point geometry: the maps $z\mapsto z-\alpha_kF(z)$ all share
$z^*$, with $F(z^*)=0$, as a common fixed point, and their local behavior is
governed by ${\rm D}F(z^*)$. 

The stochastic recursion \eqref{eq:intro-stochastic-recursion} does not have
this fixed-point geometry. Even if $F(z^*)=0$, the sampled field typically
\emph{does not} satisfy $G(z^*;\xi^k)=0$ for all $k$. Consequently, $z^*$ is generally not a fixed point of the sampled update maps, and their derivatives vary with the oracle realization. Thus, the path of sampled maps has \emph{neither a common fixed point nor a common
linearization} from which to study the instability of the dynamics.

Our first step is to fix the oracle realization. This turns the stochastic recursion into a deterministic \emph{nonautonomous system}, but it does not restore the fixed-point geometry: the sampled maps still do not share the fixed point $z^*$, and their linear parts remain sample-dependent. To recover a usable local structure, we apply a path-dependent change of variables that removes the accumulated random linear perturbation from the leading dynamics. In the transformed coordinates, the leading behavior is governed by ${\rm D}F(z^*)$. Thus, the stable, unstable, and center subspace splitting of the mean field becomes available. We then use a Lyapunov-Perron-based construction to characterize the initial states whose \emph{transformed trajectories} can remain trapped near the unstable zero. These states lie on a Lipschitz graph and hence form a Lebesgue null set. Finally, by Fubini's theorem, we conclude that 
\[ \Prob\;\!({\lim}_{k\to\infty}\,z^k\in\cZ^*)=0, \] 
where $\cZ^*$ is the set of unstable zeros of $F$. We apply this analysis framework to various algorithms to address the three mentioned limitations.

\begin{table}[t]
\centering
{\footnotesize
\setlength{\tabcolsep}{3.0pt}
\NiceMatrixOptions{cell-space-limits=3pt}
\begin{NiceTabular}{|p{3.4cm}|p{2cm}|p{2.5cm}|p{2.1cm}|p{1.55cm}|p{2.5cm}|}%
\toprule
\Block{1-1}{\textbf{framework / method}} &
\Block{1-1}{\textbf{no UE-type requirements}} &
\Block{1-1}{\textbf{nonsmooth-type objectives}} &
\Block{1-1}{\textbf{allows w/o replacement}} &
\Block{1-1}{\textbf{saddle avoidance}} &
\Block{1-1}{\textbf{references}} \\ \Hline
\Block{1-1}{Deterministic smooth map framework} &
\Block{1-1}{\cmarkt} &
\Block{1-1}{\xmarkt} &
\Block{1-1}{--} &
\Block{1-1}{\cmarkt} &
\Block{1-1}{\cite{LeePPSJR2019,panageas2019first,mucsat2026non}} \\ \Hline
\Block{1-1}{Deterministic proximal framework} &
\Block{1-1}{\cmarkt} &
\Block{1-1}{\cmarkt} &
\Block{1-1}{--} &
\Block{1-1}{\cmarkt} &
\Block{1-1}{\cite{DavisDrusvyatskiy2022}} \\ \Hline
\Block{1-1}{Stochastic approximation methods} &
\Block{1-1}{\xmarkt} &
\Block{1-1}{\xmarkt} &
\Block{1-1}{\xmarkt} &
\Block{1-1}{\cmarkt} &
\Block{1-1}{\cite{pemantle1990nonconvergence,benaim1996dynamical,benaim1999dynamics,mertikopoulos2020almost,liuYuan2023almostSureSaddleAvoidance}} \\ \Hline
\Block{1-1}{Stochastic subgradient method} &
\Block{1-1}{\xmarkt} &
\Block{1-1}{\cmarkt} &
\Block{1-1}{\xmarkt} &
\Block{1-1}{\cmarkt} &
\Block{1-1}{\cite{BianchiHachemSchechtman2023}} \\ \Hline
\Block{1-1}{Stochastic proximal-type framework} &
\Block{1-1}{\xmarkt} &
\Block{1-1}{\cmarkt} &
\Block{1-1}{\xmarkt} &
\Block{1-1}{\cmarkt} &
\Block{1-1}{\cite{DavisDrusvyatskiyJiang2025}} \\ \Hline
\Block{1-1}{Random reshuffling method} &
\Block{1-1}{\xmarkt} &
\Block{1-1}{\xmarkt} &
\Block{1-1}{\cmarkt} &
\Block{1-1}{\;\xmarkt\textsuperscript{\textdagger}} &
\Block{1-1}{\cite{Beneventano2024TrajectoriesSGDWithoutReplacement}} \\ \Hline
\Block{1-1}{Stochastic recursion framework} &
\Block{1-1}{\cmarkt} &
\Block{1-1}{\cmarkt} &
\Block{1-1}{\cmarkt} &
\Block{1-1}{\cmarkt} &
\Block{1-1}{This paper} \\
\bottomrule
\end{NiceTabular}
}
\vspace{1ex}
\caption{Comparison of saddle avoidance results. The ``no UE-type requirements'' column indicates whether the analysis avoids unit-excitation or analogous excitation assumptions. The ``nonsmooth-type objectives'' column indicates whether the analysis covers nonsmooth objectives. The ``allows w/o replacement'' column indicates whether the analysis applies to methods under a without-replacement sampling scheme, such as random reshuffling. \textdagger\,\cite[Section 6]{Beneventano2024TrajectoriesSGDWithoutReplacement} shows that, when initialized at a strict saddle point, the RR trajectory moves away from it. This is different from the standard saddle avoidance notion: for almost every initialization, the iterates do not converge to a strict saddle.}
\label{tab:literature-comparison}
\end{table}

\begin{itemize}
    \item First, the proposed framework covers stochastic mirror descent, with standard SGD as the Euclidean special case. Under the relative smoothness condition that includes the standard Lipschitz smoothness as a special case, stochastic
    mirror descent takes the form \eqref{eq:intro-stochastic-recursion} in the
    dual variable $z=\nabla h(x)$. We prove that it avoids strict saddles
    almost surely; see \cref{cor:smd-avoids-strict-saddles}. The stochastic
    oracle assumptions used in this result are verifiable in standard
    models. For example, they hold for finite-sum objectives with random index sampling, and
    for general expected loss with sub-Gaussian noise. In particular, the
    \eqref{eq:assumption UE} condition is not required to be satisfied. As a by-product, when the noise is absent, the avoidance result extends deterministic mirror descent saddle avoidance from
    Lipschitz smoothness to relative smoothness.

    \item Second, for composite objectives, we study a normal map-based proximal stochastic gradient method \cite{qiu2025normal}. The goal is to obtain convergence to local minimizers of the original nonsmooth objective function, rather than of a generically tilted one. The normal map-based perspective allows the method to be interpreted as a stochastic recursion of the form \eqref{eq:intro-stochastic-recursion}, and active manifold calculus identifies active strict saddles of the original objective with unstable zeros of the normal map field. We prove almost sure non-convergence to active strict saddles; see \cref{cor:nsgd-avoid-active-saddles}. Combining this avoidance result with iterate convergence guarantees \cite{qiu2025normal} and the active strict saddle property yields convergence to local minimizers of the original objective. The avoidance argument uses no generic tilt, no unit excitation, and no isolated stationary set assumption.

    \item Third, the framework also covers without-replacement sampling. In particular, random reshuffling (RR) satisfies the required stochastic error conditions even though the stochastic gradients are dependent and are generally not martingale differences. Our framework establishes the almost sure non-convergence of RR to strict saddles; see \cref{cor:rr-avoids-strict-saddles}. To the best of our knowledge, this is the first asymptotic almost sure strict saddle avoidance result for RR\@. Since the oracle conditions are formulated for general update fields, the same approach may also be useful for normal map-based nonsmooth extensions of RR \cite{qiu2023new}.
\end{itemize}

\subsection{Notation and Organization}
We write $\N:=\{0,1,2,\ldots\}$ and use $\|\cdot\|$ for the Euclidean norm and the induced operator norm. For $x\in\Rd$ and $r>0$, let $
B_r(x):=\{y\in\Rd:\|y-x\|<r\}$. 
The symbol $\mu$ denotes the Lebesgue measure on $\Rd$. When we say that an
initial point is drawn from a distribution with a density, we mean that its
distribution is absolutely continuous with respect to $\mu$. For a
differentiable map $F$, ${\rm D}F(x)$ denotes its derivative at $x$. For a matrix $A$,
$\operatorname{spec}(A)$ denotes its spectrum and $\Re(\lambda)$ is
the real part of $\lambda$.

The remainder of the paper is organized as follows. In \cref{sec:framework}, we
introduce the stochastic recursion, standing assumptions, and
the general saddle avoidance theorem. This theorem is built on a local
trapping theorem, which serves as a replacement for the center-stable manifold theorem used in deterministic analyses. In
\cref{sec:local-trapped-proof}, we prove \cref{thm:local-stable-null}. In \cref{sec:application-sa-methods}, we verify
the framework for methods based on stochastic approximation, including stochastic
mirror descent and a normal map-based stochastic proximal method. In
\cref{sec:application-rr}, we show that our framework can also be applied to RR, whose
sampling scheme is different from algorithms presented in 
\cref{sec:application-sa-methods}.

\section{Framework and theory}
\label{sec:framework}
This section develops a framework for stochastic recursions and states the general saddle avoidance theorem. We first show how several common
algorithms fit this form in \cref{subsec:examples}. We then specify the
probability spaces for the initialization and the oracle randomness, together
with the stochastic oracle assumptions, in
\cref{subsec:probability-oracle-assumptions}. The main avoidance results are
stated in \cref{subsec:main-avoidance-results}.
\subsection{Stochastic recursion}
Let $F:\Rd\to\Rd$ be the vector field whose zeros are the points of interest. Given an oracle $G:\Rd\times\Xi\to\Rd$ for $F$, we study the recursion
\begin{equation}\label{eq:stochastic-fixed-point-method}
z^{k+1}=z^k-\alpha_kG(z^k;\xi^k),\qquad \forall\,k\ge0.
\end{equation}
This stochastic recursion encompasses several standard stochastic methods, as
illustrated in \cref{subsec:examples}. Our goal is to show that
\eqref{eq:stochastic-fixed-point-method} avoids the unstable zeros of $F$. The formal definition is given below.
\begin{definition}[Unstable zeros]
A point $z^*\in\Rd$ is an \emph{unstable zero} of $F$ if
\begin{equation}
\label{eq:unstable-zero}
F(z^*)=0
\qquad \text{and} \qquad
\min\{\Re(\lambda):\lambda\in\operatorname{spec}({\rm D}F(z^*))\}<0.
\end{equation}
Denote the set of unstable zeros by 
\[\cZ^*:=\{z^*\in\Rd:z^*\text{ is an unstable zero of }F\}.\]
\end{definition}

\begin{assumption}
\label{assumption:F}
We impose the following assumptions on the mean field
and the sample maps.
\begin{itemize}
    \item\emph{(regularity).} The vector field $F:\Rd\to\Rd$ is $\cC^1$ in a neighborhood of $\cZ^*$.
    \item\emph{(local equicontinuity).} For every $z^*\in\cZ^*$, there exists an open neighborhood $\cU_{z^*}$ of $z^*$ such that $G(\cdot;\xi)$ is differentiable on $\cU_{z^*}$ for every $\xi\in\Xi$, and the family
    $\{{\rm D}G(\cdot;\xi)\}_{\xi\in\Xi}$ is equicontinuous on compact subsets of $\cU_{z^*}$: for every compact set $\mathcal Q\subset\cU_{z^*}$ and every $\varepsilon>0$, there exists $\delta>0$ such that
    \[
        \|{\rm D}G(z;\xi)-{\rm D}G(w;\xi)\|\leq\varepsilon
    \]
    for all $\xi\in\Xi$ and all $z,w\in\mathcal Q$ with
    $\|z-w\|\leq\delta$.
    \item\emph{(lipeomorphism).} There exists $\tilde{\alpha}>0$ such that, for every
    $0<\alpha\le\tilde{\alpha}$ and every $\xi\in\Xi$, the map
    \[
    T_{\alpha,\xi}(z):=z-\alpha G(z;\xi)
    \]
    has a well-defined inverse on $\Rd$, and both $T_{\alpha,\xi}$ and
    $T_{\alpha,\xi}^{-1}$ are locally Lipschitz.
\end{itemize}
\end{assumption}

\begin{remark}[Relation to assumptions in deterministic case]
In the deterministic case $G(\cdot;\xi)\equiv F(\cdot)$, the local equicontinuity condition follows from the continuity of ${\rm D}F$ near $\cZ^*$, and the lipeomorphism condition reduces to the usual requirement that the update map
$z\mapsto z-\alpha_kF(z)$ is invertible with a locally Lipschitz inverse. Hence, these conditions play the same role as the diffeomorphism assumptions used for deterministic first-order methods. Therefore,
\cref{assumption:F} does not impose stronger requirements than the standard deterministic setting \cite{LeePPSJR2019,DavisDrusvyatskiy2022,panageas2019first,mucsat2026non}.
\end{remark}
 
\begin{remark}[Detailed discussions of \cref{assumption:F}]\;
    \begin{itemize}
        \item The regularity condition is local. For smooth stochastic gradient methods, it holds whenever the objective is $\cC^2$ near stationary points. For nonsmooth problems, the local smoothness is typically satisfied after applying active manifold identification theory. Indeed, the objective is $\cC^2$ along the identified manifold \cite{lewis2002active,hare2004identifying,DavisDrusvyatskiy2022}.
\item The local equicontinuity condition does not require the sample maps to be differentiable globally. It only requires that, near each unstable zero, sample maps be differentiable on an open neighborhood and that the family of their Jacobians be equicontinuous on compact subsets of that neighborhood. Consequently, the framework can accommodate nonsmooth optimization problems whenever a suitable local reformulation yields sample maps with these properties. This condition is automatic for finite-sum models $
f(x)=\frac1n\sum_{i=1}^n f_i(x)$ 
when the component Hessians $\nabla^2 f_i$ are locally Lipschitz. Equicontinuity also holds for empirical-risk losses of the form
$f(x;\xi)=\ell(a^\top x;b)$ with bounded feature vectors. Indeed, as $\nabla^2 f(x;\xi)=\ell''(a^\top x;b)aa^\top$, if $\|a\|$ is uniformly bounded and $t\mapsto \ell''(t;b)$ is locally
Lipschitz uniformly in $b$, then the sample Hessians are uniformly locally
Lipschitz on compact parameter sets. This covers least squares and binary
logistic regression. Similar arguments also extend to
multinomial logistic regression and smooth generalized linear models
\cite{bottou2018optimization,Bach2010,hastibfri09}.
\item The lipeomorphism condition is a mild invertibility requirement on the update maps. One common way to verify it is through a uniform Lipschitz bound on
the oracle. That is, if
\[
    \|G(z;\xi)-G(w;\xi)\|\le \sL\|z-w\|\qquad \text{for all $z,w$ and all $\xi$,}
\]
then $I-\alpha G(\cdot;\xi)$ is bi-Lipschitz for all $\alpha\in(0,1/\sL)$. Thus, in many first-order methods, this condition
reduces to the usual ``small-step-size requirement''. In this setting, SGD works as an immediate example. In the forthcoming sections, we also verify the lipeomorphism property for proximal-type stochastic gradient methods, stochastic mirror descent, and random reshuffling.
    \end{itemize}
\end{remark}

\subsection{Examples} \label{subsec:examples}
Since our main focus is saddle avoidance for stochastic methods, we present
several instances that fit the stochastic recursion \eqref{eq:stochastic-fixed-point-method}. 
\begin{example}[Stochastic gradient methods]
Consider $\min_{x\in\Rd} f(x)$, where $f(x):=\Exp_\xi[f(x;\xi)]$. Taking $
F(\cdot)=\nabla f(\cdot)$ and $G(\cdot;\xi)=\nabla f(\cdot;\xi)$, the formulation
\eqref{eq:stochastic-fixed-point-method} becomes the standard stochastic gradient recursion
\[
z^{k+1}=z^k-\alpha_k\nabla f(z^k;\xi^k).
\]
Note that minibatch variants are included by letting $\xi^k$ encode the sampled batch.
\end{example}
 
\begin{example}[Stochastic mirror descent]
Let $f$ be the objective function and let $h$ be a Legendre distance-generating function. Stochastic mirror descent performs the dual update 
\[
z^{k+1}=z^k - \alpha_k \nabla f(x^k;\xi^k)=z^k-\alpha_k \nabla f(\nabla h^*(z^k);\xi^k),
\]
where we used the correspondences $z=\nabla h(x)$ and
$x=\nabla h^*(z)$; see \cite[Theorem~26.5]{Rockafellar1970ConvexAnalysis}.
This is of the form \eqref{eq:stochastic-fixed-point-method} with $
G(\cdot;\xi)=\nabla f(\nabla h^*(\cdot);\xi)$ and $F(\cdot)=\nabla f(\nabla h^*(\cdot))$.
\end{example}

\begin{example}[Normal map stochastic gradient]
We consider the composite problem 
\[{\min}_{x\in\Rd}~f(x)+\vp(x),\]
where $f:\Rd\to\R$ is differentiable and $\vp:\Rd\to(-\infty,+\infty]$ is a possibly nonsmooth (convex, proper, and lower semicontinuous) function with a computable proximal map. Fix $\lambda>0$, and let $\nabla f(x;\xi)$ denote a stochastic approximation of $\nabla f$ at $x$. Then, the
normal map stochastic gradient step \cite{qiu2025normal} is given by
\[
z^{k+1}
=
z^k-\alpha_k\left[\nabla f(P_\lambda(z^k);\xi^k)
+\lambda^{-1}\bigl(z^k-P_\lambda(z^k)\bigr)\right]\qquad \text{where}\qquad P_\lambda(z):=\prox{\lambda\vp}(z).
\]
This is an instance of the stochastic recursion \eqref{eq:stochastic-fixed-point-method} with
\[
G(z;\xi):=\nabla f(P_\lambda(z);\xi)+\lambda^{-1}(z-P_\lambda(z))
\qquad \text{and} \qquad 
F(z):=\nabla f(P_\lambda(z))+\lambda^{-1}(z-P_\lambda(z)).
\]
\end{example}
 
The preceding examples use one stochastic oracle call per iteration. Random
reshuffling is based on a different sampling mechanism: at epoch $k$, the seed is a
full permutation, and the component gradients are evaluated sequentially without
replacement.
 
\begin{example}[Random reshuffling]
Consider the finite-sum problem
\[
    {\min}_{x\in\Rd}~f(x):=\frac1n{\sum}_{i=1}^n f_i(x).
\]
The associated mean field in this case is simply $F=\nabla f$. 
Let $\eta_k>0$ be the step-size and let $\pi^k$ denote a
permutation of $\{1,\ldots,n\}$. Random
reshuffling (RR) initializes $y_0^k:=x^k$ and performs the following updates
\[
x^{k+1}
=
x^k-\eta_k{\sum}_{i=1}^n\nabla f_{\pi_i^k}(y_{i-1}^k)\qquad \text{where} \qquad y_i^k
=
y_{i-1}^k-\eta_k\nabla f_{\pi_i^k}(y_{i-1}^k),\qquad i=1,\ldots,n.
\]
Setting $\xi^k:=(\eta_k,\pi^k)$, we define the intermediate iterates within one epoch for a candidate point $x$ via:
\[
 y_0(x;\xi^k):=x\qquad \text{and} \qquad y_i(x;\xi^k)
    :=
    y_{i-1}(x;\xi^k)
    -\eta_k\nabla f_{\pi_i^k}(y_{i-1}(x;\xi^k)),
    \qquad i=1,\ldots,n.
\]
Furthermore, setting $
    G(x;\xi^k):=
    \frac1n{\sum}_{i=1}^n
    \nabla f_{\pi_i^k}(y_{i-1}(x;\xi^k))$ and $\alpha_k:=n\eta_k$, we can write RR's update as $
x^{k+1}=x^k-\alpha_kG(x^k;\xi^k)$ as desired.
\end{example}

\subsection{Probability spaces and oracle assumptions}
\label{subsec:probability-oracle-assumptions}
We separate the randomness in the initial point from the randomness in the
oracle calls. This allows us to later fix a seed realization and apply the local dynamical argument pathwise.

\begin{definition}[Probability model and measurability]
\label{def:probability-model}
Let $(\Omega_0,\cF_0,\Prob_0)$ support the initial point $z^0\sim\D$, where
$\D$ is any distribution that is absolutely continuous with respect to $\mu$. Let
$(\Omega_\xi,\cF_\xi,\Prob_\xi)$ support the seed sequence
$\{\xi^k\}_{k}$, where each $\xi^k$ takes values in a measurable space
$(\Xi,\cA)$. Its natural filtration is given by
\[
\cF_k^\xi:=\sigma(\xi^0,\dots,\xi^{k-1}),
\qquad \text{and} \qquad
\cF_0^\xi:=\{\emptyset,\Omega_\xi\}.
\]
We assume that $z^0$ and $\{\xi^k\}_{k\ge0}$ are independent and work on the
product space
\[
(\Omega,\cF,\Prob)
:=
(\Omega_0\times\Omega_\xi,\cF_0\otimes\cF_\xi,\Prob_0\otimes\Prob_\xi).
\]
We use the same symbols $z^0$ and $\xi^k$ on $\Omega$, and set $
    \cF_k:=\sigma(z^0)\vee\cF_k^\xi$. We assume that
$ (z,\xi)\mapsto G(z;\xi)$ is measurable from $\Rd\times\Xi$, equipped with
$\mathcal B(\Rd)\otimes\cA$, to $\Rd$.
Moreover, for every $z^*\in\cZ^*$, we assume that the restriction
$(z,\xi)\mapsto {\rm D}G(z;\xi)$ is measurable from
$\cU_{z^*}\times\Xi$, equipped with
$\mathcal B(\cU_{z^*})\otimes\cA$, to $\R^{d\times d}$.
Hence, $G(z;\xi^k)$ is a well-defined random vector for every fixed $z\in\Rd$,
and ${\rm D}G(z;\xi^k)$ is a well-defined random matrix whenever
$z\in\cU_{z^*}$ for some $z^*\in\cZ^*$.
\end{definition}

The next assumption imposes pathwise regularity on the stochastic oracle near an
unstable zero of the mean field $F$. 
 
\begin{assumption}[Stochastic oracle]
\label{assumption:stochastic-oracle} Let $\bar z\in\cZ^*$ be arbitrary. 
Denote 
\begin{equation}\label{eq:b_k and J_k}
b_k:=G(\bar z;\xi^k)-F(\bar z),\qquad  J_k:={\rm D}G(\bar z;\xi^k)-{\rm D}F(\bar z)\qquad \text{and} \qquad S_k:={\sum}_{i=k}^{\infty}\alpha_iJ_i.    
\end{equation}
Assume that there exists an event $\cE_{\rm reg}(\bar z)\in\cF_\xi$ with
$\Prob_\xi(\cE_{\rm reg}(\bar z))=1$ such that, on $\cE_{\rm reg}(\bar z)$, we have
\[
    S_k\to0,\qquad S_{k+1}J_k\to0,\qquad
    {\sum}_{k=0}^{\infty}\alpha_k b_k \quad \text{and}\quad 
    {\sum}_{k=0}^{\infty}\alpha_k S_{k+1}b_k \;\, \text{converge}.
\]
\end{assumption}
This assumption controls the realized value errors $b_k$ and linearization errors $J_k$ through the accumulated quantities in \eqref{eq:b_k and J_k}. In the deterministic case $G(\cdot;\xi)\equiv F(\cdot)$, one has $b_k=0$, $J_k=0$, and $S_k=0$, so the
condition holds automatically. In the stochastic case, the assumption only requires these errors to have controlled weighted tails. This is not related to the unit excitation condition \eqref{eq:assumption UE}. Indeed, it allows the noise to vanish at $\bar z$ or to lie in a lower-dimensional subspace. Although this pathwise formulation is less common in classical stochastic approximation, it can be
verified in several basic sampling models. We first consider the stochastic approximation with finite moments.

\begin{proposition}[Martingale-type noise]
\label{lem:iid-moment-bounds-imply-stochastic-oracle}
Let $\{(b_k,J_k)\}_{k}$ be defined as in \eqref{eq:b_k and J_k}. Assume $\{\xi^k\}_{k}$ are i.i.d.\ on $(\Omega_\xi,\cF_\xi,\Prob_\xi)$ and satisfy, for some $q>2$,
\[
\Exp[b_k]=0,\qquad
\Exp[\|b_k\|^2]<\infty,\qquad
\Exp[J_k]=0,\qquad
\Exp[\|J_k\|_F^q]<\infty.
\]
Moreover, assume that $\sum_{k=0}^{\infty} r_{k+1}^{q/2}<\infty$ with $r_k=\sum_{i=k}^{\infty}\alpha_i^2$. Then \cref{assumption:stochastic-oracle} holds.
\end{proposition}
The proof is deferred to \cref{app:martingale-type-noise}. The requirements on $b_k$ are the usual 
martingale-type noise conditions used in the stochastic approximation literature. The condition on $J_k$ is less standard, because classical stochastic approximation results usually do not track the sample-wise linearization error. Nevertheless, it is a pointwise finite-moment condition and is easy to verify in many standard models. Typical examples include finite-sum models and stochastic gradient oracles with finite sample Hessians. Detailed verifications and discussions are presented in \cref{sec:application-sa-methods}. 

A second class of examples comes from without-replacement sampling in
finite-sum problems. The standard algorithmic example is random reshuffling, which draws a permutation at the beginning of each epoch and then visits every component once. This sampling rule is common in finite-sum
optimization and has been studied extensively; see, e.g.,
\cite{gurbuzbalaban2015random,nagaraj2019sgd,rajput2020closing,
mishchenko2020random,nguyen2020unified}. The next proposition verifies \cref{assumption:stochastic-oracle} for the stochastic oracle induced by sampling without replacement. The proof is provided
in \cref{app:sampling-without-replacement}. 

\begin{proposition}[Sampling without replacement]
\label{prop:without-replacement-implies-stochastic-oracle}
Let $G_1,\ldots,G_n:\Rd\to\Rd$ be $\cC^1$ maps with locally Lipschitz derivatives, and set $
F:=\frac1n\sum_{i=1}^n G_i$. 
For a step-size $\eta>0$ and a permutation $\pi$ of $\{1,\ldots,n\}$ define the epoch iterates
\[
Z_0(z;\eta,\pi)=z,\qquad
Z_s(z;\eta,\pi)=Z_{s-1}(z;\eta,\pi)-\eta G_{\pi_s}(Z_{s-1}(z;\eta,\pi)),
\quad s=1,\ldots,n,
\]
and the stochastic oracle $G(z;\eta,\pi):=\frac1n{\sum}_{s=1}^{n}G_{\pi_s}(Z_{s-1}(z;\eta,\pi))$. Then we have
\begin{itemize}
    \item For every $\bar\eta>0$, the family $
    \{ {\rm D}G(\cdot;\eta,\pi):0<\eta\le \bar\eta,\ 
    \pi \text{ is a permutation of }\{1,\ldots,n\}\}$ 
is equicontinuous on compact sets.
    \item Let $\{\eta_k\} \subset \R_{++}$ be given with $\sum_{k=0}^\infty\eta_k^2<\infty$, and set
    $\alpha_k:=n\eta_k$. For any sequence of permutations $\{\pi^k\}_k$, \cref{assumption:stochastic-oracle} holds with $\xi^k:=(\eta_k,\pi^k)$ and $G(z;\xi^k):=G(z;\eta_k,\pi^k)$. 
\end{itemize}
\end{proposition}

\subsection{Main avoidance results}
\label{subsec:main-avoidance-results}
Throughout the remainder of the paper, we assume that the step-sizes satisfy
\[
0<\alpha_k\le\tilde\alpha,\qquad \alpha_k\to 0 
\qquad \text{and} \qquad
{\sum}_{k=0}^{\infty}\alpha_k=\infty.
\]
Here, $\tilde\alpha$ is the step-size threshold in the lipeomorphism condition of \cref{assumption:F}. We first state a local result that characterizes trajectories of \eqref{eq:stochastic-fixed-point-method}. Its proof is deferred to \cref{sec:local-trapped-proof}. 

\begin{theorem}[Locally trapped iterates stay in null sets]\label{thm:local-stable-null}
Suppose \cref{assumption:F,assumption:stochastic-oracle} hold. Fix $z^*\in\cZ^*$.
There exists $\delta_0>0$ such that, for every $\omega_\xi\in\cE_{\rm reg}(z^*)$, there are $m_0\in\N$ and Lebesgue null sets $\{\mathcal N_m\}_{m}$ in $\Rd$ such that, for every $m\ge m_0$, any orbit with
\[
\|z^k-z^*\|\le\delta_0\qquad \forall\,k\ge m 
\]
has to satisfy $z^m\in\mathcal N_m$.
\end{theorem}

\cref{thm:local-stable-null} shows that, if a trajectory remains in a
neighborhood of an unstable zero from some time $m\in\N$ onward, then its state
at $m$ must lie in a Lebesgue null set. This local null set result forms the backbone of the almost sure avoidance theorem.

\begin{theorem}[Almost sure avoidance of unstable zeros]
\label{thm:avoid-strict-saddles}
Under the
probabilistic setting of \cref{def:probability-model}, suppose \cref{assumption:F,assumption:stochastic-oracle} are satisfied. Then the stochastic recursion \eqref{eq:stochastic-fixed-point-method} almost surely avoids unstable zeros, i.e.,
\[ \Prob\;\!({\lim}_{k\to\infty}\,z^k\in\cZ^*)=0. \]
\end{theorem}

\begin{proof}
If $\cZ^*=\emptyset$, the claim is immediate. Otherwise, for each $z\in\cZ^*$, let $\delta_z>0$ be the radius from \cref{thm:local-stable-null}. The balls $B_{\delta_z/2}(z)$, $z\in\cZ^*$, cover $\cZ^*$. Since $\Rd$ is second countable, choose a countable set of points $\{z_i\}_i$ with $z_i\in\cZ^*$, $i\ge1$, such that
\[
\cZ^*\subset {\bigcup}_{i=1}^{\infty}B_{\delta_i/2}(z_i),
\qquad \text{and} \qquad
\delta_i:=\delta_{z_i}.
\]
Set $\cE_*:={\bigcap}_{i=1}^{\infty}\cE_{\rm reg}(z_i)$. By countability of $\{z_i\}_i$ and \cref{assumption:stochastic-oracle}, $\Prob_\xi(\cE_*)=1$. Next, fix $\omega_\xi\in\cE_*$, and set $T_{\omega_\xi,k}(z):=z-\alpha_kG(z;\xi^k(\omega_\xi))$. For $m\ge1$, define
\[
\Phi_{\omega_\xi,m}:=
T_{\omega_\xi,m-1}\circ\cdots\circ T_{\omega_\xi,0},
\qquad \text{and} \qquad
\Phi_{\omega_\xi,0}:=I.
\]
By the lipeomorphism condition (\cref{assumption:F}) and $\sup_k\alpha_k\le\tilde{\alpha}$, each $T_{\omega_\xi,k}$ is invertible and $T_{\omega_\xi,k}^{-1}$ is locally Lipschitz. Hence, for every finite $m$, the map $\Phi_{\omega_\xi,m}$ has the locally Lipschitz inverse
\[
\Phi_{\omega_\xi,m}^{-1}
=
T_{\omega_\xi,0}^{-1}\circ\cdots\circ T_{\omega_\xi,m-1}^{-1}.
\]
For $z\in\Rd$, write $Z^k(z,\omega_\xi):=\Phi_{\omega_\xi,k}(z)$ for the trajectory generated by 
\[ Z^{k+1}=T_{\omega_\xi,k}(Z^k)\qquad \text{with} \qquad Z^0=z.\]
For the fixed seed $\omega_\xi$, define the set of initial points whose
trajectories converge to an unstable zero:
\[
\mathcal A(\omega_\xi):= \{z\in\Rd:{\lim}_{k\to\infty}~Z^k(z,\omega_\xi)\in\cZ^*\}.
\]
We show that $\mu(\mathcal A(\omega_\xi))=0$.

For each $i$, \cref{thm:local-stable-null} applied at $z_i$ and
$\omega_\xi$ gives an index $m_i$ and null sets $\mathcal N_{i,m}$,
$m\ge m_i$. If $z\in\mathcal A(\omega_\xi)$, then
$\bar z:=\lim_k Z^k(z,\omega_\xi)$ lies in $B_{\delta_i/2}(z_i)$ for
some $i$. Hence $Z^k(z,\omega_\xi)\in B_{\delta_i}(z_i)$ for all large
$k$. Choose $m\ge m_i$ so that this tail condition holds from time $m$
onward. Then \cref{thm:local-stable-null} gives
$Z^m(z,\omega_\xi)\in\mathcal N_{i,m}$. Since
$Z^m(z,\omega_\xi)=\Phi_{\omega_\xi,m}(z)$,
\[
\mathcal A(\omega_\xi)
\subset
{\bigcup}_{i=1}^{\infty}{\bigcup}_{m=m_i}^{\infty}
\Phi_{\omega_\xi,m}^{-1}(\mathcal N_{i,m}).
\]
For each $i,m$, the set $\mathcal N_{i,m}$ is null. Since
$\Phi_{\omega_\xi,m}^{-1}$ is locally Lipschitz, the set
$\Phi_{\omega_\xi,m}^{-1}(\mathcal N_{i,m})$ is null; see, e.g., 
\cite[Proposition~6.5]{Lee2012IntroductionSmoothManifolds}.
The union is countable, so $\mu(\mathcal A(\omega_\xi))=0$. Since $\D\ll\mu$, we obtain
\[
\D(\mathcal A(\omega_\xi))=0
\qquad\text{for every }\omega_\xi\in\cE_*.
\]
Let $\cE := \left\{
(\omega_0,\omega_\xi)\in\Omega_0\times\Omega_\xi:
\lim_{k\to\infty}
Z^k\bigl(z^0(\omega_0),\omega_\xi\bigr)
\in\cZ^*
\right\}$. By the joint measurability of the trajectory maps, $\cE$ is measurable. Since $z^0$ is independent of the seed sequence and has law
$\D$, Fubini's theorem yields
\[
\Prob(\cE)=
\int_{\Omega_\xi}
\Prob_0\bigl(z^0\in\mathcal A(\omega_\xi)\bigr)
\,{\rm d}\Prob_\xi(\omega_\xi) = \int_{\Omega_\xi}
\D\bigl(\mathcal A(\omega_\xi)\bigr)
\,{\rm d}\Prob_\xi(\omega_\xi)
=0,
\]
because the integrand is zero on $\cE_*$ and $\Prob_\xi(\Omega_\xi\setminus\cE_*)=0$.
\end{proof}

\section{\texorpdfstring{Proof of \cref{thm:local-stable-null}}{Proof of the local stable null theorem}}\label{sec:local-trapped-proof}

We first highlight the difficulty. The proof is pathwise: we fix an unstable
zero $z^*$ and a seed path $\{\xi^k\}_k$. The goal is to show that any
trajectory that stays near $z^*$ must eventually pass through a Lebesgue null set.
Once the seed path is fixed, the stochastic recursion is deterministic but
driven by a time-dependent sample field:
\[
z^{k+1}=z^k-\alpha_kG(z^k;\xi^k).
\]
Time dependence alone would not be the main issue if the sample maps shared a
\emph{common fixed point} and a \emph{common linearization}. The main difficulty is
that, in the stochastic setting, these two common structures are generally
lost.

To see this, consider the deterministic iteration
$z^{k+1}=z^k-\alpha_k F(z^k)$; see, e.g.,
\cite{LeePPSJR2019,panageas2019first,mucsat2026non}. If $F(z^*)=0$, then
$z^*$ is fixed by every update map, and for
$\bar z^k:=z^k-z^*$, the local form is
\[
\bar z^{k+1}
=[I-\alpha_k{\rm D}F(z^*)]\bar z^k
+o(\|\bar z^k\|), 
\]
where every leading linear factor is generated by the same matrix
${\rm D}F(z^*)$. Thus, the local dynamics share a common
center-stable and unstable splitting $\Rd=E_{\cs}\oplus E_u$.

By contrast, for stochastic recursions, the corresponding centered expansion is
\[
\bar z^{k+1}
=
[I-\alpha_k{\rm D}G(z^*;\xi^k)]\bar z^k
-\alpha_kG(z^*;\xi^k)
+o(\|\bar z^k\|).
\]
Here, $G(z^*;\xi^k)$ generally does not vanish, and
${\rm D}G(z^*;\xi^k)$ varies in each iteration. In other words, the sample
maps \emph{do not} necessarily share $z^*$ as a fixed point, and their leading linear parts are also \emph{sample-dependent}. These are the main obstacles in proving saddle avoidance for stochastic methods.

We resolve these obstacles in four steps to obtain the null set conclusion stated in \Cref{thm:local-stable-null}.

\begin{enumerate}[label=\emph{Step \arabic*.}, leftmargin=*]
\item \emph{(Pathwise analysis and change of variables).}
In \cref{subsec:shifted-dynamics}, we translate $z^*$ to the origin and write
the fixed seed recursion in terms of the sample fields. We then construct
path-dependent invertible matrices $U_k$ and set $y^k=U_k z^k$. The purpose of
this transformation is to remove the sample-dependent Jacobian perturbation
from the leading linear term, so the transformed sequence has the form
\[
y^{k+1}
= [I-\alpha_k{\rm D}F(z^*)]y^k
+ \text{transformed forcing terms} + \text{residual terms}.
\]

\item \emph{(Linear transition estimates).}
The leading dynamics in the transformed $y$-coordinates are generated by the
common matrix ${\rm D}F(z^*)$. The transformed forcing and residual terms
remain, but they no longer determine the splitting. Therefore, we can use a fixed
splitting $\Rd=E_{\cs}\oplus E_u$ for the linear transitions
$[I-\alpha_k{\rm D}F(z^*)]$. We derive the transition estimates for forward growth on $E_{\cs}$ and backward decay on $E_u$ in \cref{subsec:spectral-transition-bounds}. 

\item \emph{(Lyapunov--Perron graph).} To further characterize the local
behavior of $\{y^k\}_k$, we construct the Lyapunov--Perron map in
\cref{subsec:truncated-lp}. We then show that the map is contractive, and that its fixed points form a Lipschitz graph over $E_{\cs}$. This graph contains all initial $y$-points whose asymptotic trajectories remain in a certain bounded region.

\item \emph{(Null-set conclusion).}
Finally, we transfer these results back to the original $z$-trajectory
in \cref{subsec:local-null-set-argument}. If a $z$-trajectory stays near $z^*$ from
some time onward, then the corresponding
$y$-trajectory must start on the Lyapunov--Perron graph, which is Lebesgue null. Since the change from $z$ to $y$ is invertible, the corresponding $z$-points also lie in a null set.
\end{enumerate}

This argument differs from stochastic non-convergence or saddle avoidance results in
which the key mechanism is noise excitation or nondegeneracy in unstable
directions; see, e.g.,
\cite{pemantle1990nonconvergence,BianchiHachemSchechtman2023,DavisDrusvyatskiyJiang2025}.
 
\subsection{Shifted dynamical system}\label{subsec:shifted-dynamics}
Fix an unstable zero $z^*\in\cZ^*$. Let $\cE_{\rm reg}^*:=\cE_{\rm reg}(z^*)$, where $\cE_{\rm reg}(z^*)$ is given by \cref{assumption:stochastic-oracle}, and fix a seed realization $\omega_\xi\in\cE_{\rm reg}^*$. Without loss of generality, we may assume that the coordinates are translated so that $z^*=0$ (in particular, we will keep the same symbols $F$ and $G$). Then $F(0)=0$ and defining $H:={\rm D}F(0)$, it holds that
\[  g_k(z):=G\bigl(z;\xi^k(\omega_\xi)\bigr),\qquad J_k:={\rm D}g_k(0) - H \qquad \text{and} \qquad  b_k:=g_k(0).
\]
The sequence $\{b_k\}_k$ acts as a perturbation of the recursion and the stochastic recursion takes the form:
\begin{equation}\label{eq:z_recursion}
    z^{k+1}=z^k-\alpha_k g_k(z^k), \qquad \forall\, k\in\N.    
\end{equation}

\paragraph{Taylor expansion.}
Define the Taylor expansion remainder of $g_k:\Rd\to\Rd$ as
\[
    R_k(z):=g_k(z)-g_k(0)-{\rm D}g_k(0)z.
\]
Then, we have $g_k(z)=g_k(0)+(H+J_k)z+R_k(z)$. 
Substituting this into \eqref{eq:z_recursion}, it follows
\begin{equation}\label{eq:z_recursion_new}
    z^{k+1}
    =
    (I-\alpha_k H-\alpha_k J_k)z^k
    -\alpha_k b_k
    -\alpha_k R_k(z^k).
\end{equation}
The next lemma gives a uniform local Lipschitz bound for the Taylor remainders
$R_k$, with a modulus that vanishes as the neighborhood shrinks. Its proof is
deferred to \cref{subsec:proof-Lip-of-R-k}.
\begin{lemma}\label{lem:Lip of R_k}
There exist $\bar\rho>0$ and a modulus $\Upsilon:\mathbb{R}_+\to\mathbb{R}_+$ (both independent of the seed realization) satisfying $\Upsilon(\rho)\to 0$ as $\rho\to 0$, such that
\[
\|{\rm D}R_k(z)\|\le \Upsilon(\rho) \qquad \text{and} \qquad \|R_k(z)-R_k(w)\|\le \Upsilon(\rho)\|z-w\|
\]
for all $z,w\in B_\rho(0)$, all $\rho\in(0,2\bar\rho]$, and all $k\in\N$.
\end{lemma}

\paragraph{Change of variables.}
The linear part in \eqref{eq:z_recursion_new} contains the sample-dependent
Jacobian perturbation $J_k$. We use a change of variables to
absorb the tail contribution of these perturbations and recover a leading
linear term governed by the fixed matrix $H$.
Define the tail aggregate of the linear perturbations via
\[
S_k:={\sum}_{i=k}^\infty \alpha_i J_i.
\]
Fix the parameter $\varepsilon>0$ that shall be specified in \eqref{eq:choose-delta-q}. By \cref{assumption:stochastic-oracle}, it holds that $S_k\to0$ and $S_{k+1}J_k\to0$. Hence, the corresponding tail suprema converge to zero, and we may choose a path-dependent index $m=m(\omega_\xi,\varepsilon)$ such that, for every $k\ge m$,
\begin{equation}\label{eq:tail-smallness}
\|S_k\|\le\tfrac12
\qquad \text{and} \qquad
{\sup}_{\ell\ge k}~(\|S_{\ell+1}\|\|H\|+\|H\|\|S_\ell\|+\|S_{\ell+1}J_\ell\|)\le\varepsilon.
\end{equation}
Next, we introduce the path-dependent matrices
\[
    U_k:=I-S_k \qquad \text{and} \qquad y^k:=U_k\, z^k\qquad \forall\,k\geq m.
\]
Then, each $U_k$ is invertible and admits the bounds
\begin{equation}
    \label{eq:U_k bounds}
    \|U_k\|\le \tfrac32 \qquad \text{and} \qquad \|U_k^{-1}\|\le 2, \qquad \forall\, k\ge m.
\end{equation}
Based on $S_k=\alpha_k J_k+S_{k+1}$, a direct expansion gives
\begin{equation}
    \label{eq:def B_k}
    U_{k+1}(I-\alpha_k H-\alpha_k J_k)
    =
    (I-\alpha_k H)U_k+\alpha_k B_k\qquad \text{where}\qquad B_k:=S_{k+1}H+S_{k+1}J_k-HS_{k}.
\end{equation}
Multiplying the recursion \eqref{eq:z_recursion_new} by $U_{k+1}$ and using the above identity, we obtain the shifted dynamics
\begin{equation}\label{eq:y_recursion}
    y^{k+1}
    =
    (I-\alpha_k H)y^k
    -\alpha_k \underbracket[.8pt][3pt]{ U_{k+1}b_k}_{=:\widetilde b^k} 
    -\alpha_k \underbracket[.8pt][3pt]{ \left(U_{k+1}R_k(U_k^{-1}y^k)-B_kU_k^{-1}y^k\right)}_{=:\widetilde{R}_k(y^k)}.
\end{equation}
Importantly, the change of variables from $\{z^k\}_k$ to $\{y^k\}_k$ removes the
$k$-dependent perturbation $J_k$ from the leading linear component: the term
$I-\alpha_k H-\alpha_k J_k$ is reduced to $I-\alpha_k H$, whose dependence on
$k$ enters only through the scalar step-size $\alpha_k$. This is the key point of
the reduction. It gives a common center, stable, and unstable splitting for the
matrix sequence $\{I-\alpha_k H\}_k$, whereas no such fixed splitting is
available in general for $\{I-\alpha_k H-\alpha_k J_k\}_k$.

\subsection{Subspace splitting and properties of transition matrices}\label{subsec:spectral-transition-bounds}

Since $\alpha_k\to0$, we may increase the index $m$, if necessary, so that $\alpha_k\|H\|\le\frac12$ for every $k\ge m$. After discarding the iterates before $m$ and reindexing all sequences and maps, we retain the same notation. Then \eqref{eq:tail-smallness} holds for every $k\in\N$. The recursion \eqref{eq:y_recursion} can then be restated as follows
\begin{equation}
\label{eq:linear-dynamics}
y^{k+1}
=
M_k y^k-\alpha_k \widetilde{R}_k(y^k)-\alpha_k \widetilde{b}^k,
\qquad\text{where}\qquad
M_k:=I-\alpha_k H\qquad \forall\, k\in\N.
\end{equation}
In particular, using $\alpha_k\|H\| \le \frac12$, each matrix $M_k$ is invertible. Let us define the cumulative time \[t_{k,k}:=0 \qquad\text{and}\qquad t_{k,j}:={\sum}_{i=j}^{k-1}\,\alpha_i \quad \text{for $k>j$}.\]
Since each $M_k$ is an affine function of the fixed real matrix $H$, every generalized $H$-invariant subspace is also invariant under $M_k$. Let $E_{\cs}$ and $E_u$ denote the real spectral subspaces of $H$ associated with $\{\lambda \in \operatorname{spec}(H): \Re(\lambda)\ge0\}$ and $\{\lambda \in \operatorname{spec}(H): \Re(\lambda)<0\}$, respectively. Therefore,
\[
\mathbb{R}^d = E_{\cs}\oplus E_u,
\]
and both $E_{\cs}$ and $E_u$ are invariant under $M_k$ for every $k\in\N$. Note that, in general, this splitting is not orthogonal because $H$ is not necessarily symmetric. We also refer to \Cref{subsec:proof-growth-rates-nonsym} and \cite[Chapter 3]{HornJohnson2012} for further details. Let $P_{\cs}$ and $P_u$ denote the projections onto $E_{\cs}$ and $E_u$, respectively. Since the matrices $\{M_k\}_k$ are invertible, define the forward and backward transition operators
\begin{equation}\label{eq:transition-operators}
\Phi^+(k,j):=
\begin{cases}
M_{k-1}\cdots M_j, & k>j,\\[0.5ex]
I, & k=j,
\end{cases}
\qquad
\Phi^-(k,j):=\left[\Phi^+(k,j)\right]^{-1},
\qquad \forall\, k\ge j.
\end{equation}
Because both $E_{\cs}$ and $E_u$ are invariant under $M_k$ for every $k\in\N$, they are also invariant under $\Phi^{\pm}(k,j)$. Define the restrictions on these subspaces by
\begin{equation}\label{eq:restricted-transition-operators}
\Phi_{\cs}^{\pm}(k,j)=\Phi^{\pm}(k,j)P_{\cs}=P_{\cs}\Phi^{\pm}(k,j)
\qquad \text{and} \qquad 
\Phi_u^{\pm}(k,j)=\Phi^{\pm}(k,j)P_u=P_u\Phi^{\pm}(k,j).
\end{equation}
Below, we present several properties of the residual term $\widetilde{R}_k$ and the transition matrices. The proofs are given in \cref{subsec:proof-psi-properties,subsec:proof-growth-rates-nonsym}, respectively.

\begin{lemma}[Local properties of the residual term]
\label{lem:psi_properties}
Let $\widetilde{R}_k:\Rd\to\Rd$ be defined as in \eqref{eq:y_recursion}, i.e.,
\[
\widetilde{R}_k(y)=U_{k+1}R_k(U_k^{-1}y)-B_kU_k^{-1}y.
\]
Let $\bar\rho>0$ and $\Upsilon:\R_+\to\R_+$ be given in \cref{lem:Lip of R_k}. Then, for every $k\in\N$, every $\rho\in(0,\bar\rho]$, and all $w,y\in B_\rho(0)$, it holds that
\[
\|\widetilde{R}_k(y) - \widetilde{R}_k(w)\| \leq \left[3\Upsilon(2\rho)+2(\|S_{k+1}\|\|H\| + \|H\| \|S_k\| + \|S_{k+1}J_k\|)\right]\cdot \|y-w\|.
\]
\end{lemma}
\begin{lemma}[Transition bounds on spectral subspaces]
\label{lem:growth_rates_nonsym}
Suppose that the step-sizes satisfy $\alpha_\ell\to0$ and
$\alpha_\ell\|H\|\le1/2$ for all $\ell$. Then there exist constants
$0<\kappa<\nu$ and $C_H\ge1$ such that
\[
\|\Phi_u^{-}(k,j)\|\le C_H\exp(-\nu t_{k,j})\qquad \text{and} \qquad 
\|\Phi_{\cs}^{+}(k,j)\|\le C_H\exp(\kappa t_{k,j}) \qquad \forall\,k \geq j.
\]
\end{lemma}

\subsection{Truncated dynamics and Lyapunov--Perron operator}\label{subsec:truncated-lp}
We first localize the nonlinear residual while leaving it unchanged near the
origin. Let $0<\delta\le\bar\rho/2$ be a radius, to be fixed in the later analyses. Let $\chi:\Rd\to[0,1]$ be a $\cC^1$-cutoff such that
\begin{equation}\label{eq:truncation}
\chi(y)=1\quad \text{for }\|y\|\le\delta,\qquad
\chi(y)=0\quad \text{for }\|y\|\ge2\delta\qquad \text{and} \qquad
\|\nabla\chi\|_\infty\le c_\chi/\delta.
\end{equation}
Here, $c_\chi>0$ is independent of $\delta$ and the seed realization.
Next, we define the truncated nonlinear term and its associated dynamics:
\begin{equation}
\label{eq:dyn-trunc}
R_k^\delta(y):=\chi(y)\widetilde R_k(y)\qquad \text{and} \qquad y^{k+1}=M_k y^k-\alpha_k R_k^\delta(y^k)-\alpha_k\widetilde b^k,
\end{equation}
where $\widetilde b^k$ is defined in \eqref{eq:y_recursion}. 
Clearly, the truncated dynamics \eqref{eq:dyn-trunc} agrees with \eqref{eq:linear-dynamics} when $\|y^k\|\le\delta$. In the following, we state Lipschitz properties of the truncated residual terms $\{R_k^\delta\}_k$. The proof is deferred to \cref{subsec:proof-global-psi-property}.
\begin{lemma}[Lipschitz property of residual terms]\label{lem:global_psi_property}
Let $0<\delta\le\bar\rho/2$. Denote
\[
\ell_\delta:={\sup}_{k\in\N}\operatorname{Lip}\bigl(\widetilde R_k;B_{2\delta}(0)\bigr)\qquad \text{where}\qquad \operatorname{Lip}(\widetilde R_k;B_{2\delta}(0)):={\sup}_{\substack{y,w\in B_{2\delta}(0)\\ y\ne w}}\frac{\|\widetilde R_k(y)-\widetilde R_k(w)\|}{\|y-w\|}.
\]
Then, $R_k^\delta:\Rd\to\Rd$ is globally Lipschitz, uniformly in $k$, with modulus $L_\delta:=(1+2c_\chi)\ell_\delta$.
\end{lemma}

Using the notation of \cref{lem:growth_rates_nonsym}, choose a weight parameter $
\theta\in(\kappa,\nu)$. 
Furthermore, choose $\varepsilon>0$ in \eqref{eq:tail-smallness} sufficiently small and independently of the seed realization. Then fix $0<\delta\le\bar\rho/2$, also independently of the seed realization, such that
\begin{equation}\label{eq:choose-delta-q}
C_H(1+2c_\chi)\bigl[3\Upsilon(4\delta)+2\varepsilon\bigr]
\left(\frac{1}{\theta-\kappa}+\frac{1}{\nu-\theta}\right)<1.
\end{equation}
The choice of $m$ in \eqref{eq:tail-smallness} gives $
\sup_{k\in\N}(\|S_{k+1}\|\|H\|+\|H\|\|S_k\|+\|S_{k+1}J_k\|)\le\varepsilon$. 
By \cref{lem:psi_properties}, with $\rho=2\delta$, and by \cref{lem:global_psi_property}, it holds that
\[
L_\delta\le (1+2c_\chi)\bigl(3\Upsilon(4\delta)+2\varepsilon\bigr).
\]

\begin{definition}[Weighted sequence space and Lyapunov--Perron (LP) operator]
\label{def:LP-operator}
Define
\[
\cY_\theta:=\left\{\by=(y^0,y^1,\ldots):
\|\by\|_\theta:={\sup}_{k\in\N}\exp(-\theta t_k)\|y^k\|<\infty\right\}\qquad \text{with}\qquad t_k:=t_{k,0}.
\]
This is a Banach space equipped with $\|\cdot\|_\theta$. 
For $\zeta\in E_{\cs}$ and $\by\in\cY_\theta$, define the LP operator by %
\begin{equation}
\label{eq:LP} 
\begin{aligned}
(\cT_\zeta \by)_{\cs}^k
&:= \Phi_{\cs}^{+}(k,0)\zeta
-{\sum}_{j=0}^{k-1}\alpha_j\Phi_{\cs}^{+}(k,j+1)
\Bigl(R_j^\delta(y^j)+\widetilde b^j\Bigr),\\
(\cT_\zeta \by)_{u}^k
&:= {\sum}_{j=k}^{\infty}\alpha_j\Phi_u^{-}(j+1,k)
\Bigl(R_j^\delta(y^j)+\widetilde b^j\Bigr),\\
(\cT_\zeta \by)^k
&:= (\cT_\zeta \by)_{\cs}^k+(\cT_\zeta \by)_u^k.
\end{aligned}
\end{equation}
\end{definition}
To identify trajectories that remain near the unstable zero, we use a variation-of-parameters approach. The center-stable component is propagated forward from its initial value, and the unstable component is represented by a backward tail. In the subsequent results, we show that the LP operator is well-defined and a contraction on $\cY_\theta$, and that it allows characterizing trajectories of \eqref{eq:dyn-trunc}.
 
\begin{lemma}[Well-definedness]\label{lem:well_define}
For the fixed seed realization $\omega_\xi\in\cE_{\rm reg}^*$ and every
$\zeta\in E_{\cs}$, the formulas in \eqref{eq:LP} define a map
$\cT_\zeta:\cY_\theta\to\cY_\theta$.
\end{lemma}
The proof estimates the forward center-stable sum and the backward unstable
tail in \eqref{eq:LP} under the $\|\cdot\|_\theta$ norm; see
\cref{subsec: proof of well_define}. Next, we present the contraction properties for the operator $\cT_\zeta$. 

\begin{lemma}[Contraction]\label{lem:LP_contraction} 
The operator $\cT_\zeta$ is a contraction on $(\cY_\theta,\|\cdot\|_\theta)$, i.e., 
for all $\by,\bw\in\cY_\theta$, we have
\begin{equation}\label{eq:def q}
\|\cT_\zeta(\by)-\cT_\zeta(\bw)\|_\theta
\le q\,\|\by-\bw\|_\theta
\qquad \text{where} \qquad
q:=C_HL_\delta\left(\frac{1}{\theta-\kappa}+\frac{1}{\nu-\theta}\right)<1.
\end{equation}
\end{lemma}
The proof is deferred to \cref{subsec: proof of LP_contraction}. 
This lemma ensures the existence of a unique fixed point of $\cT_\zeta$ for each choice of the center-stable coordinate $\zeta$. The initial unstable component of this fixed point is then determined by $\zeta$. The next result uses this
dependence to define a graph over $E_{\cs}$ and shows that every truncated
trajectory in $\cY_\theta$ starts on that graph.

\begin{proposition}[Characterizing trajectories]\label{lem:characterization}
There exists a Lipschitz map $h:E_{\cs}\to E_u$ such that every orbit $\by$ of \eqref{eq:dyn-trunc} satisfying $\by\in\cY_\theta$ has $y^0\in \cG_h:=\{\zeta+h(\zeta):\zeta\in E_{\cs}\}$ with $\mu(\cG_h)=0$.
\end{proposition}
\begin{proof} 
For each $\zeta\in E_{\cs}$, \cref{lem:LP_contraction} gives a unique
fixed point $\by_\zeta=\cT_\zeta(\by_\zeta)$, whose $k$-th term is denoted by $y_\zeta^k$. Evaluating the center-stable component in \eqref{eq:LP} at $k=0$ yields
$P_{\cs}y_\zeta^0=(\cT_\zeta\by_\zeta)_{\cs}^0=\zeta$. This means that the center-stable initial component is fixed by $\zeta$, and the unstable initial component is determined by the fixed point sequence $\by_\zeta$. Define the mapping and the associated graph
\begin{equation}\label{eq:def h}
h(\zeta):=P_u y_\zeta^0\qquad \text{and} \qquad
\cG_h:=\{\zeta+h(\zeta):\zeta\in E_{\cs}\}.
\end{equation}
We now show that every truncated trajectory in $\cY_\theta$ coincides with a fixed point of $\cT_\zeta$ for a certain choice of $\zeta$. Specifically, let $\by\in\cY_\theta$ be an
arbitrary orbit of \eqref{eq:dyn-trunc}, and set $\zeta:=P_{\cs}y^0$; we will
show that $\by=\by_\zeta$. For ease of notation, the $k$-th term of $\by$ is denoted by $y^k$, and its projections onto $E_{\cs}$ and $E_u$ are denoted by $y_{\cs}^k$ and $y_u^k$, respectively. Iterating the center-stable component of
\eqref{eq:dyn-trunc}, we obtain the center-stable identity in \eqref{eq:LP}, i.e.,
\[
y_{\cs}^k
=\Phi_{\cs}^{+}(k,0)\zeta
-{\sum}_{j=0}^{k-1}\alpha_j\Phi_{\cs}^{+}(k,j+1)\bigl(R_j^\delta(y^j)+\widetilde b^j\bigr).
\]
For $m>k$, the unstable component satisfies
\begin{equation}
    \label{eq:unstable component}
    y_u^k
=\Phi_u^{-}(m,k)y_u^m
+{\sum}_{j=k}^{m-1}\alpha_j\Phi_u^{-}(j+1,k)\bigl(R_j^\delta(y^j)+\widetilde b^j\bigr).
\end{equation}
Moreover, by \cref{lem:growth_rates_nonsym}, we have for all $m\geq k$,
\[
\|\Phi_u^{-}(m,k)y_u^m\|
\le C_H\exp(-\nu t_{m,k})\exp(\theta t_m)\|\by\|_\theta
=C_H\exp(\nu t_k)\exp(-(\nu-\theta)t_m)\|\by\|_\theta\to0.
\]
Letting $m\to\infty$ in \eqref{eq:unstable component} yields the unstable identity $y_u^k = (\cT_\zeta\by)_u^k$ in \eqref{eq:LP}. Hence, $\by=\cT_\zeta(\by)$. Uniqueness of the fixed point gives $\by=\by_\zeta$; hence, $y^0=\zeta + P_uy^0=\zeta+h(\zeta)\in\cG_h$. Finally, we show that the graph $\cG_h$ has measure zero. To this end, we use the
Lipschitz continuity of $h$. The proof of the following claim is deferred to
\cref{subsec: proof of Lipschitz h}.

\begin{claim}\label{claim:Lipschitz-h}
The mapping $h:E_{\cs}\to E_u$ defined in \eqref{eq:def h} is Lipschitz continuous.
\end{claim}
Since $E_u\neq\{0\}$, the space $E_{\cs}$ is a proper subspace of $\Rd$. After a linear change of coordinates,
$\cG_h$ is the graph of a Lipschitz map from $\R^{\dim(E_{\cs})}$ to
$\R^{\dim(E_{u})}$. Therefore, it has Lebesgue measure zero.
\end{proof}

\subsection{Completion of the local null-set argument}
\label{subsec:local-null-set-argument}
Returning to the original indexing, we now complete the proof of \cref{thm:local-stable-null}. Fix
$z^*\in\cZ^*$ and a seed realization
$\omega_\xi\in\cE_{\rm reg}(z^*)$. By construction, the radius $\delta$ in \eqref{eq:choose-delta-q} is independent of the seed realization. Hence, the quantity $\delta_0:=2\delta/3$ is independent of $\omega_\xi$. Since $\alpha_k\to0$ and the tail quantities in \eqref{eq:tail-smallness} vanish, choose $m_0\in\N$ sufficiently large so that the transition bounds in \cref{lem:growth_rates_nonsym} apply and such that, for every $k\ge m_0$,
\[
\alpha_k\|H\|\le\tfrac12,\qquad
\|S_k\|\le\tfrac12
\qquad \text{and} \qquad
{\sup}_{\ell\ge k}~(\|S_{\ell+1}\|\|H\|+\|H\|\|S_\ell\|+\|S_{\ell+1}J_\ell\|)\le\varepsilon.
\]
Fix $m\ge m_0$. Discard the first $m$ iterates and relabel $m$ as the
initial index. Let $\cG_{h,m}$ denote the null Lipschitz graph supplied by
\cref{lem:characterization} for this reindexed system, and
define
\[ 
\mathcal N_m:=
\{x\in\Rd:U_m(x-z^*)\in\cG_{h,m}\}\qquad \text{where}\qquad U_m:=I-S_m.
\]
Since $U_m$ is invertible and $\mu(\cG_{h,m})=0$, this affine preimage is
Lebesgue null. Next, we consider an arbitrary orbit satisfying
\[
\|z^k-z^*\|\le\delta_0\qquad\forall\, k\ge m.
\]
In the reindexed and centered coordinates, set
\[
\bar z^j:=z^{m+j}-z^*
\qquad \text{and}\qquad 
y^j:=U_{m+j}\bar z^j,
\qquad \forall\,j\ge0.
\] 
By the bounds $\|S_{m+j}\|\leq \frac12$ and $\|\bar{z}^j\|\leq\delta_0$ for all $j\geq 0$, we have $
\|y^j\|\le \|U_{m+j}\|\,\|\bar z^j\|
\le \tfrac32\delta_0=\delta$. 
Hence, the cutoff is inactive along the trajectory, and the truncated dynamics \eqref{eq:dyn-trunc} agree with the shifted dynamics \eqref{eq:linear-dynamics} on this orbit. In addition, we have $\by\in\cY_\theta$ due to
\[
\|\by\|_\theta
\le \sup_{j\in\N}\exp(-\theta t_j)\delta
\le \delta<\infty.
\]
Applying \cref{lem:characterization} to the reindexed
system gives $y^0\in\cG_{h,m}$. Since $y^0=U_m(z^m-z^*)$, the definition of
$\mathcal N_m$ yields $z^m\in\mathcal N_m$. Thus, for every $m\ge m_0$, any
orbit trapped in $B_{\delta_0}(z^*)$ from time $m$ onward has its $m$-th iterate
in the Lebesgue null set $\mathcal N_m$. \hfill $\qed$

\section{Application I: Stochastic approximation-based methods}
\label{sec:application-sa-methods}

We first consider stochastic approximation methods under typical i.i.d.\ sampling conditions. In these applications, sampled
gradients are used to approximate the gradient of an expected loss, and our main task is to verify the pathwise oracle conditions formulated in \Cref{assumption:stochastic-oracle}.

\begin{assumption}[Sample gradients and Hessians]
\label{assumption:sample-losses}
Let $(\Xi,\cA,\Prob_\Xi)$ be the single-sample probability space. Let
$\xi$ denote a generic sample with values in $\Xi$, and let
$\{\xi^k\}_{k\in\N}$ be the samples used by the algorithm. We assume that
$\{\xi^k\}_{k\in\N}$ are i.i.d.\ with common law $\Prob_\Xi$. Thus, each
$\xi^k$ takes values in $\Xi$. Let
$\cX\subseteq\Rd$ be open. For each $\xi\in\Xi$, let
$f(\cdot;\xi):\cX\to\R$ be $\cC^2$, and define
$f(x)=\Exp[f(x;\xi)]$ for all $x\in\cX$. We make the following assumptions:
\begin{itemize}
    \item \emph{(local Hessian Lipschitzness).} For every compact
    $\mathcal Q\subset\cX$, there is $L_{\mathcal Q}<\infty$ such that
    \[
        \|\nabla^2 f(x;\xi)-\nabla^2 f(y;\xi)\|
        \le L_{\mathcal Q}\|x-y\|
        \qquad
        \forall\,x,y\in\mathcal Q,\quad \forall\,\xi\in\Xi.
    \]
    \item \emph{(unbiased derivatives).} For every fixed $\bar x\in\cX$
    and every $k\in\N$,
    \[
        \Exp[\nabla f(\bar x;\xi^k)]=\nabla f(\bar x)
        \qquad \text{and} \qquad
        \Exp[\nabla^2 f(\bar x;\xi^k)]=\nabla^2 f(\bar x).
    \]
    \item \emph{(finite moment bounds).} There exists an exponent $q>2$ such
    that, for every fixed $\bar x\in\cX$ and every $k\in\N$,
    \[
        \Exp[\|\nabla f(\bar x;\xi^k)-\nabla f(\bar x)\|^2]
        <\infty
        \qquad \text{and} \qquad
        \Exp[\|\nabla^2 f(\bar x;\xi^k)-\nabla^2 f(\bar x)\|^q]
        <\infty.
    \]
\end{itemize}
\end{assumption}
The conditions in \cref{assumption:sample-losses} are used to verify the
pathwise oracle condition in the abstract theorem. The i.i.d.\ condition
corresponds to sampling with replacement. The Lipschitz condition on the sample Hessians is local, which is natural in finite-sum models.
The unbiasedness and finite moment bounds provide the
pointwise stochastic estimates needed in
\cref{lem:iid-moment-bounds-imply-stochastic-oracle}. 
We now list several examples in which the finite $q$-th moment bound holds.

\begin{example}[Finite-sum objectives]
Consider $f(x)=\frac1n{\sum}_{i=1}^n f_i(x)$, where each $f_i$ is $\cC^2$ on
$\cX$. Take $\Xi=\{1,\ldots,n\}$ with the uniform distribution and set
$f(\cdot;i):=f_i$. If the samples $\{\xi^k\}_{k\in\N}$ are drawn i.i.d.\ from
this distribution, then,
for every fixed $\bar x\in\cX$ and every $k\in\N$,
\[
\Exp[\nabla f(\bar x;\xi^k)]=\nabla f(\bar x)
\qquad \text{and} \qquad
\Exp[\nabla^2 f(\bar x;\xi^k)]=\nabla^2 f(\bar x).
\]
Moreover, $
\Exp[\|\nabla f(\bar x;\xi^k)-\nabla f(\bar x)\|^2]
\le
\max_{1\le i\le n}\|\nabla f_i(\bar x)-\nabla f(\bar x)\|^2
<\infty$ and 
\[
\Exp[\|\nabla^2 f(\bar x;\xi^k)-\nabla^2 f(\bar x)\|^q]
\le
{\max}_{1\le i\le n}\|\nabla^2 f_i(\bar x)-\nabla^2 f(\bar x)\|^q
<\infty.
\]
Thus, when 
the Hessians $\nabla^2 f_i$ are locally Lipschitz on $\cX$, the finite moment bound holds for all $q > 2$. 
\end{example}

Next, we consider expected-loss models with general sampling distributions.

\begin{example}[Uniformly smooth sample losses] \label{ex:uniform}
Consider $f(x)=\Exp[f(x;\xi)]$ and suppose that each sample loss
$f(\cdot;\xi)$ is $\cC^2$ on $\cX$ and has $\sL$-Lipschitz gradient ---
uniformly in $\xi$. Then, we have
\[
\|\nabla^2 f(x;\xi)\|\le \sL
\qquad
\forall\,x\in\cX,\quad \forall\,\xi\in\Xi.
\]
Moreover, under the condition
$\Exp[\nabla^2 f(\bar x;\xi)]=\nabla^2 f(\bar x)$, we can infer $\|\nabla^2 f(\bar x)\|\le\sL$, and it follows that
\[
\Exp[\|\nabla^2 f(\bar x;\xi)-\nabla^2 f(\bar x)\|^q]
\le (2\sL)^q<\infty
\qquad
\forall\,q>0.
\]
Consequently, the finite $q$-th moment bound condition naturally holds in this setting.
\end{example}

We now show that sub-Gaussian noises satisfy the finite
moment conditions. This does not require each $f(\cdot;\xi)$ to satisfy Lipschitz-smoothness conditions as in \Cref{ex:uniform}.

\begin{example}[Sub-Gaussian noise]
The finite moment bounds in \cref{assumption:sample-losses} follow
directly from the standard equivalence between sub-Gaussian tails and moment
growth \cite[Proposition~2.5.2]{Vershynin2018HighDimensionalProbability}.
Specifically, fix $\bar x\in\cX$ and suppose
that, for some constant $K_{\bar x}>0$, we have
\begin{align*}
    \Prob\!\left(
    \left|\langle \nabla f(\bar x;\xi^k)-\nabla f(\bar x),u\rangle\right|
    \ge t
\right)
&\le
2\exp\!\left(-{t^2}{K_{\bar x}^{-2}}\right)
\qquad
\forall\,\|u\|=1,\quad \forall\,k\in\N,\quad \forall\,t\ge0,\\
\Prob\!\left(
    \left|\langle \nabla^2 f(\bar x;\xi^k)-\nabla^2 f(\bar x),U\rangle\right|
    \ge t
\right)
&\le
2\exp\!\left(-{t^2}{K_{\bar x}^{-2}}\right)
\qquad
\forall\,\|U\|_F=1,\quad \forall\,k\in\N,\quad \forall\,t\ge0.
\end{align*}
Here, the space of symmetric $d \times d$ matrices is equipped with the Frobenius inner product. Applying \cite[Proposition~2.5.2]{Vershynin2018HighDimensionalProbability} to
coordinate projections gives
\[
\Exp[\|\nabla f(\bar x;\xi^k)-\nabla f(\bar x)\|^2]<\infty
\qquad \text{and} \qquad
\Exp[\|\nabla^2 f(\bar x;\xi^k)-\nabla^2 f(\bar x)\|^q]<\infty
\]
for every $q\ge1$ and $k\in\N$. Thus, the finite moment bounds in
\cref{assumption:sample-losses} hold.
\end{example}

\subsection{Stochastic mirror descent}
\label{subsec:smd-application}
This subsection applies the abstract avoidance theorem to stochastic mirror
descent. In particular, we establish almost sure non-convergence to strict saddles without the unit-excitation condition \eqref{eq:assumption UE}.

Let $\cX\subset\Rd$ be open and convex and consider the expected loss problem
\begin{equation}\label{eq:smd-objective}
{\min}_{x\in\cX}~f(x):=\Exp[f(x;\xi)].
\end{equation}
Here, $\xi$ denotes a generic sample, while the samples used by the algorithm
are denoted by $\{\xi^k\}_k$. Let $h$ be a distance-generating function (also called a kernel) and denote the associated Bregman divergence by 
\[
    D_h(y,x):=h(y)-h(x)-\langle\nabla h(x),y-x\rangle.
\]
We study the following version of the stochastic mirror descent method
\cite{nemjudlansha08,bottou2018optimization}.

\begin{algorithm}[h]
\caption{Stochastic mirror descent}
\label{alg:smd-iteration}
\begin{algorithmic}
    \Require Choose an initial point $x^0\in\cX$ and step-sizes
    $\{\alpha_k\}_{k\ge0}$.
    \For{$k=0,1,2,\ldots$}
        \State Draw a sample $\xi^k$ and set
        \[
        x^{k+1}
        =
        {\argmin}_{y\in\overline{\cX}}~\left\{
        \langle \nabla f(x^k;\xi^k),y-x^k\rangle
        + {\alpha_k}^{-1}D_h(y,x^k)
        \right\}.
        \]
    \EndFor
\end{algorithmic}
\end{algorithm}
Using the quadratic (Euclidean) kernel $h(x)=\frac12\|x\|^2$, we obtain
$D_h(y,x)=\frac12\|y-x\|^2$. In this special case and if $\cX=\Rd$, \cref{alg:smd-iteration} reduces to the classical stochastic gradient descent method:
\[
    x^{k+1}=x^k-\alpha_k\nabla f(x^k;\xi^k).
\]
Hence, the results derived in this section naturally apply to SGD\@. Throughout this subsection, we work under the following assumptions on the mirror geometry and the sampled
losses.

\begin{assumption}
\label{assumption:smd-regularity}
We consider the following conditions:
\begin{enumerate}[label=\emph{(\roman*)}, leftmargin=*] 
\item \emph{Mirror geometry.}
The kernel $h:\Rd\to(-\infty,+\infty]$ is essentially
smooth\footnote{Essential smoothness means that whenever
$x^j\in\cX$ converges to a boundary point of $\operatorname{dom}h$,
one has $\|\nabla h(x^j)\|\to\infty$.} with
\[
\cX=\operatorname{int}(\operatorname{dom}h)
\qquad \text{and} \qquad
\operatorname{dom}h^*=\Rd.
\]
On $\cX$, $h$ is $\cC^2$ and satisfies
$\nabla^2 h(x)\succ 0$ for all $x\in\cX$.

\item \emph{Relative smoothness.}
Each $f(\cdot;\xi)$ is relatively smooth with respect to $h$: There
is $\sL>0$ such that
\begin{equation}\label{eq:smd-relative-smoothness}
\left|
f(y;\xi)-f(x;\xi)-\langle\nabla f(x;\xi),y-x\rangle
\right|
\le \sL D_h(y,x)
\qquad
\forall\,x,y\in\cX,\quad \forall\,\xi\in\Xi.
\end{equation}
\end{enumerate}
\end{assumption}
The mirror geometry assumption is a standard regularity condition for the kernel $h$, which ensures the well-definedness of the mirror update. 
Relative smoothness is the alternative to global Lipschitz gradient
continuity used in Bregman first-order methods
\cite{BauschkeBolteTeboulle2017,LuFreundNesterov2018}. For the quadratic kernel $h(x)=\frac12\|x\|^2$, \eqref{eq:smd-relative-smoothness} reduces to the usual descent property.

We now reformulate the mirror descent iteration in the recursion \eqref{eq:stochastic-fixed-point-method} considered by the abstract avoidance theorem. By the conjugacy theorem \cite[Theorem~26.5]{Rockafellar1970ConvexAnalysis},
$\nabla h:\cX\to\operatorname{int}(\operatorname{dom}h^*)=\Rd$ is bijective,
with inverse $\nabla h^*$. Since $\nabla^2h$ is positive definite on
$\cX$, the inverse function theorem
\cite[Theorem~C.34]{Lee2012IntroductionSmoothManifolds} shows that these maps
are inverse $\cC^1$-diffeomorphisms. Consequently, we have $h^*\in\cC^2(\Rd)$ and 
\[
\nabla h^*(z)=(\nabla h)^{-1}(z)
\qquad \text{and} \qquad
\nabla^2h^*(z)=
\bigl[\nabla^2h(\nabla h^*(z))\bigr]^{-1} \qquad \forall\,z\in\Rd.
\]
Set $z^k:=\nabla h(x^k)$. The optimality condition for the mirror descent update in
\cref{alg:smd-iteration} gives
\begin{equation}\label{eq:smd-dual-fixed-point}
z^{k+1}=z^k-\alpha_kG(z^k;\xi^k)
\qquad \text{where} \qquad
G(z;\xi):=\nabla f(\nabla h^*(z);\xi),
\end{equation}
see \cite[Section 2]{BeckTeboulle2003}. The associated mean field is 
\[ F(z):=\nabla f(\nabla h^*(z)). \] 
Thus, stochastic mirror descent fits the stochastic recursion \eqref{eq:stochastic-fixed-point-method}. Next, we verify that
the oracle $G$ defined in \eqref{eq:smd-dual-fixed-point} satisfies the structural conditions in \cref{assumption:F}
and the stochastic oracle conditions in \cref{assumption:stochastic-oracle} required by the avoidance framework. 

\begin{proposition}
\label{prop:smd-sample-losses-imply-stochastic-oracle}
Suppose
\cref{assumption:smd-regularity,assumption:sample-losses} hold and the
step-sizes satisfy
\[
{\sum}_{k=0}^\infty \alpha_k=\infty\qquad \text{and} \qquad {\sum}_{k=0}^{\infty}r_{k+1}^{q/2}<\infty \quad \text{with}\quad r_k:={\sum}_{i=k}^{\infty}\alpha_i^2.
\]
Then, \cref{assumption:F,assumption:stochastic-oracle} hold for $G(\cdot;\xi):=\nabla f(\nabla h^*(\cdot);\xi)$.
\end{proposition}

\begin{proof}
We first verify the regularity requirements in \cref{assumption:F}, and then the stochastic oracle conditions in \cref{assumption:stochastic-oracle}. We begin with two estimates used in both parts. For $z\in\Rd$, the chain rule gives
\begin{equation}\label{eq:smd-chain-rule-jacobians}
{\rm D}G(z;\xi)=\nabla^2 f(\nabla h^*(z);\xi)\nabla^2h^*(z)
\qquad \text{and} \qquad
{\rm D}F(z)=\nabla^2f(\nabla h^*(z))\nabla^2h^*(z).
\end{equation}
Together with the $\cC^2$-regularity stated in
\cref{assumption:smd-regularity,assumption:sample-losses}, \eqref{eq:smd-chain-rule-jacobians} implies that $F$ is $\cC^1$ on $\Rd$, which is the regularity condition in \cref{assumption:F}.
Since $f(\cdot;\xi)$ and $h$ are $\cC^2$, by \cite[Lemma 2.1]{bolte2018first}, the relative smoothness implies that
\[
-\sL\nabla^2h(x)
\preceq
\nabla^2 f(x;\xi)
\preceq
\sL\nabla^2h(x)
\qquad
\forall\,x\in\cX,\quad \forall\,\xi\in\Xi.
\]
Hence, for every compact set $\mathcal Q_x\subset\cX$, it holds that
\begin{equation}\label{eq:smd-compact-hessian-bound}
{\sup}_{\xi\in\Xi}{\sup}_{x\in\mathcal Q_x}
\|\nabla^2 f(x;\xi)\|
\le
\sL\cdot {\sup}_{x\in\mathcal Q_x}\|\nabla^2 h(x)\|<\infty.
\end{equation}

\noindent\textbf{Part I:} Verifying \cref{assumption:F}.
We first prove that $z\mapsto z-\alpha G(z;\xi)$ is a lipeomorphism for
every $\xi\in\Xi$ and every $0<\alpha <1/\sL$. Fix such $\xi$ and
$\alpha$, and define
\[
H_{\alpha,\xi}:=h-\alpha f(\cdot;\xi),
\qquad
S_{\alpha,\xi}:= \nabla H_{\alpha,\xi} = \nabla h-\alpha\nabla f(\cdot;\xi).
\]
By the relative smoothness, $\sL h-f(\cdot;\xi)$ is convex on $\cX$, see \cite{BauschkeBolteTeboulle2017,bolte2018first,LuFreundNesterov2018}. After rearranging, we obtain
\[
H_{\alpha,\xi}
=
(1-\alpha\sL)h+\alpha(\sL h-f(\cdot;\xi)),
\]
and thus, $H_{\alpha,\xi}$ is strictly convex on $\cX$, since $1-\alpha\sL>0$. Applying the well-definedness result for mirror descent \cite[Theorem 26.5]{Rockafellar1970ConvexAnalysis}, this yields, for every $y\in\Rd$, a unique minimizer $x_y\in\cX$ of the problem
\[
{\min}_{x\in\overline{\cX}}~\left\{H_{\alpha,\xi}(x)-\langle y,x\rangle\right\}.
\]
The associated first-order necessary condition is $S_{\alpha,\xi}(x_y)=y$, see \cite[Theorem 27.4]{Rockafellar1970ConvexAnalysis}. Since $y$ is arbitrary
and the minimizer is unique, the map $S_{\alpha,\xi}:\cX\to\Rd$ is bijective.

We next show that this bijection is $\cC^1$ and has a $\cC^1$-inverse.
Relative smoothness implies
\[
{\rm D}S_{\alpha,\xi}(x)
=
\nabla^2h(x)-\alpha\nabla^2 f(x;\xi)
\succeq
(1-\alpha\sL)\nabla^2h(x)
\succ 0 \qquad \text{for every $x\in\cX$}.
\]
Hence, the inverse function theorem gives a local $\cC^1$-inverse around every
point. Since $S_{\alpha,\xi}$ is globally bijective, these local inverses
agree, and $S_{\alpha,\xi}$ is a $\cC^1$-diffeomorphism from $\cX$ to $\Rd$.
Finally, noting that
\[
z-\alpha G(z;\xi)=S_{\alpha,\xi}(\nabla h^*(z)),
\]
and that $\nabla h^*:\Rd\to\cX$ is a $\cC^1$-diffeomorphism, we infer that the mapping $z\mapsto z-\alpha G(z;\xi)$ is a $\cC^1$-diffeomorphism from $\Rd$ to $\Rd$. Hence, the lipeomorphism condition in
\cref{assumption:F} holds for any fixed $\tilde\alpha<1/\sL$.

It remains to verify the local equicontinuity of the transformed Jacobians. Fix a compact set $\mathcal Q_z\subset\Rd$. After replacing $\mathcal Q_z$ by its convex hull, set $\mathcal Q_x:=\nabla h^*(\mathcal Q_z)\subset\cX$. Since
$\nabla h^*$ is continuous, $\mathcal Q_x$ is compact. Since
$h^*\in\cC^2(\Rd)$ and $\mathcal Q_z$ is compact and convex, there exists a
constant $M_{\mathcal Q}\ge1$ such that, for all $z,w\in\mathcal Q_z$, we have
\[
\|\nabla^2h^*(z)\|\le M_{\mathcal Q}\qquad \text{and}\qquad
\|\nabla h^*(z)-\nabla h^*(w)\|\le M_{\mathcal Q}\|z-w\|.
\]
The local Lipschitz continuity of the Hessian in \cref{assumption:sample-losses} and the bound \eqref{eq:smd-compact-hessian-bound} give, after increasing $M_{\mathcal Q}$ if necessary, for all $x,y\in\mathcal Q_x$ and all $\xi\in\Xi$ that
\[
\|\nabla^2 f(x;\xi)\|\le M_{\mathcal Q}\qquad \text{and}\qquad
\|\nabla^2 f(x;\xi)-\nabla^2 f(y;\xi)\|
\le M_{\mathcal Q}\|x-y\|.
\]
By the chain rule \eqref{eq:smd-chain-rule-jacobians}, for
$z,w\in\mathcal Q_z$ and $\xi\in\Xi$, we then obtain
\begin{align*}
&\hspace{5mm}\|{\rm D}G(z;\xi)-{\rm D}G(w;\xi)\|\\
&\le
\|\nabla^2 f(\nabla h^*(z);\xi)-\nabla^2 f(\nabla h^*(w);\xi)\|
\|\nabla^2h^*(z)\|
+\|\nabla^2 f(\nabla h^*(w);\xi)\|
\|\nabla^2h^*(z)-\nabla^2h^*(w)\| \\
&\le
M_{\mathcal Q}^2\|\nabla h^*(z)-\nabla h^*(w)\|
+M_{\mathcal Q}\|\nabla^2h^*(z)-\nabla^2h^*(w)\| \\
&\le
M_{\mathcal Q}^3\|z-w\|
+M_{\mathcal Q}\|\nabla^2h^*(z)-\nabla^2h^*(w)\|.
\end{align*}
Since $h^*\in\cC^2(\Rd)$, the Hessian $\nabla^2h^*$ is uniformly continuous
on $\mathcal Q_z$. The above estimate is independent of $\xi$, and therefore
the family $\{{\rm D}G(\cdot;\xi)\}_{\xi\in\Xi}$ is equicontinuous on
$\mathcal Q_z$.\\[-2mm]

\noindent\textbf{Part II:} Verifying \cref{assumption:stochastic-oracle}. 
Fix $\bar z\in\Rd$ and set $\bar x:=\nabla h^*(\bar z)$. The relation
\eqref{eq:smd-chain-rule-jacobians} gives
\begin{align*}
b_k&:=G(\bar z;\xi^k)-F(\bar z)
=
\nabla f(\bar x;\xi^k)-\nabla f(\bar x)\qquad \text{and}\\
J_k&:={\rm D}G(\bar z;\xi^k)-{\rm D}F(\bar z)=
\bigl(\nabla^2 f(\bar x;\xi^k)-\nabla^2f(\bar x)\bigr)
\nabla^2h^*(\bar z).
\end{align*}
It follows from \cref{assumption:sample-losses} that $\Exp[b_k]=0$,
$\Exp[\|b_k\|^2]<\infty$, $\Exp[J_k]=0$, and
$\Exp[\|J_k\|_F^q]<\infty$. Moreover, since both $b_k$ and $J_k$ are functions
of the same i.i.d.\ sample $\xi^k$, the sequence $\{(b_k,J_k)\}$ is
i.i.d. Thus, the hypotheses of
\cref{lem:iid-moment-bounds-imply-stochastic-oracle} are verified. Applying \cref{lem:iid-moment-bounds-imply-stochastic-oracle} provides
the desired probability-one event $\cE_{\rm reg}(\bar z)$ in
\cref{assumption:stochastic-oracle}.
\end{proof}

For the smooth objective in \eqref{eq:smd-objective}, we define the set of strict saddle points
\[
\cX^*:=\{x\in\cX:\nabla f(x)=0,\ \lambda_{\min}(\nabla^2f(x))<0\}.
\]

This is the same strict saddle notion used by Lee et al.
\cite{LeePPSJR2019}. For
mirror descent-type methods, the natural state is the mirror coordinate
$z=\nabla h(x)$, and the mean field is
$F(z)=\nabla f(\nabla h^*(z))$. Thus, we need to show that strict
saddles of the primal objective correspond to unstable zeros of this mean field. 

\begin{lemma}[Mirror map preserves strict saddles]
\label{lem:smd-strict-saddle-transfer}
Let $x^*\in\cX^*$ be given and set $z^*:=\nabla h(x^*)$. Then, $F(z^*)=0$, and ${\rm D}F(z^*)$ has a negative eigenvalue.
\end{lemma}

\begin{proof}
Let us set $x=\nabla h^*(z)$. Since $\nabla^2h^*(z)=[\nabla^2h(x)]^{-1}$, it holds that $
{\rm D}F(z)=\nabla^2 f(x)[\nabla^2h(x)]^{-1}$.
At the point $x^*$, the matrix ${\rm D}F(z^*)$ is similar to
\[
[\nabla^2h(x^*)]^{-\frac12}
\nabla^2f(x^*)
[\nabla^2h(x^*)]^{-\frac12}.
\]
By Sylvester's law of inertia \cite[Theorem 4.5.8]{HornJohnson2012}, this symmetric matrix has a negative eigenvalue whenever $\nabla^2f(x^*)$ does.
\end{proof}

\begin{corollary}[Stochastic mirror descent avoids strict saddles]
\label{cor:smd-avoids-strict-saddles}
Suppose
\cref{assumption:smd-regularity,assumption:sample-losses}
hold and the step-sizes satisfy
\[
\alpha_k\in(0,1/\sL), \qquad 
{\sum}_{k=0}^\infty \alpha_k=\infty, \qquad \text{and} \qquad {\sum}_{k=0}^{\infty}r_{k+1}^{q/2}<\infty \quad \text{with}\quad r_k:={\sum}_{i=k}^{\infty}\alpha_i^2.
\]
If $z^0=\nabla h(x^0)$ is drawn from any distribution with a density, then
\[
\Prob\,(
{\lim}_{k\to\infty}~x^k
\in \cX^*
)=0.
\]
\end{corollary}

\begin{proof}
The mean field $F(z)=\nabla f(\nabla h^*(z))$ is $\cC^1$ on $\Rd$. By
\cref{prop:smd-sample-losses-imply-stochastic-oracle}, the mirror-coordinate oracle satisfies the equicontinuity and lipeomorphism conditions in
\cref{assumption:F}, as well as
\cref{assumption:stochastic-oracle}.
Since
$\alpha_k<1/\sL$ for all $k$, choose
$\tilde\alpha$ with $\sup_k\alpha_k\le\tilde\alpha<1/\sL$.
By \cref{lem:smd-strict-saddle-transfer}, strict saddles of $f$ correspond under
$z=\nabla h(x)$ to unstable zeros of $F$. Applying
\cref{thm:avoid-strict-saddles} to the recursion \eqref{eq:smd-dual-fixed-point} gives almost
sure avoidance in the dual space. Since $\nabla h$ and $\nabla h^*$ are inverse
maps, convergence
of $\{x^k\}_k$ to a strict saddle is equivalent to convergence of $\{z^k\}_k$ to the
corresponding unstable zero.
\end{proof}

\subsection{Normal map-based proximal stochastic gradient method}
\label{subsec:nsgd-application}
 
This subsection establishes avoidance results and convergence guarantees to a local minimizer for a normal map-based stochastic proximal-type method. For stochastic proximal methods, general convergence theory often only gives
subsequential stationarity rather than convergence of the whole sequence
\cite{LiMil22}. Full convergence can be obtained by imposing additional stringent conditions such as isolatedness of the stationary points, or by passing to a generic linear tilt of the objective \cite{DavisDrusvyatskiyJiang2025}. These routes do not directly answer whether the underlying algorithm --- applied to the original objective --- converges (to a local minimizer).

In this part, we investigate a normal map-based proximal stochastic gradient method. In contrast to the standard stochastic proximal gradient method, it comes with convergence and identification guarantees \cite{qiu2025normal}. Moreover, the concept of ``active strict saddles'' introduced in \cite{DavisDrusvyatskiy2022} is applicable to this methodology. Our goal is to provide a strict saddle avoidance result that does not require unit excitation, does not pass to a tilted objective, and does not impose isolatedness conditions. Combined with the Kurdyka-{\L}ojasiewicz-based iterate convergence theory developed in \cite{qiu2025normal}, this yields convergence to local minimizers of the original composite objective under the active strict saddle property.

We now consider composite-type problems of the form
\begin{equation} \label{eq:composite} {\min}_{x\in\Rd}\;\psi(x):=f(x)+\vp(x).\end{equation}
We make the following assumptions on the objective function $\psi$. 
\begin{assumption}
\label{assumption:normal-map-regularity}
The function $f:\Rd\to\R$ is $\cC^2$, and $\nabla f$ is $\sL$-Lipschitz on $\Rd$. The function $\vp:\Rd\to(-\infty,\infty]$ is proper, closed, and convex. Moreover, each sample loss has $\sL$-Lipschitz gradient:
\[
\|\nabla f(x;\xi)-\nabla f(y;\xi)\|
\le
\sL\|x-y\|
\qquad
\forall\,x,y\in\Rd,\quad \forall\,\xi\in\Xi.
\]
\end{assumption}

Lipschitz smoothness of $f$ is a standard assumption for composite problems and proximal methods; see, e.g., \cite{AttBolSva13,DavisDrusvyatskiy2022,qiu2025normal}. Fix $\lambda>0$ and let $\prox{\lambda\vp}(z) := \argmin_{y \in \Rd}\,\vp(y)+\frac{1}{2\lambda}\|z-y\|^2$ denote the standard proximity operator of $\vp$. We then define the so-called normal map, \cite{robinson1992normal}, associated with problem \eqref{eq:composite} via
\[
F_\lambda(z):=\nabla f(\prox{\lambda\vp}(z))+\lambda^{-1}(z-\prox{\lambda\vp}(z)).
\]
In the following, $\partial \psi = \nabla f + \partial \vp$ denotes the limiting subdifferential of $\psi$, and $\partial\vp$ denotes the standard subdifferential for convex functions, cf.\ \cite{rocwet98}. A point $\bar x \in \mathrm{dom}(\vp) := \{x: \vp(x) < \infty\}$ is a stationary point of \eqref{eq:composite} if $0 \in \partial \psi(\bar x)$. The normal map is a valid stationarity measure for the composite problem \eqref{eq:composite} in the following sense: If $F_\lambda(\bar z) = 0$, then $0 \in \partial \psi(\bar x)$ where $\bar x := \prox{\lambda\vp}(\bar z)$, i.e., $\bar x$ is a stationary point of $\psi$. Conversely, if $0 \in \partial \psi(\bar x)$, then $\bar z := \bar x - \lambda \nabla f(\bar x)$ satisfies $F_\lambda(\bar z) = 0$ (for all $\lambda > 0$), cf.\ \cite{ouyang2021trust}.

In \cref{alg:nsgd-iteration}, we provide an overview of the normal map-based proximal stochastic gradient method proposed in \cite{qiu2025normal}. The key update of \cref{alg:nsgd-iteration} can be represented as follows
\begin{equation}\label{eq:nsgd-normal-map}
z^{k+1}
=
z^k-\alpha_kG_\lambda(z^k;\xi^k)
\quad \text{and} \quad
G_\lambda(z;\xi):=\nabla f(\prox{\lambda\vp}(z);\xi)+\lambda^{-1}(z-\prox{\lambda\vp}(z)),
\end{equation}
and $x^k:=\prox{\lambda\vp}(z^k)$. In particular, the step \eqref{eq:nsgd-normal-map} corresponds to \eqref{eq:stochastic-fixed-point-method} with $F=F_\lambda$.

\begin{algorithm}[t]
\caption{Normal map-based proximal stochastic gradient}
\label{alg:nsgd-iteration}
\begin{algorithmic}
    \Require Choose initial points $z^0\in\Rd$ and $x^0=\prox{\lambda\vp}(z^0)$, $\lambda>0$, and
    step-sizes $\{\alpha_k\}_{k\ge0}$.
    \For{$k=0,1,2,\ldots$}
        \State Draw a sample $\xi^k$ and set
\[ z^{k+1} = z^k-\alpha_k [ \nabla f(x^k;\xi^k)+\lambda^{-1}(z^k-x^k)] \quad \text{and} \quad x^{k+1} = \prox{\lambda\vp}(z^{k+1}).
        \]
    \EndFor
\end{algorithmic}
\end{algorithm}

\begin{lemma}[Normal map properties and update lipeomorphism]
\label{lem:normal-map-bounds}
\label{lem:nsgd-sample-map-lipeomorphism}
Suppose \cref{assumption:normal-map-regularity} holds. Then
$F_\lambda$ is Lipschitz with modulus $\sL+\lambda^{-1}$, and
\[
\dist(0,\partial\psi(\prox{\lambda\vp}(z)))\le \|F_\lambda(z)\|
\qquad\forall\,z\in\Rd.
\]
Moreover, for every
$\xi\in\Xi$ and every $0<\alpha<(\sL+\lambda^{-1})^{-1}$, the map $z \mapsto T_{\alpha,\xi}(z):=z-\alpha G_\lambda(z;\xi)$ has a well-defined inverse on $\Rd$, and both $T_{\alpha,\xi}$ and $T_{\alpha,\xi}^{-1}$ are globally Lipschitz. 
\end{lemma}

\begin{proof}
The proximity operator $\prox{\lambda\vp}$ is firmly nonexpansive. Hence, both $\prox{\lambda\vp}$
and $I-\prox{\lambda\vp}$ are nonexpansive \cite{bauschke2011convex}. Therefore, for
all $z,w\in\Rd$, we have
\[
\|F_\lambda(z)-F_\lambda(w)\|
\le
(\sL+\lambda^{-1})\|z-w\|.
\]
If $x=\prox{\lambda\vp}(z)$, then $\lambda^{-1}(z-x)\in\partial\vp(x)$. Hence
$F_\lambda(z)\in\nabla f(x)+\partial\vp(x)=\partial\psi(x)$, which gives the
distance bound. Similarly, for all $z,w\in\Rd$ and $\xi\in\Xi$, we obtain
\[
\begin{aligned}
\|G_\lambda(z;\xi)-G_\lambda(w;\xi)\|
&\le
\|\nabla f(\prox{\lambda\vp}(z);\xi)-\nabla f(\prox{\lambda\vp}(w);\xi)\|
\\ & \hspace{4ex}+\lambda^{-1}\|(I-\prox{\lambda\vp})z-(I-\prox{\lambda\vp})w\| \le
(\sL+\lambda^{-1})\|z-w\|.
\end{aligned}
\]
Thus $G_\lambda(\cdot;\xi)$ is $(\sL+\lambda^{-1})$-Lipschitz, and
\cref{lem:invertible} proves the claimed lipeomorphism property.
\end{proof}

Based on \cite{DavisDrusvyatskiy2022}, we now introduce the notion of active strict saddle points which will be nonsmooth analogues of the strict saddle points considered earlier in the smooth setting. Due to the inherent nonsmoothness of the model \eqref{eq:composite}, we will use second subderivatives to substitute Hessian information. The second subderivative
$\mathrm{d}^2\psi(x|v)(h)$ of $\psi$ at $x$ relative to $v$ in the direction $h \in \Rd$ is given by
\[ \mathrm{d}^2\psi(x|v)(h) := \liminf_{t \downarrow 0, \tilde h \to h}\, \Delta^2_t\!\;\psi(x|v)(\tilde h), \quad \Delta^2_t\!\;\psi(x|v)(h) := \frac{\psi(x+th)-\psi(x)-t \cdot \iprod{v}{h}}{\frac12 t^2}. \]
Using the $\mathcal C^2$-smoothness of $f$, we may also compute $\mathrm{d}^2\psi(x|v)(h) = \mathrm{d}^2\vp(x|v-\nabla f(x))(h) + \langle{h},{\nabla^2 f(x)h}\rangle$, $h \in \Rd$. We refer to \cite[Chapter 13]{rocwet98} for more details on the concept of second subderivatives. 

\begin{definition}[Active strict saddles]\label{def:active strict saddle}
A set $\cM\subset\Rd$ is an \emph{active manifold} for $\psi$ at $\bar x$ if, in a neighborhood $\cU$ of $\bar x$, the set $\cM\cap \cU$ is a $\cC^2$-manifold, $\psi|_{\cM\cap \cU}$ is $\cC^2$, and ${\inf}_{x\in \cU\setminus\cM}~\dist(0,\partial\psi(x)) >0$. The point $\bar x$ is an \emph{active strict saddle} if $0 \in \partial\psi(\bar x)$, it admits an active manifold $\cM$, and
\[ \min_{\|h\|=1}~{\mathrm{d}}^2\psi(\bar x|0)(h) < 0. \] 
We use $\cX^*$ to denote the set of all active strict saddles of $\psi$. We say that $\psi$ satisfies the \emph{active strict saddle property}, if every stationary point of $\psi$ is either a local minimizer or an active strict saddle.
\end{definition}

The active manifold is the smooth surface identified by the problem (near a point of interest). On such a surface, the restriction of $\psi$ is $\cC^2$. The sharpness condition $\inf_{x\in \cU\setminus\cM}\dist(0,\partial\psi(x))>0$ says that nearly stationary points must lie on the active manifold. Thus, an active strict saddle is a critical point with a negative tangential curvature direction. When $\cM=\Rd$, this reduces to the usual smooth strict saddle condition. The notion of active strict saddles was introduced by Davis and Drusvyatskiy in \cite{DavisDrusvyatskiy2022} to establish avoidance results in the nonsmooth setting. Active manifolds are also known as identifiable manifolds and have strong connections to identifiability and partial smoothness \cite{lewis2002active,hare2004identifying,drusvyatskiy2014optimality,lewis2025identify}.

\begin{lemma}[Derivative of the proximal map]
\label{lem:active-prox-derivative}
Let \cref{assumption:normal-map-regularity} hold. Fix $z^*\in\Rd$ with $F_\lambda(z^*)=0$, and set $x^*:=\prox{\lambda\vp}(z^*)$. Suppose that $\cM$ is an active manifold for $\psi$ at $x^*$. Then $\prox{\lambda\vp}$ is $\cC^1$ near $z^*$ and there are a linear subspace $S \subset \Rd$ and a positive semidefinite, symmetric matrix $Q \in \R^{d \times d}$ with $\mathrm{range}(Q) \subset S$ such that for all $h \in \Rd$, we have
\begin{equation}\label{eq:active-prox-derivative}
{\mathrm{d}}^2\psi(x^*|0)(h) = \iprod{h}{[\nabla^2 f(x^*)+Q]h} + \iota_S(h) \quad \text{and} \quad {\rm D}\prox{\lambda\vp}(z^*)h =
 P_S(I+\lambda Q)^{-1}P_Sh.
\end{equation}
Here, $P_S$ and $\iota_S$ denote the orthogonal projector onto $S$ and the indicator function of $S$, respectively.
\end{lemma}

\begin{proof} By \cref{lem:normal-map-bounds}, $x^*$ is a stationary point of $\psi$. The active manifold $\mathcal M$ is identifiable at $x^*$ for $0 \in \partial \psi(x^*)$, \cite{drusvyatskiy2014optimality,DavisDrusvyatskiy2022}. Thus, applying \cite[Proposition 10.12]{drusvyatskiy2014optimality} or \cite[Theorem 8.5]{lewis2025identify}, we can infer that $\psi$ is $\mathcal C^2$-partly smooth at $x^*$ for zero relative to $\mathcal M$ with $0 \in \mathrm{ri}(\partial\psi(x^*))$. In particular, $\vp$ is $\mathcal C^2$-partly smooth at $x^*$ for $-\nabla f(x^*)$ relative to $\mathcal M$ with $-\nabla f(x^*) \in \mathrm{ri}(\partial\vp(x^*))$, \cite[Corollary 4.7]{lewis2002active}.
Hence, by \cite[Theorem 28]{DanHarMal06}, the proximal mapping $\prox{\lambda\vp}$ is $\mathcal C^1$ in a neighborhood of $x^*-\lambda\nabla f(x^*) = z^*-\lambda F_\lambda(z^*) = z^*$; see also \cite[Lemma 4.3]{HuTiaPanWen23} or \cite[Corollary 3.19]{hang2026partly}. Applying a mild and straightforward generalization of \cite[Theorems 3.8 and 3.9]{poliquin1996generalized}, differentiability of $\prox{\lambda\vp}$ implies (among various other variational properties) that the second subderivative $\mathrm{d}^2\vp(x^*|-\nabla f(x^*))$ is generalized quadratic, i.e., 
\[ \mathrm{d}^2\vp(x^*|-\nabla f(x^*))(h) = \iprod{h}{Qh} + \iota_S(h) \quad \forall\,h \in \Rd, \]
for some linear subspace $S \subset \Rd$ and a symmetric matrix $Q \in \R^{d \times d}$; let us also refer to \cite[Proposition 2.2]{ouyang2024variational} for a complementary derivation. Without loss of generality, we may assume $\mathrm{range}(Q) \subset S$ (this can always be achieved by substituting $Q$ with $P_S Q P_S$). Moreover, using the convexity of $\vp$ and \cite[Theorem 13.20]{rocwet98}, it holds that $Q \succeq 0$.
Following the proof of \cite[Proposition 2.3]{ouyang2024variational} and combining \cite[Exercises 13.18, 13.35, and 13.45]{rocwet98}, we obtain
\[ {\rm D}\prox{\lambda\vp}(z^*)h=\argmin_{s\in\Rd}~\frac{1}{2}\mathrm{d}^2\vp(x^*|-\nabla f(x^*))(s)+\frac{1}{2\lambda}\|s-h\|^2=P_{S}(I+\lambda Q)^{-1}P_{S}h. \]
This finishes the proof.
\end{proof}

The subspace $S$ in \Cref{lem:active-prox-derivative} can be shown to coincide with the tangent space $T_{\cM}(x^*)$. Furthermore, using a local representation $\mathcal M \cap \cU = \{x \in \cU : c(x) = 0\}$, $Q$ can be connected to the Hessian of a Lagrangian involving $\vp|_{\cM\cap \cU}$ and the constraints $c(x) = 0$. We refer to \cite{hang2026partly} for additional details. 

\begin{lemma}[Stochastic oracle conditions]
\label{lem:nsgd-sample-losses-imply-stochastic-oracle}
Let \cref{assumption:normal-map-regularity,assumption:sample-losses}
(with $\cX=\Rd$) hold and assume that
\[
{\sum}_{k=0}^{\infty}\, r_{k+1}^{q/2}<\infty
\qquad \text{where}\qquad
r_k:={\sum}_{i=k}^{\infty}\alpha_i^2.
\]
Fix $\bar z\in\Rd$ with $F_\lambda(\bar z)=0$, set
$\bar x:=\prox{\lambda\vp}(\bar z)$, and suppose that $\bar x$ admits an active manifold $\cM$ for $\psi$. Then 
the local equicontinuity condition in \cref{assumption:F} holds at $\bar z$, and
\cref{assumption:stochastic-oracle} is satisfied at $\bar z$ for the pair
$(F_\lambda,G_\lambda)$.
\end{lemma}

\begin{proof}
By \cref{lem:active-prox-derivative}, $\prox{\lambda\vp}$ is $\cC^1$ near
$\bar z$. Hence, there is a neighborhood $\cU_{\bar z}$ on which every $G_\lambda(\cdot;\xi)$ is differentiable. The chain rule, the local Lipschitz continuity of the Hessian in \cref{assumption:sample-losses}, the uniform Hessian bound in \cref{assumption:normal-map-regularity}, and the continuity of ${\rm D}\prox{\lambda\vp}$ imply that $\{{\rm D}G_\lambda(\cdot;\xi)\}_{\xi\in\Xi}$ is equicontinuous on compact subsets of $\cU_{\bar z}$. Thus, the local equicontinuity condition holds.
Set
\[
b_k:=G_\lambda(\bar z;\xi^k)-F_\lambda(\bar z)
=
\nabla f(\bar x;\xi^k)-\nabla f(\bar x).
\]
The chain rule gives $
{\rm D}G_\lambda(\bar z;\xi)
=
\nabla^2 f(\bar x;\xi){\rm D}\prox{\lambda\vp}(\bar z)
+\lambda^{-1}(I-{\rm D}\prox{\lambda\vp}(\bar z))$,
and similarly, $
{\rm D}F_\lambda(\bar z)
=
\nabla^2 f(\bar x){\rm D}\prox{\lambda\vp}(\bar z)
+\lambda^{-1}(I-{\rm D}\prox{\lambda\vp}(\bar z))$. 
Thus, 
\[
J_k:={\rm D}G_\lambda(\bar z;\xi^k)-{\rm D}F_\lambda(\bar z)
=
\bigl(\nabla^2 f(\bar x;\xi^k)-\nabla^2 f(\bar x)\bigr)
{\rm D}\prox{\lambda\vp}(\bar z).
\]
By \cref{assumption:sample-losses}, we have $\Exp[b_k]=0$,
$\Exp[\|b_k\|^2]<\infty$, $\Exp[J_k]=0$, and
$\Exp[\|J_k\|_F^q]<\infty$ for some $q>2$. Together with the stated step-size condition, these bounds verify the hypotheses of
\cref{lem:iid-moment-bounds-imply-stochastic-oracle}. Hence,
\cref{assumption:stochastic-oracle} holds at $\bar z$.
\end{proof}

\begin{proposition}[Active saddles induce unstable zeros]
\label{prop:normal-map-active-saddle-unstable}
Let \cref{assumption:normal-map-regularity} hold and let $z^*$ be given with $F_\lambda(z^*)=0$. Set $x^*:=\prox{\lambda\vp}(z^*)$. If $x^*\in\cX^*$, then ${\rm D}F_\lambda(z^*)$ has a negative eigenvalue.
\end{proposition}

\begin{proof}
According to \Cref{lem:active-prox-derivative}, there exist a linear subspace $S\subset \Rd$ and a positive semidefinite matrix $Q$ with $\mathrm{range}(Q) \subset S$ such that ${\mathrm d}^2\psi(x^*|0)(\cdot) = \iprod{\cdot}{[\nabla^2 f(x^*) + Q]\cdot} + \iota_S(\cdot)$ and ${\rm D}\prox{\lambda\vp}(z^*) = P_SBP_S$, where $B := (I+\lambda Q)^{-1} \succ 0$. Moreover, by assumption, there is $h \in S$ with $\langle{h},{[\nabla^2 f(x^*) + Q]h}\rangle < 0$. Using $I = P_S + P_{S^\perp}$, $P_S^2 = P_S$, $(I-B)B^{-1} = \lambda Q$, and $P_SQ = Q = QP_S$, we can write
\begin{equation} \label{eq:represent-df} {\rm D}F_\lambda(z^*) = \nabla^2 f(x^*) P_SBP_S + \lambda^{-1} (I-P_SBP_S) = [\nabla^2 f(x^*) + Q]P_SBP_S + \lambda^{-1} P_{S^\perp}. \end{equation}
Let $r=\dim(S)$, and let $U \in \R^{d \times r}$ and $V \in \R^{d \times (d-r)}$ be orthonormal bases for $S$ and $S^\perp$, respectively. Set $W = [U,V]$. Then, $P_S = UU^\top$ and $P_{S^\perp} = VV^\top$, and using $Q = Q^\top \succeq 0$ and $\mathrm{range}({Q}) \subset S$, we obtain
\[ W^\top B W = \begin{bmatrix} B_S & 0 \\ 0 & I \end{bmatrix} \quad \text{and} \quad U^\top B U = B_S, \quad \text{where} \quad B_S := (I+\lambda U^\top Q U)^{-1} \succ 0.  \]
Defining $C_S := U^\top[\nabla^2 f(x^*) + Q]U$, and using \eqref{eq:represent-df}, we can infer
\[ W^\top {\rm D}F_\lambda(z^*) W = W^\top [\nabla^2 f(x^*) + Q]U [B_S, 0] +  \begin{bmatrix} 0 & 0 \\ 0 & \lambda^{-1} I \end{bmatrix} = \begin{bmatrix} C_S B_S & 0 \\ V^\top [\nabla^2 f(x^*) + Q]U B_S & \lambda^{-1} I \end{bmatrix}. \]
This implies $\operatorname{spec}({\rm D}F_\lambda(z^*)) = \operatorname{spec}(C_SB_S) \cup \operatorname{spec}(\lambda^{-1}I)$. Furthermore, thanks to $B_S \succ 0$, the matrices $C_SB_S$ and $B_S^{1/2} C_SB_S^{1/2}$ are similar to each other and have the same eigenvalues. Since $h \in S$, there exists $y$ such that $h = Uy$. Setting $\tilde y := B_S^{-1/2} y$, we obtain
\[ \langle{\tilde y},{B_S^{1/2} C_SB_S^{1/2}\tilde y}\rangle = \langle{y},{C_Sy}\rangle = \langle{h},{[\nabla^2f(x^*)+Q]h}\rangle < 0. \]
Thus, the symmetric matrix $B_S^{1/2} C_SB_S^{1/2}$ has a negative eigenvalue which finishes the proof.
\end{proof}

Next, we present the main avoidance result for the normal map-based proximal stochastic gradient method (NSGD). Its proof relies on the almost sure convergence of $F_\lambda(z^k)$, which was established in \cite[Theorem~3.6]{qiu2025normal} under a different stochastic assumption. Below, we impose a martingale difference condition with bounded variance and lower boundedness of $\psi$, to satisfy the requirements in \cite{qiu2025normal}.
 
\begin{assumption}\label{assumption:NSGD bounded variance}
Denote $\Exp_k[\cdot]:=\Exp[\cdot\mid\cF_k]$, where $\cF_k=\sigma(z^0)\vee\cF_k^\xi$ is defined in \cref{def:probability-model}. Assume that, for every $k\in\N$, $\Exp_k[\nabla f(x^k;\xi^k)]=\nabla f(x^k)$ and that there is $\sigma\geq0$ such that $\Exp_k[\|\nabla f(x^k;\xi^k)-\nabla f(x^k)\|^2]\leq\sigma^2$ almost surely. Furthermore, we assume that both $f$ and $\vp$ are bounded from below.
\end{assumption}

\begin{corollary}[NSGD avoids active strict saddles]
\label{cor:nsgd-avoid-active-saddles}
Let \cref{assumption:normal-map-regularity,assumption:sample-losses,assumption:NSGD bounded variance} be satisfied
(with $\cX=\Rd$). Let $\{x^k\}_k$ be generated by \cref{alg:nsgd-iteration} using the step-sizes:
\begin{equation} \label{eq:step-sizes-nsgd-0}
\alpha_k\in\left(0,({\sL+\lambda^{-1}})^{-1}\right), \quad
{\sum}_{k=0}^\infty \alpha_k=\infty, \quad \text{and}\quad 
{\sum}_{k=0}^{\infty}r_{k+1}^{q/2}<\infty
\quad \text{where}\quad
r_k:={\sum}_{i=k}^{\infty}\alpha_i^2.
\end{equation}
If $z^0$ is drawn from any distribution
with a density and we set $x^0=\prox{\lambda\vp}(z^0)$, then it holds that
\[
\Prob({\lim}_{k\to\infty}~x^k\in\cX^*)=0.
\]
\end{corollary}

\begin{proof}
We first define the saddle point set for the normal map variable $z$ as follows
\[
\cZ_\lambda^*:=\{z:F_\lambda(z)=0,\ \prox{\lambda\vp}(z)\in\cX^*\}.
\]
For every $z\in\cZ_\lambda^*$, the point $\prox{\lambda\vp}(z)$ admits an active
manifold by definition of $\cX^*$. Hence,
\cref{lem:nsgd-sample-losses-imply-stochastic-oracle} verifies local
equicontinuity and \cref{assumption:stochastic-oracle} at $z$.
Moreover, \cref{lem:active-prox-derivative} shows that $F_\lambda$ is $\cC^1$ near $z$,
\cref{lem:nsgd-sample-map-lipeomorphism} yields the lipeomorphism condition,
and \cref{prop:normal-map-active-saddle-unstable} shows that $z$ is an unstable zero.
Hence, all hypotheses used in the proof of
\cref{thm:avoid-strict-saddles} hold at every point of $\cZ_\lambda^*$, and we can infer 
\[
\Prob ({\lim}_{k\to\infty}~z^k\in\cZ_\lambda^*)=0.
\]
Moreover, as mentioned, the assumptions in \cite[Theorem 3.6]{qiu2025normal} are satisfied, and it follows that $F_\lambda(z^k)\to0$ almost surely. Next, suppose $x^k=\prox{\lambda\vp}(z^k)\to x^*\in\cX^*$ and let us set $z^*:=x^*-\lambda\nabla f(x^*)$. Due to $0 \in \partial \psi(x^*)$, we have $x^*=\prox{\lambda\vp}(z^*)$ and $F_\lambda(z^*)=0$. This implies $z^k = x^k-\lambda\nabla f(x^k)+\lambda F_\lambda(z^k) \to z^*$ (almost surely).
Thus, on the almost sure event $\{F_\lambda(z^k)\to0\}$, we obtain
\[
\{{\lim}_{k\to\infty}~\prox{\lambda\vp}(z^k)\in\cX^*\}
\subset
\{{\lim}_{k\to\infty}~z^k\in\cZ_\lambda^*\}.
\]
The right-hand event has probability zero, and the claim follows.
\end{proof}

\cref{cor:nsgd-avoid-active-saddles} proves saddle avoidance for NSGD\@. To turn this into convergence to local minimizers, we use the Kurdyka-{\L}ojasiewicz (KL)-based iterate convergence result from \cite[Theorem~4.2]{qiu2025normal}. Once the iterates are known to converge to a stationary point, \cref{cor:nsgd-avoid-active-saddles} rules out active strict saddles and leaves only local minimizers (under the active strict saddle property). 

The result in \cite[Theorem~4.2]{qiu2025normal} requires additional and different conditions on the geometry of $\psi$ and on the step-sizes. In particular, we additionally assume that the objective function $\psi$ is definable (in an $o$-minimal structure). Definability is a standard and mild geometric assumption covering many practical and important function classes, such as semi-algebraic and globally subanalytic functions; see \cite{van1998tame,kur98,AttBolSva13} for more background and formal definitions. Concerning step-sizes, we will work with the following additional condition:
\begin{equation}\label{eq:step-sizes for convergence}
    {\sum}_{k=0}^\infty
    \alpha_k^2
    \Big({\sum}_{i=0}^k \alpha_i\Big)^\varrho
    <\infty\qquad \text{for some}\quad \varrho>1.
\end{equation}
A standard step-size choice ensuring both iterate convergence and saddle avoidance is, e.g.,
\begin{equation*}
     0<\alpha<\frac{1}{\sL+\lambda^{-1}},
    \qquad
    \alpha_k=\frac{\alpha}{(k+1)^{p}},
    \qquad
    p\in\Big(
        \max\Big\{\frac{2}{3},\frac{1}{2}+\frac{1}{q}\Big\},
        1
    \Big].
\end{equation*}
Indeed, $p\leq 1$ implies $\sum_{k=0}^\infty \alpha_k=\infty$ and the condition $p>\frac{1}{2}+\frac{1}{q}$ ensures $
        {\sum}_{k=0}^{\infty}
        \left({\sum}_{i=k}^{\infty}\alpha_i^2\right)^{q/2}
        <\infty$,
while $p>\frac{2}{3}$ guarantees \eqref{eq:step-sizes for convergence}.
\begin{corollary}[NSGD converges to local minimizers]
\label{cor:nsgd-converges-local-minimizers}
We consider the setting in \cref{cor:nsgd-avoid-active-saddles}. Let the sequence $\{x^k\}_k$ be generated by \cref{alg:nsgd-iteration} with step-sizes satisfying \eqref{eq:step-sizes-nsgd-0} and \eqref{eq:step-sizes for convergence}. Furthermore, suppose $\psi$ is definable and $\{x^k\}_k$ is bounded almost surely.
Then, the sequence $\{x^k\}_k$ converges to a local minimizer of $\psi$ almost surely --- provided that the active strict saddle property holds.
\end{corollary}

\begin{proof}
Under definability of $\psi$ and boundedness of $\{x^k\}_k$, \cite[Theorem~4.2]{qiu2025normal}
shows that the sequence $\{x^k\}_k$ converges to a stationary point $x^*$ of
$\psi$ almost surely. On this event,
\cref{cor:nsgd-avoid-active-saddles} rules out $x^*\in\cX^*$. Since every
stationary point that is not a local minimizer is an active strict saddle, the
limit $x^*$ must be a local minimizer of $\psi$.
\end{proof}

\section{Application II: Random reshuffling}
\label{sec:application-rr}

Random reshuffling (RR) is a basic without-replacement implementation of stochastic methods for finite-sum problems. At the beginning of each epoch, the data are permuted and then processed once. This sampling rule is common in large-scale training, but it differs from the classical stochastic approximation methods. 

In this section, we show that RR fits our framework for stochastic recursions. In particular, we prove almost sure avoidance of strict saddles for the unmodified RR\@. The result does not require the local RR escape condition in \cite{Beneventano2024TrajectoriesSGDWithoutReplacement}, which assumes that the bias generated by reshuffling has a nonzero projection onto an escaping direction at the saddle.

Consider the finite-sum problem
\[
    {\min}_{x\in\Rd}~f(x):=\frac1n{\sum}_{i=1}^n \, f_i(x).
\]
The random reshuffling iteration is given in \cref{alg:rr-iteration}.

\begin{algorithm}[h]
\caption{Random reshuffling}
\label{alg:rr-iteration}
\begin{algorithmic}
    \Require Choose an initial point $x^0$ and step-sizes $\{\eta_k\}_{k\ge0}$.
    \For{$k=0,1,2,\ldots$}
        \State Draw a permutation $\pi^k$ of $\{1,\ldots,n\}$ and set $y_0^k= x^k$.
        \For{$i=1,\ldots,n$}
            \State $y_i^k=
            y_{i-1}^k-\eta_k\nabla f_{\pi_i^k}(y_{i-1}^k)$.
        \EndFor
        \State Set $x^{k+1}= y_n^k$.
    \EndFor
\end{algorithmic}
\end{algorithm}

We now write one RR epoch as one sampled update step of the stochastic recursion, with mean field
$F=\nabla f$. Let $\xi=(\eta,\pi)$
denote the epoch seed. Starting from $x$, set $y_0(x;\xi)=x$ and
\[
    y_i(x;\xi)
    =
    y_{i-1}(x;\xi)-\eta\nabla f_{\pi_i}(y_{i-1}(x;\xi)),
    \qquad i=1,\ldots,n.
\]
Then, with $\alpha_k=n\eta_k$ and $\xi^k=(\eta_k,\pi^k)$,
\cref{alg:rr-iteration} becomes
\begin{equation}\label{eq:rr-fixed-point-recursion}
    x^{k+1}
    =
    x^k-\alpha_kG(x^k;\xi^k)\qquad \text{where}\qquad G(x;\xi):=
    \frac1n{\sum}_{i=1}^n\nabla f_{\pi_i}(y_{i-1}(x;\xi)).
\end{equation}

We impose the following regularity condition on the component functions.

\begin{assumption}[Regularity]
\label{assumption:rr-regularity}
Each $f_i:\Rd\to\R$ is $\cC^2$, $\nabla f_i$ is globally $\sL$-Lipschitz
with a common modulus $\sL>0$, and the Hessians $\nabla^2 f_i$, $i \in \{1,\dots,n\}$, are locally Lipschitz.
\end{assumption}

\cref{assumption:rr-regularity} is standard in the analysis of random
reshuffling; see, e.g.,
\cite{nguyen2020unified,li2021convergence,qiu2023new,josz2024proximal}. It ensures the smoothness of the epoch map and the local controls needed by our framework for stochastic recursions. The proposition below verifies the corresponding stochastic oracle condition for $G$.

\begin{proposition}[Random reshuffling oracle]
\label{prop:rr-oracle}
Suppose \cref{assumption:rr-regularity} holds and fix
$0<\bar\eta\le(4n\sL)^{-1}$. For epoch seeds $\xi=(\eta,\pi)$ with
$0<\eta\le\bar\eta$, the oracle $G$ satisfies \cref{assumption:F}.
Moreover, if
\[
    \eta_k\in(0,\bar\eta],\qquad
    \alpha_k=n\eta_k,\qquad
    {\sum}_{k=0}^\infty \eta_k^2<\infty,
\]
then \cref{assumption:stochastic-oracle} holds for the seed sequence
$\xi^k=(\eta_k,\pi^k)$, for every sequence of permutations.
\end{proposition}

\begin{proof}
The form \eqref{eq:rr-fixed-point-recursion} and
\cref{assumption:stochastic-oracle} follow from
\cref{prop:without-replacement-implies-stochastic-oracle} with
$G_i=\nabla f_i$. In addition, \cref{prop:without-replacement-implies-stochastic-oracle} establishes the equicontinuity of ${\rm D}G(\cdot;\xi)$ for all epoch seeds $\xi=(\eta,\pi)$ with $0<\eta\le\bar\eta$. Since
$F=\nabla f$ and each $f_i$ is $\cC^2$, the mean field is $\cC^1$.

It remains to verify the lipeomorphism condition. For any
$\xi=(\eta,\pi)$ with $0<\eta\le\bar\eta$, the inner maps satisfy
\[
    \|y_i(z;\xi)-y_i(w;\xi)\|
    \le
    (1+\eta \sL)^i\|z-w\|,
    \qquad \forall\, i=0,\ldots,n.
\]
Hence, $G(\cdot;\xi)$ is Lipschitz continuous, uniformly over all such
$\xi$. That is,
\[
    \|G(z;\xi)-G(w;\xi)\|
    \le
    \widetilde \sL\|z-w\|\qquad \text{where}
    \qquad
    \widetilde \sL
    :=
    \frac{\sL}{n}{\sum}_{i=0}^{n-1}(1+\bar\eta \sL)^i < 2\sL.
\]
The bound $\widetilde \sL<2\sL$ holds because $\bar\eta \sL\le(4n)^{-1}$. 
By \cref{lem:invertible}, $
    z\mapsto z-\alpha G(z;\xi)$ 
is a lipeomorphism for every $0<\alpha<1/\widetilde \sL$. Thus, 
\cref{assumption:F} is verified.
\end{proof}

Due to the smoothness of the objective function, the strict saddle set is given by
\[
\cX^*:=\big\{x\in\Rd:\nabla f(x)=0,\ \lambda_{\min}(\nabla^2f(x))<0 \big\}
\]
in this application. We now present our main avoidance result for random reshuffling.

\begin{corollary}[Random reshuffling avoids strict saddles]
\label{cor:rr-avoids-strict-saddles}
Suppose \cref{assumption:rr-regularity} holds. Let $\{x^k\}_k$ be generated by \cref{alg:rr-iteration} with step-sizes $\{\eta_k\}_k$ satisfying
\[
0<\eta_k \leq \frac{1}{4n\sL},\qquad  {\sum}_{k=0}^\infty\, \eta_k=\infty
    \qquad\text{and}\qquad 
    {\sum}_{k=0}^\infty\, \eta_k^2<\infty.
\]
If $x^0$ is drawn from any distribution with a density, then $\Prob(\lim_{k\to\infty}x^k \in \cX^*) = 0$.
\end{corollary}

\begin{proof}
By \cref{prop:rr-oracle}, the epoch recursion
\eqref{eq:rr-fixed-point-recursion} satisfies \cref{assumption:F} and
\cref{assumption:stochastic-oracle}. For $F=\nabla f$,
every point satisfying
$\nabla f(x)=0$ and $\lambda_{\min}(\nabla^2 f(x))<0$ is an unstable zero.
The conclusion follows from \cref{thm:avoid-strict-saddles} with $z^k = x^k$ and $\cZ^*=\cX^*$.
\end{proof}

\begin{remark}[Existing avoidance results for random reshuffling]
Most convergence results for random reshuffling focus on first-order
stationarity, rates, last-iterate convergence, or full sequence convergence; see,
e.g.,
\cite{nguyen2020unified,li2021convergence,qiu2024random,josz2024proximal}.
For finite-time saddle avoidance, the work of Yu and Li
\cite{YuLi2023HighProbabilityRR} gives a high-probability complexity bound for reaching an $\varepsilon$-second-order stationary point with a perturbed RR scheme. This is in line with finite-time saddle escape guarantees for first-order methods based on injected perturbations
\cite{GeHuangJinYuan2015,JinGeNetrapalliKakadeJordan2017,JinNetrapalliGeKakadeJordan2019}.
Moreover, in \cite{Beneventano2024TrajectoriesSGDWithoutReplacement}, Beneventano proved a local saddle escape result. Near a given saddle, the analysis shows that the without-replacement bias can push the trajectory along an escaping direction. The key assumption is a nonzero projection condition. At a saddle $x$ and along an escaping direction $v$, the following assumption is made:
\[
    \frac1n{\sum}_{i=1}^n
    \langle v,\nabla^2 f_i(x)\nabla f_i(x)\rangle \ne 0.
\]
Since this condition ensures that the escape mechanism has a component
in an unstable direction, it plays a role analogous to unit excitation
\eqref{eq:assumption UE}.

Our result is an asymptotic avoidance statement for the unmodified RR\@. Relative to standard convergence analyses of RR \cite{nguyen2020unified,li2021convergence,qiu2024random,josz2024proximal}, the additional condition used here is \emph{local Lipschitz continuity} of the Hessians. No injected perturbations or unit excitation-type assumption is imposed.
\end{remark}

\begin{remark}[Composite problems]
The framework may also extend to composite finite-sum problems through the normal map-based proximal random reshuffling method \cite{qiu2023new}. The normal map-based analysis in \cref{prop:normal-map-active-saddle-unstable} shows that active strict saddles give rise to unstable zeros of the normal map field. Sampling without replacement has been verified to satisfy the required oracle conditions in \cref{prop:without-replacement-implies-stochastic-oracle}. These arguments give saddle avoidance for normal map-based RR\@. Combined with the iterate convergence guarantees in \cite[Theorem~21]{qiu2023new}, such an avoidance result would yield convergence to a local minimizer under the active strict saddle property.
\end{remark}

\section{Conclusion}
We develop a framework for proving that stochastic recursions almost surely avoid unstable zeros, without requiring unit excitation. Instead of using noise to push the iterates away, we show that only a null set of initial states can generate trajectories that remain near an unstable zero. The sampled maps need not have a common fixed point or a common linearization. We verify the required conditions for martingale-type sampling and sampling without replacement. Specifically, we apply the framework to establish avoidance guarantees for stochastic mirror descent, a normal map-based proximal stochastic gradient method, and random reshuffling.

We expect that the proposed framework can be used to study a broader range of algorithms, including distributed methods \cite{tsitsiklis1986distributed,nedic2009distributed}, stochastic momentum methods \cite{sutskever2013importance}, stochastic alternating methods \cite{driggs2021stochastic}, and stochastic block-coordinate methods \cite{dang2015stochastic}. Applying our framework to these algorithms may require some algorithm-specific adaptations, but the main proof strategy and techniques should remain highly similar. For distributed methods, one may first analyze the averaged iterate and then use consensus to transfer the saddle-avoidance result to the local iterates. For stochastic momentum methods, one may instead analyze an augmented state consisting of the iterate and the momentum variable. We leave the study of such additional applications to future research.

\section*{Declarations}

\paragraph{Funding.} Junyu Zhang was partly supported by the Singapore Ministry of Education under the AcRF Tier 2 Grant No.\ MOE-T2EP20125-0007. Andre Milzarek was partly supported by the National Natural Science Foundation of China under Grant No.\ W2532005 and by the Guangdong Provincial Key Laboratory of Mathematical Foundations for Artificial Intelligence (2023B1212010001).
\paragraph{Conflict of interest.} The authors declare that they have no conflict of interest.
\appendix
\section{Preparatory tools}

\begin{lemma}[Lipeomorphism induced by a Lipschitz map]
\label{lem:invertible}
Suppose that $F:\Rd\to\Rd$ is Lipschitz continuous with constant $\sL_F>0$.
For any $\alpha\in(0,1/\sL_F)$, define $T:=I-\alpha F$. Then, $T$ has a
well-defined inverse on $\Rd$, and both $T$ and $T^{-1}$ are globally
Lipschitz. Moreover,
\[
(1-\alpha \sL_F)\|x-z\|
\le
\|T(x)-T(z)\|
\le
(1+\alpha \sL_F)\|x-z\|
\qquad
\forall\,x,z\in\Rd.
\]
\end{lemma}

\begin{proof}
Let $x,z\in\Rd$. Then, we have
\[
T(x)-T(z)
=
x-z-\alpha\bigl(F(x)-F(z)\bigr).
\]
Hence, by the reverse triangle inequality and the Lipschitz continuity of $F$,
\[
\|T(x)-T(z)\|
\ge
\|x-z\|-\alpha\|F(x)-F(z)\|
\ge
(1-\alpha \sL_F)\|x-z\|.
\]
The upper bound follows similarly from the triangle inequality. Since
$1-\alpha \sL_F>0$, the lower bound implies injectivity. To prove surjectivity,
fix $y\in\Rd$. Solving $T(x)=y$ is to find a fixed point of
\[
\Psi_y(x):=y+\alpha F(x).
\]
The map $\Psi_y$ is a contraction because $\alpha \sL_F<1$, so Banach's fixed
point theorem gives a unique fixed point $x_y$. Then $T(x_y)=y$ and thus, $T$ is bijective. The lower bound gives
\[
\|T^{-1}(u)-T^{-1}(v)\|
\le
(1-\alpha \sL_F)^{-1}\|u-v\|
\qquad
\forall\,u,v\in\Rd,
\]
so $T^{-1}$ is Lipschitz. Since $T$ is also Lipschitz, the proof is complete.
\end{proof}
Next, we state a growth estimate for a generic Jordan block. We write a
Jordan block associated with an eigenvalue $\lambda\in\mathbb C$ as
\begin{equation}
    \label{eq:generic Jordan}
    J=\lambda I_m+N,
\end{equation}
where $I_m$ is the $m\times m$ identity matrix and
$N\in\R^{m\times m}$ is the nilpotent Jordan shift. In particular,
$N^m=0$; see \cite[Theorem~3.1.5]{HornJohnson2012}.

\begin{lemma}[Generic Jordan block estimate]
\label{lem:generic-jordan-block-transition}
Let $J$ be defined in \eqref{eq:generic Jordan}. Let 
$\{\alpha_\ell\}\subset\R_{++}$ satisfy $\alpha_\ell\to0$ and
$\alpha_\ell|\lambda|\le 1/2$ for all $\ell$. For $k\ge j$, set
$t_{k,j}:={\sum}_{\ell=j}^{k-1}\alpha_\ell$ and
$\Phi_J(k,j):={\prod}_{\ell=j}^{k-1}(I-\alpha_\ell J)$, with the convention
$\Phi_J(j,j)=I_m$. The following bounds hold.
\begin{itemize}
    \item If $\Re(\lambda)\ge0$, then for any $\kappa>0$ there exists
$D^+\ge1$ such that
\[
\|\Phi_J(k,j)\|\le D^+\exp(\kappa t_{k,j}),\qquad k\ge j.
\]
\item If $\Re(\lambda)<0$, then there exist $\nu>0$ and $D^-\ge1$ such that
\[
\|\Phi_J(k,j)^{-1}\|\le D^-\exp(-\nu t_{k,j}),\qquad k\ge j.
\]
\end{itemize}
\end{lemma}

\begin{proof}
Set $\beta_\ell:=\alpha_\ell/(1-\alpha_\ell\lambda)$.
Since $I-\alpha_\ell J=(1-\alpha_\ell\lambda)(I-\beta_\ell N)$, we have
\begin{equation}
    \label{eq:lem Jordan 0}
    \begin{aligned}
        \Phi_J(k,j)
&=\left[{\prod}_{\ell=j}^{k-1}(1-\alpha_\ell\lambda)\right]
{\prod}_{\ell=j}^{k-1}(I-\beta_\ell N)\qquad
\text{and} \\ 
\Phi_J(k,j)^{-1}
&=\left[{\prod}_{\ell=j}^{k-1}(1-\alpha_\ell\lambda)^{-1}\right]
{\prod}_{\ell=j}^{k-1}(I-\beta_\ell N)^{-1}.
    \end{aligned}
\end{equation}
Thus, the proof separates the scalar factors $|1-\alpha_\ell\lambda|$ from
the nilpotent products generated by $I-\beta_\ell N$. \\[-2mm]

\noindent\textbf{Step 1: Estimates of scalar factors.}
When $\Re(\lambda)\ge0$, it follows from $1+x \leq \exp(x)$ for $x\geq 0$ that
\begin{equation*}
|1-\alpha\lambda|^2
\le1+\alpha^2|\lambda|^2
\le \exp(\alpha^2|\lambda|^2) \qquad \Longrightarrow \qquad |1-\alpha\lambda|\le \exp(\alpha^2|\lambda|^2/2).    
\end{equation*}
Fix $\kappa>0$. Since
$\alpha_\ell\to0$, there exists $K_\kappa\ge0$ such that
$\alpha_\ell|\lambda|^2\le\kappa$ for all $\ell\ge K_\kappa$. Set
$B_\kappa:={\sum}_{\ell=0}^{K_\kappa-1}\alpha_\ell^2$, with
$B_\kappa=0$ if $K_\kappa=0$. Then, $
|\lambda|^2{\sum}_{\ell=j}^{k-1}\alpha_\ell^2
\le |\lambda|^2B_\kappa+\kappa t_{k,j}$. 
Hence, we obtain
\[
{\prod}_{\ell=j}^{k-1}|1-\alpha_\ell\lambda|
\le \exp\left(\frac{|\lambda|^2}{2}
{\sum}_{\ell=j}^{k-1}\alpha_\ell^2\right)
\le \exp\left(\frac{|\lambda|^2B_\kappa}{2}\right)
\exp\left(\frac{\kappa t_{k,j}}{2}\right).
\]
If $\Re(\lambda)<0$, set $\eta:=-\Re(\lambda)$ and choose
$\nu >0$ such that $2\nu <(\log2)\eta$. Since
$0\le\alpha \eta\le\alpha|\lambda|\le1/2$, concavity of $\log(1+y)$ on $[0,1]$
gives
\[
\log|1-\alpha\lambda|
\ge\frac12\log(1+2\alpha \eta)
\ge(\log2)\alpha \eta
\ge2\nu \alpha.
\]
Since 
$|1-\alpha_\ell\lambda| \geq 1- \alpha_\ell|\lambda|\ge1/2$, we have $|\beta_\ell|\le2\alpha_\ell$.
Consequently, for $k\ge j$,
\begin{equation}\label{eq:jordan-scalar-product-bound}
{\sum}_{\ell=j}^{k-1}|\beta_\ell|\le 2t_{k,j}\qquad \text{and} \qquad
\begin{cases}
{\prod}_{\ell=j}^{k-1}|1-\alpha_\ell\lambda|
\le \exp(|\lambda|^2B_\kappa/2)\exp(\kappa t_{k,j}/2),
& \Re(\lambda)\ge0,\\[2mm]
{\prod}_{\ell=j}^{k-1}|1-\alpha_\ell\lambda|^{-1}
\le \exp(-2\nu t_{k,j}),
& \Re(\lambda)<0.
\end{cases}
\end{equation}

\noindent\textbf{Step 2: Estimates of nilpotent products.} 
Since $N^m=0$, it holds that
\[
(I-\beta_\ell N){\sum}_{r=0}^{m-1}(\beta_\ell N)^r
=I-(\beta_\ell N)^m=I
\qquad\Longrightarrow\qquad
(I-\beta_\ell N)^{-1}={\sum}_{r=0}^{m-1}\beta_\ell^rN^r.
\]
Consider either ${\prod}_{\ell=j}^{k-1}(I-\beta_\ell N)$ or
${\prod}_{\ell=j}^{k-1}(I-\beta_\ell N)^{-1}$, and write it as
${\sum}_{q=0}^{m-1}c_qN^q$. Then $|c_0|=1$. For $1\le q<m$, the coefficient
$c_q$ is a sum of monomials in $\beta_j,\ldots,\beta_{k-1}$ of total degree
$q$, with repeated indices allowed in the inverse product. By the triangle
inequality,
\[
|c_q|
\le
{\sum}_{\ell_1=j}^{k-1}\cdots{\sum}_{\ell_q=j}^{k-1}
|\beta_{\ell_1}|\cdots|\beta_{\ell_q}|
=
\left({\sum}_{\ell=j}^{k-1}|\beta_\ell|\right)^q.
\]
Set $C_J:=\max_{0\le q<m}\|N^q\|$. Then
\[
\left\|{\sum}_{q=0}^{m-1}c_qN^q\right\|
\le {\sum}_{q=0}^{m-1}|c_q|\,\|N^q\|
\le C_J{\sum}_{q=0}^{m-1}
\left({\sum}_{\ell=j}^{k-1}|\beta_\ell|\right)^q
\le C_J\left(1+{\sum}_{\ell=j}^{k-1}|\beta_\ell|\right)^{m-1}.
\]
Applying this estimate to the forward and inverse products gives
\begin{equation}\label{eq:generic-jordan-nilpotent-factor}
\max\left\{
\left\|{\prod}_{\ell=j}^{k-1}(I-\beta_\ell N)\right\|,
\left\|{\prod}_{\ell=j}^{k-1}(I-\beta_\ell N)^{-1}\right\|
\right\}
\le C_J\left(1+{\sum}_{\ell=j}^{k-1}|\beta_\ell|\right)^{m-1}.
\end{equation}

\noindent\textbf{Step 3: Combining the scalar and nilpotent bounds.}
It follows from  
\eqref{eq:lem Jordan 0}--\eqref{eq:generic-jordan-nilpotent-factor} that
\begin{equation}\label{eq:generic-jordan-block-combined}
\begin{cases}
\|\Phi_J(k,j)\|
\le C_J\exp\left(\frac{|\lambda|^2B_\kappa}{2}\right)
(1+2t_{k,j})^{m-1}\exp(\kappa t_{k,j}/2),
& \Re(\lambda)\ge0,\\[2mm]
\|\Phi_J(k,j)^{-1}\|
\le C_J(1+2t_{k,j})^{m-1}\exp(-2\nu t_{k,j}),
& \Re(\lambda)<0.
\end{cases}
\end{equation}
Define
\[
D^+:=C_J\exp\Big(\tfrac{|\lambda|^2B_\kappa}{2}\Big)
\sup_{t\ge0}(1+2t)^{m-1}\exp(-\kappa t/2) \qquad \text{and} \qquad D^-:=C_J\sup_{t\ge0}(1+2t)^{m-1}\exp(-\nu t).
\]
Both constants are finite, since the exponential terms dominate the polynomial
factor. Hence, \eqref{eq:generic-jordan-block-combined} yields
$\|\Phi_J(k,j)\|\le D^+\exp(\kappa t_{k,j})$ when
$\Re(\lambda)\ge0$, and
$\|\Phi_J(k,j)^{-1}\|\le D^-\exp(-\nu t_{k,j})$ when
$\Re(\lambda)<0$.
\end{proof}

\section{Verification of the stochastic oracle condition}
\label{app:verification-stochastic-oracle}

\subsection{Martingale-type noise}
\label{app:martingale-type-noise}

\begin{proof}[Proof of \cref{lem:iid-moment-bounds-imply-stochastic-oracle}]
Since $b_k$ and $J_k$ are measurable functions of $\xi^k$, the sequence
$\{(b_k,J_k)\}_{k}$ is also i.i.d. 
Consider the partial sums
\[
B_n:={\sum}_{k=0}^{n}\alpha_k b_k, 
\]
which form a martingale with respect to the shifted natural filtration
$\{\cF_{n+1}^{\xi}\}_{n\ge0}$ in
\cref{def:probability-model}. 
By orthogonality of martingale differences, we have
\[
\Exp[\|B_n\|^2]
\le
{\sum}_{k=0}^{n}\alpha_k^2\Exp[\|b_0\|^2]
\le
\Exp[\|b_0\|^2]\cdot{\sum}_{k=0}^{\infty}\alpha_k^2<\infty.
\]
Here, ${\sum}_{k=0}^{\infty}\alpha_k^2<\infty$ follows from ${\sum}_{k=0}^{\infty}r_{k+1}^{q/2}<\infty$ and we use the fact that the errors $\{b_k\}_k$ are i.i.d., that is, $\Exp[\|b_k\|^2] = \Exp[\|b_0\|^2]$ for all $k$. Hence, the martingale convergence theorem gives $B_\infty$ such that $B_n\to B_\infty$ almost surely. Similarly, the process
\[
T_n:={\sum}_{i=0}^{n}\alpha_iJ_i,\qquad T_{-1}:=0,
\]
can be shown to be a martingale with respect to
$\{\cF_{n+1}^{\xi}\}_{n\ge0}$. More specifically, $\{T_n\}_n$ has uniformly bounded second moments as $q>2$ implies
$\Exp[\|J_0\|_F^2]<\infty$. Thus, $T_n$ converges almost surely to some $T_\infty$. Therefore, $S_k=T_\infty-T_{k-1}$ is well-defined and $S_k\to0$ almost surely. Define
\[
M_n:={\sum}_{i=1}^{n}\alpha_iJ_iB_{i-1}.
\]
Since $B_{n-1}$ is $\cF_n^\xi$-measurable, $J_n$ is independent of
$\cF_n^\xi$, and $\Exp[J_n]=0$, the process
$\{M_n\}_n$ is a martingale with respect to
$\{\cF_{n+1}^{\xi}\}_{n\ge0}$. Moreover,
\[
\Exp[\|M_n-M_{n-1}\|^2]
\le
\alpha_n^2 \Exp[\|J_0\|_F^2]\,\Exp[\|B_{n-1}\|^2]
\le
\alpha_n^2\Exp[\|J_0\|_F^2]\,\Exp[\|b_0\|^2] \cdot
{\sum}_{k=0}^{\infty}\alpha_k^2.
\]
Therefore, 
${\sum}_{n=1}^{\infty}\Exp\|M_n-M_{n-1}\|^2<\infty$, and $\{M_n\}_n$
converges almost surely. 

Using the identity $S_{k+1}={\sum}_{i=k+1}^{\infty}\alpha_iJ_i$, we can expand the term ${\sum}_{k=0}^{n}\alpha_kS_{k+1}b_k$ as follows:
\begin{equation}\label{eq:verify stochastic oracle 1}
	\begin{aligned}
	{\sum}_{k=0}^{n}\alpha_kS_{k+1}b_k &= {\sum}_{k=0}^{n}  {\sum}_{i=k+1}^{\infty}\alpha_k\alpha_i J_i b_k \\
	&= {\sum}_{k=0}^{n-1}{\sum}_{i=k+1}^{n}\alpha_k\alpha_iJ_i b_k+
{\sum}_{k=0}^{n}{\sum}_{i=n+1}^{\infty}\alpha_k\alpha_iJ_i b_k =M_n+S_{n+1}B_n,
\end{aligned}
\end{equation}
where the last equality follows from
\[
M_n
=
{\sum}_{i=1}^{n}\alpha_iJ_iB_{i-1}
=
{\sum}_{i=1}^{n} \alpha_iJ_i {\sum}_{k=0}^{i-1}\alpha_k b_k
=
{\sum}_{k=0}^{n-1}{\sum}_{i=k+1}^{n}\alpha_k\alpha_iJ_i b_k.
\]
Since the sequence $\{M_n\}_n$ converges, $S_{n+1}\to0$, and $B_n\to B_\infty$ almost
surely, \eqref{eq:verify stochastic oracle 1} implies that the process
$\{{\sum}_{k=0}^{n}\alpha_kS_{k+1}b_k\}_n$ converges almost surely.

It remains to prove $S_{k+1}J_k\to0$. For $N>k+1$, set
\[
R_{k,N}:={\sum}_{i=k+1}^{N}\alpha_iJ_i.
\]
For fixed $k$, the stopped sums
$\{R_{k,\ell}\}_{\ell=k+1}^{N}$ form a martingale with
respect to $\{\cF_{\ell+1}^{\xi}\}_{\ell=k+1}^{N}$. By the Burkholder--Davis--Gundy inequality
\cite{burdavgun72}, applied to the Frobenius vectorization, there exist
$C_q, C_q'<\infty$ independent of $k$ and $N$ such that
\begin{align*}
\bigl(\Exp[\|R_{k,N}\|_F^q]\bigr)^{1/q}
&\le
C_q
\left(
\Exp\left[\left({\sum}_{i=k+1}^{N}\alpha_i^2\|J_i\|_F^2\right)^{q/2}\right]
\right)^{1/q}\\
&\le
C_q
\left(
{\sum}_{i=k+1}^{N}
\alpha_i^2\bigl(\Exp[\|J_i\|_F^q]\bigr)^{2/q}
\right)^{1/2}
\le C_q^\prime\, r_{k+1}^{1/2},
\end{align*}
where the second inequality follows from $L^{q/2}$-Minkowski inequality and the last inequality uses the common $q$-th moment bound on $J_i$.
Consequently, we obtain $\Exp[\|R_{k,N}\|_F^q] \le (C_q^\prime)^q r_{k+1}^{q/2}$. 

Since ${\sum}_{i=0}^{\infty}\alpha_iJ_i$ converges almost surely,
we have $R_{k,N}\to S_{k+1}$ almost surely as $N\to\infty$. By Fatou's lemma, we deduce
\[
\Exp[\|S_{k+1}\|_F^q]
\le
{\liminf}_{N\to\infty}\,\Exp[\|R_{k,N}\|_F^q]
\le
(C_q^\prime)^q r_{k+1}^{q/2}.
\]
Using the independence of $S_{k+1}$ and $J_k$ and the bound
$\|AB\|_F\le\|A\|_F\|B\|_F$, we obtain
\[
\Exp[\|S_{k+1}J_k\|_F^q]
\le
\Exp[\|S_{k+1}\|_F^q]\,\Exp[\|J_k\|_F^q]
\le C_q^{\prime\prime} r_{k+1}^{q/2},\qquad \text{for some $C_q^{\prime\prime}<\infty$}.
\]
Markov's inequality and the summability of $r_{k+1}^{q/2}$ imply that, for
every $\varepsilon>0$,
\[
{\sum}_{k=0}^{\infty}\, 
\Prob\bigl(\|S_{k+1}J_k\|_F>\varepsilon\bigr)<\infty.
\]
Consequently, by Borel--Cantelli, we conclude that $\|S_{k+1}J_k\|_F\to0$ almost surely.
\end{proof}

\subsection{Sampling without replacement}
\label{app:sampling-without-replacement}

\begin{proof}[Proof of \cref{prop:without-replacement-implies-stochastic-oracle}]
The proof is divided into two parts.\\[-2mm]

\noindent\textbf{Part I: Equicontinuity.} Fix a compact set $\mathcal Q\subset\Rd$ and
$\bar\eta>0$, and let $\widehat{\mathcal Q}$ be the convex hull of
$\mathcal Q$. Since there are finitely many permutations and
$0\le\eta\le\bar\eta$ ranges over a compact interval, the inner iterates
\[
\{Z_s(z;\eta,\pi):z\in\widehat{\mathcal Q},\ 0\le\eta\le\bar\eta,\ 
s=0,\ldots,n,\ \pi \text{ is a permutation}\}
\]
lie in a compact convex set $\mathcal Q_+$. Let $M$ be a common bound on
$\|{\rm D}G_i\|$ and on the Lipschitz constants of ${\rm D}G_i$ on
$\mathcal Q_+$, for $i=1,\ldots,n$.

For fixed $(\eta,\pi)$, the chain rule gives
\begin{equation}\label{eq:prop sample w/o 0}
{\rm D}Z_s(z;\eta,\pi)
=
\bigl(I-\eta {\rm D}G_{\pi_s}(Z_{s-1}(z;\eta,\pi))\bigr)
{\rm D}Z_{s-1}(z;\eta,\pi) \qquad \text{with}
\qquad {\rm D}Z_0(z;\eta,\pi)=I.
\end{equation}
By induction, it holds that 
\begin{equation}\label{eq:prop sample w/o 0/1}
\|{\rm D}Z_s(z;\eta,\pi)\|\le(1+\bar\eta M)^s\qquad \text{for all}\quad z\in\widehat{\mathcal Q}.
\end{equation} 
Since the segment between any two points $z,w\in\mathcal Q$ is contained in
$\widehat{\mathcal Q}$,
\begin{equation}\label{eq:prop sample w/o 1}
    \begin{aligned}
   \|Z_s(z;\eta,\pi)-Z_s(w;\eta,\pi)\|
&=\left\|\int_0^1
{\rm D}Z_s\big(w+t(z-w);\eta,\pi\big)(z-w)\,{\rm d}t  \right\|\\
&\leq 
\int_0^1
\|{\rm D}Z_s\big(w+t(z-w);\eta,\pi\big)\|\,{\rm d} t \cdot \|z-w\| \leq
(1+\bar\eta M)^s\|z-w\|.
\end{aligned}
\end{equation}
We also need the same type of bound for the derivative of $Z_s(\cdot ;\eta,\pi)$. The case $s=0$ is
trivial. For $s\ge1$, subtracting the recursion
\eqref{eq:prop sample w/o 0} at $z$ and $w$ gives
\begin{align*}
&\quad \|{\rm D}Z_s(z;\eta,\pi)-{\rm D}Z_s(w;\eta,\pi)\|
\\ &\le
(1+\bar\eta M)
\|{\rm D}Z_{s-1}(z;\eta,\pi)-{\rm D}Z_{s-1}(w;\eta,\pi)\|\\
&\quad
+\eta
\|{\rm D}G_{\pi_s}(Z_{s-1}(z;\eta,\pi))
-{\rm D}G_{\pi_s}(Z_{s-1}(w;\eta,\pi))\|
\|{\rm D}Z_{s-1}(w;\eta,\pi)\|\\
&\overset{\text{(i)}}{\leq} (1+\bar\eta M)
\|{\rm D}Z_{s-1}(z;\eta,\pi)-{\rm D}Z_{s-1}(w;\eta,\pi)\| 
+M\eta (1+\bar\eta M)^{s-1}
\|Z_{s-1}(z;\eta,\pi) - Z_{s-1}(w;\eta,\pi)\|\\
&\overset{\text{(ii)}}{\leq} (1+\bar\eta M)
\|{\rm D}Z_{s-1}(z;\eta,\pi)-{\rm D}Z_{s-1}(w;\eta,\pi)\|
+M\eta (1+\bar\eta M)^{2s-2} \|z-w\|,
\end{align*}
where (i) uses the bound
\eqref{eq:prop sample w/o 0/1} and the
Lipschitz continuity of ${\rm D}G_{\pi_s}$ on $\mathcal Q_+$, and (ii) follows
from \eqref{eq:prop sample w/o 1}. Unfolding this recursion and using
$\eta\le\bar\eta$ yields
\begin{equation}\label{eq:prop sample w/o 2}
\|{\rm D}Z_s(z;\eta,\pi)-{\rm D}Z_s(w;\eta,\pi)\|
\le
M\bar\eta
{\sum}_{\ell=1}^{n}(1+\bar\eta M)^{n+\ell-2}\|z-w\|
=:C_M\|z-w\|.
\end{equation}
Now, we show the equicontinuity of
\[
    \{ {\rm D}G(\cdot;\eta,\pi):0<\eta\le \bar\eta,\ 
    \pi \text{ is a permutation of }\{1,\ldots,n\}\}.
\]
By the chain rule, $
{\rm D}G(z;\eta,\pi)
=
\frac1n{\sum}_{s=1}^n
{\rm D}G_{\pi_s}(Z_{s-1}(z;\eta,\pi)){\rm D}Z_{s-1}(z;\eta,\pi)$. 
For any $z,w\in\mathcal Q$, each summand satisfies
\begin{align*}
&\quad
\|{\rm D}G_{\pi_s}(Z_{s-1}(z;\eta,\pi)){\rm D}Z_{s-1}(z;\eta,\pi)
-
{\rm D}G_{\pi_s}(Z_{s-1}(w;\eta,\pi)){\rm D}Z_{s-1}(w;\eta,\pi)\|\\
&\le
\|{\rm D}G_{\pi_s}(Z_{s-1}(z;\eta,\pi))\|
\|{\rm D}Z_{s-1}(z;\eta,\pi)-{\rm D}Z_{s-1}(w;\eta,\pi)\|\\
&\qquad+
\|{\rm D}G_{\pi_s}(Z_{s-1}(z;\eta,\pi))
-{\rm D}G_{\pi_s}(Z_{s-1}(w;\eta,\pi))\|
\|{\rm D}Z_{s-1}(w;\eta,\pi)\|\\
&\le
\left[
MC_M+M(1+\bar\eta M)^{2s-2}
\right]\|z-w\|,
\end{align*}
where the last line holds due to $\|{\rm D}G_i(z)\|\leq M$ and the estimates \eqref{eq:prop sample w/o 0/1}--\eqref{eq:prop sample w/o 2}.
Therefore, ${\rm D}G(\cdot;\eta,\pi)$ is Lipschitz on $\mathcal Q$ with a
constant independent of $\eta$ and $\pi$. This proves the equicontinuity
claim.\\[-2mm]

\noindent \textbf{Part II: Stochastic oracle estimates.} 
Fix $\bar z\in\Rd$. The estimates below are evaluated at this candidate point and
do not involve the orbit generated by any algorithm. Choose $r>0$ and let
$\mathcal Q:=\{z:\|z-\bar z\|\le r\}$. Let $M$ be a common bound on
$\|G_i\|$, $\|{\rm D}G_i\|$, and the local Lipschitz constants of $G_i$ and
${\rm D}G_i$ on $\mathcal Q$, for all $i=1,\ldots,n$. Since $\eta_k\to0$, it
suffices to estimate all sufficiently large $k$, because this does not affect
the convergence of the series. In the subsequent analysis, we consider $k$
large enough such that 
\begin{equation}\label{eq:prop sample w/o 3}
    n\eta_kM\le r/2\qquad \text{and} \qquad 
(1+\eta_kM)^{n-1}\le2.
\end{equation}
We first show that the inner iterates remain in $\mathcal Q$. The claim is
trivial for $s=0$. If $Z_0^k(\bar z),\ldots,Z_{s-1}^k(\bar z)\in\mathcal Q$,
then the inner recursion gives $
Z_s^k(\bar z)-\bar z
=
-\eta_k{\sum}_{\ell=1}^{s}
G_{\pi_\ell^k}(Z_{\ell-1}^k(\bar z))$.
Hence,
\begin{equation}\label{eq:prop sample w/o 4}
\|Z_s^k(\bar z)-\bar z\|
\le \eta_k{\sum}_{\ell=1}^{s}
\|G_{\pi_\ell^k}(Z_{\ell-1}^k(\bar z))\|
\le n\eta_kM.
\end{equation}
By induction and the bound in \eqref{eq:prop sample w/o 3}, we have $\|Z_s^k(\bar z)-\bar z\|\leq r/2$, showing that $Z_s^k(\bar z)\in\mathcal Q$ for every $s=0,\ldots,n$. 

Since $\{\pi_1^k,\ldots,\pi_n^k\}=\{1,\ldots,n\}$, we have $
F(\bar z)=\frac1n{\sum}_{s=1}^nG_{\pi_s^k}(\bar z)$ and 
${\rm D}F(\bar z)=\frac1n{\sum}_{s=1}^n{\rm D}G_{\pi_s^k}(\bar z)$. 
Thus, by the definition of the epoch oracle,
\[
G(\bar z;\xi^k)-F(\bar z)
=
\frac1n{\sum}_{s=1}^n
\bigl[G_{\pi_s^k}(Z_{s-1}^k(\bar z))-G_{\pi_s^k}(\bar z)\bigr].
\]
The local Lipschitz continuity and the bound \eqref{eq:prop sample w/o 4} imply 
\begin{equation}
    \label{eq:prop sample-w/o-replace 1}
\|G(\bar z;\xi^k)-F(\bar z)\|
\le
\frac{M}{n}{\sum}_{s=1}^n\|Z_{s-1}^k(\bar z)-\bar z\|
\le
M^2\alpha_k \qquad \text{where}\quad \alpha_k:=n\eta_k.    
\end{equation}
For the derivative, the chain rule gives $
{\rm D}G(\bar z;\xi^k)
=
\frac1n{\sum}_{s=1}^{n}{\rm D}G_{\pi_s^k}(Z_{s-1}^k(\bar z))\,{\rm D}Z_{s-1}^k(\bar z)$. Hence, 
\begin{align}
\|{\rm D}G(\bar z;\xi^k)-{\rm D}F(\bar z)\|
&\leq 
\frac1n{\sum}_{s=1}^n\left[ 
\|{\rm D}G_{\pi_s^k}(Z_{s-1}^k(\bar z))-{\rm D}G_{\pi_s^k}(\bar z)\| + \|{\rm D}G_{\pi_s^k}(Z_{s-1}^k(\bar z))\|\|{\rm D}Z_{s-1}^k(\bar z)-I\|\right] \notag\\
 &\overset{\text{(i)}}{\leq}  \frac1n{\sum}_{s=1}^n \left[ M \|Z_{s-1}^k(\bar z) - \bar z\| + M\|{\rm D}Z_{s-1}^k(\bar z)-I\|\right]  \notag\\ \label{eq:prop sample-w/o-replace 2} &\overset{\text{(ii)}}{\leq}  M^2\alpha_k + M\cdot \max_{s=1,\ldots,n}\|{\rm D}Z_{s}^k(\bar z)-I\|,
\end{align}
where (i) uses the bound $\|{\rm D}G_i\|\le M$ and the Lipschitz continuity
of ${\rm D}G_i$ on $\mathcal Q$, while (ii) uses
\eqref{eq:prop sample w/o 4}. It remains to control the derivative term in
\eqref{eq:prop sample-w/o-replace 2}. Differentiating the inner recursion
gives, for $s=1,\ldots,n$,
\[
{\rm D}Z_s^k(\bar z)
=
[I-\eta_k{\rm D}G_{\pi_s^k}(Z_{s-1}^k(\bar z))]
{\rm D}Z_{s-1}^k(\bar z),
\qquad
{\rm D}Z_0^k(\bar z)=I.
\]
Since $Z_{s-1}^k(\bar z)\in \mathcal Q$, we have
$\|{\rm D}G_{\pi_s^k}(Z_{s-1}^k(\bar z))\|\le M$.
Therefore, $\|{\rm D}Z_s^k(\bar z)\|\le(1+\eta_kM)^s$. Moreover, we have for all $s=1,\ldots,n$,
\begin{align*}
\|{\rm D}Z_s^k(\bar z)-I\|
&= \|{\rm D}Z_s^k(\bar z)-{\rm D}Z_0^k(\bar z)\| \\
&\leq  {\sum}_{\ell=1}^{s}
\|{\rm D}Z_\ell^k(\bar z)-{\rm D}Z_{\ell-1}^k(\bar z)\| \le \eta_k {\sum}_{\ell=1}^{s}
\|{\rm D}G_{\pi_{\ell}^k}(Z_{\ell-1}^k(\bar z))\|\cdot \|{\rm D}Z_{\ell-1}^k(\bar z)\|\\
&\leq \eta_kM{\sum}_{\ell=1}^{s}(1+\eta_kM)^{\ell-1}
\le s\eta_kM(1+\eta_kM)^{n-1} \leq 2n\eta_kM,
\end{align*}
where the last inequality follows from \eqref{eq:prop sample w/o 3}. Combining this estimate with \eqref{eq:prop sample-w/o-replace 2}, we obtain 
\[
\|{\rm D}G(\bar z;\xi^k)-{\rm D}F(\bar z)\| \leq 3M^2 \alpha_k.
\]
Thus, for $C_{\bar z}:=4M^2$, the quantities
$b_k:=G(\bar z;\xi^k)-F(\bar z)$ and
$J_k:={\rm D}G(\bar z;\xi^k)-{\rm D}F(\bar z)$ satisfy
$\|b_k\|+\|J_k\|\le C_{\bar z}\alpha_k$. Hence,
${\sum}_{k=0}^\infty \alpha_kb_k$ and
${\sum}_{k=0}^\infty \alpha_kJ_k$ converge. In addition, with
$S_k:={\sum}_{i=k}^{\infty}\alpha_iJ_i$, we can deduce
\[
\|S_k\|\le C_{\bar z}{\sum}_{i=k}^{\infty}\alpha_i^2\to0.
\]
Since $\sum_{k=0}^\infty \alpha_k^2<\infty$ and $\|b_{k}\| = \cO(\alpha_k)$, the series ${\sum}_{k=0}^\infty \alpha_kS_{k+1}b_k$ converges absolutely. Finally,
\[
\|S_{k+1}J_k\|
\le
C_{\bar z}^2
\left({\sum}_{i=k+1}^{\infty}\alpha_i^2\right)
\alpha_k
\to0.
\]
As a result, the event $\cE_{\rm reg}(\bar z)$ in \cref{assumption:stochastic-oracle} can be taken as the whole seed space.
\end{proof}

\section{Proofs of technical lemmas}
\subsection{\texorpdfstring{Proof of \cref{lem:Lip of R_k}}{Proof of uniform Taylor remainder bound}}
\label{subsec:proof-Lip-of-R-k}
\begin{proof}
Recall that in \cref{subsec:shifted-dynamics} we translated $z^*$ to the origin and retained the notation $G$ for the maps $u\mapsto G(z^*+u;\xi)$. Let $\cU_{z^*}$ be the open neighborhood from \cref{assumption:F}. Choose $\bar\rho>0$ such that $\overline B_{2\bar\rho}(0)\subset\cU_{z^*}-z^*$ and define
\[
\Upsilon(r):=
\sup_{\xi\in\Xi}\sup_{\|u\|\le r}
\|{\rm D}G(u;\xi)-{\rm D}G(0;\xi)\|,
\qquad 0\le r\le2\bar\rho.
\]
We also extend $\Upsilon$ constantly to $[2\bar\rho,\infty)$. By local equicontinuity, $\Upsilon(r)<\infty$ for $r\le2\bar\rho$ and $\Upsilon(r)\to0$ as $r\to0$.
For any $w,z\in B_\rho(0)$ with
$\rho\in(0,2\bar\rho]$, we have
\begin{align}
\notag\|R_k(z)-R_k(w)\| &= \left\|\int_0^1 {\rm D}R_k(w+t(z-w)) (z-w) {\rm d}t\right\| \leq \|z-w\| \int_0^1 \|{\rm D}R_k(w+t(z-w))\|  {\rm d}t \\[1mm]
  \label{eq:Lipschitz R_k} &\leq \|z-w\| \cdot \sup_{\|u\|\leq \rho}\|{\rm D}R_k(u)\|  = \|z-w\| \cdot \sup_{\|u\|\leq \rho}\|{\rm D}g_k(u)-{\rm D}g_k(0)\|.
\end{align}
Moreover, for every $\|z\|\le\rho$, it holds that
\[
\|{\rm D}R_k(z)\|
=\|{\rm D}G(z;\xi^k(\omega_\xi))-{\rm D}G(0;\xi^k(\omega_\xi))\|
\le\Upsilon(\rho).
\]
Combining this estimate and the definition of $\Upsilon$ with \eqref{eq:Lipschitz R_k} completes the proof.
\end{proof}

\subsection{\texorpdfstring{Proof of \cref{lem:psi_properties}}{Proof of local properties of the residual term}}
\label{subsec:proof-psi-properties}
\begin{proof}  
By \eqref{eq:U_k bounds}, we have $\|U_k\|\leq \frac32$ and $\|U_k^{-1}\|\leq 2$ for all $k\in\N$. It follows from the triangle inequality that
\begin{equation}\label{eq:tilde R_k 1}
\begin{aligned}
        \|\widetilde{R}_k(y) - \widetilde{R}_k(w)\| &\leq \|U_{k+1}\| \left\|R_k(U_k^{-1}y) - R_k(U_k^{-1}w) \right\| + \|B_k\| \left\|U_k^{-1}\right\| \|y-w\|\\
        &\leq \frac32\left\|R_k(U_k^{-1}y) - R_k(U_k^{-1}w) \right\| + 2 \|B_k\| \|y-w\|.
\end{aligned}
\end{equation}
First, we upper bound the term $\|R_k(U_k^{-1}y) - R_k(U_k^{-1}w)\|$. For any $w,y\in B_\rho(0)$ with $\rho\in(0,\bar\rho]$, since $\|U_k^{-1}\|\leq 2$, we have $U_k^{-1}w,U_k^{-1}y\in B_{2\rho}(0)$. As $\rho\leq \bar\rho$, \cref{lem:Lip of R_k} applies on $B_{2\rho}(0)$.
\begin{equation}\label{eq:tilde R_k 2}
\|R_k(U_k^{-1}y) - R_k(U_k^{-1}w)\| \leq \Upsilon(2\rho)\|U_k^{-1}\|\|y-w\| \leq 2\Upsilon(2\rho)\|y-w\|.
\end{equation}
Based on the definition of $B_k$ in \eqref{eq:def B_k}, we have
\begin{equation}\label{eq:tilde R_k 3}
\|B_k\| \leq \|S_{k+1}\|\|H\| + \|S_{k+1}J_k\| + \|H\| \|S_k\|.
\end{equation}
Merging \eqref{eq:tilde R_k 2}--\eqref{eq:tilde R_k 3} into \eqref{eq:tilde R_k 1} leads to the desired result.
\end{proof}

\subsection{\texorpdfstring{Proof of \cref{lem:growth_rates_nonsym}}{Proof of transition bounds on spectral subspaces}}
\label{subsec:proof-growth-rates-nonsym}
\begin{proof}
\textbf{Step 1: Spectral splitting.} Let $\Lambda_u:=\{\lambda\in\operatorname{spec}(H):\Re(\lambda)<0\}$ and $\Lambda_{\cs}:=\{\lambda\in\operatorname{spec}(H):\Re(\lambda)\ge0\}$. Since $H$ is real, its nonreal eigenvalues appear in conjugate pairs, which have the same real part. The real Jordan form for real matrices \cite[Section 3.4]{HornJohnson2012} allows the real blocks to be grouped by $\Lambda_u$ and $\Lambda_{\cs}$. Thus, there is a real nonsingular matrix $Q$ such that
\[
Q^{-1}HQ=\diag(H_u,H_{\cs})\qquad \text{and} \qquad
\operatorname{spec}(H_u)=\Lambda_u,\quad
\operatorname{spec}(H_{\cs})=\Lambda_{\cs}.
\]
The same $Q$ block diagonalizes every $I-\alpha_\ell H$ and hence every transition product. To distinguish the coordinate blocks from the full-space projected operators in \eqref{eq:restricted-transition-operators}, define, for $k\ge j$,
\[
\widehat\Phi_u^+(k,j):={\prod}_{\ell=j}^{k-1}(I-\alpha_\ell H_u)
\qquad \text{and}\qquad 
\widehat\Phi_{\cs}^+(k,j):={\prod}_{\ell=j}^{k-1}(I-\alpha_\ell H_{\cs}),
\]
with the convention that an empty product is the identity, and the inverse transition operators as $\widehat\Phi_u^-(k,j):=\bigl(\widehat\Phi_u^+(k,j)\bigr)^{-1}$, $\widehat\Phi_\cs^-(k,j):=\bigl(\widehat\Phi_\cs^+(k,j)\bigr)^{-1}$. Then, it holds that
\begin{equation}\label{eq:spectral-block-transition}
Q^{-1}\Phi^\pm(k,j)Q
= \diag(\widehat\Phi_u^\pm(k,j), \widehat\Phi_{\cs}^\pm(k,j))
\end{equation}
\textbf{Step 2: Jordan formulation.} 
After the real splitting is fixed as \eqref{eq:spectral-block-transition}, we complexify the blocks \(H_u\) and
\(H_{\cs}\) and apply the complex Jordan form inside each block.
By the Jordan canonical form theorem \cite[Theorem 3.1.5]{HornJohnson2012}, we may choose fixed complex nonsingular matrices $Z_u,Z_{\cs}$ such that
\begin{equation}\label{eq:jordan-block-reduction}
Z_u^{-1}H_uZ_u
=J_u:=\diag\bigl(J_{u,1},\ldots,J_{u,p}\bigr) \qquad \text{and} \qquad 
Z_{\cs}^{-1}H_{\cs}Z_{\cs}
=J_{\cs}:=\diag\bigl(J_{\cs,1},\ldots,J_{\cs,q}\bigr),
\end{equation}
where the individual Jordan blocks are
\begin{align*}
    J_{u,a}&:=\lambda_{u,a}I_{m_{u,a}}+N_{u,a},
\qquad
N_{u,a}^{m_{u,a}}=0,
\qquad
\lambda_{u,a}\in\Lambda_u,
\qquad a=1,\ldots,p,\qquad \text{and}\\
J_{\cs,b}&:=\lambda_{\cs,b}I_{m_{\cs,b}}+N_{\cs,b},
\qquad
N_{\cs,b}^{m_{\cs,b}}=0,
\qquad
\lambda_{\cs,b}\in\Lambda_{\cs},
\qquad b=1,\ldots,q.
\end{align*}
The block sizes $m_{u,a}$, $m_{\cs,b}$, and the nilpotent matrices
$N_{u,a}$, $N_{\cs,b}$, may vary from block to block.
For $k\ge j$, the same similarities give
\begin{equation}\label{eq:jordan-transition-blocks}
Z_u^{-1}\widehat\Phi_u^+(k,j)Z_u={\prod}_{\ell=j}^{k-1}(I-\alpha_\ell J_u)
\qquad \text{and} \qquad
Z_{\cs}^{-1}\widehat\Phi_{\cs}^+(k,j)Z_{\cs}={\prod}_{\ell=j}^{k-1}(I-\alpha_\ell J_{\cs}).
\end{equation}
The inverses of the two identities give the corresponding formulas for $\widehat\Phi_u^-(k,j)$ and $\widehat\Phi_\cs^-(k,j)$. 

\smallskip
\noindent\textbf{Step 3: Generic Jordan block estimate.}
Since $\alpha_\ell\|H\|\le1/2$ and every eigenvalue of a Jordan block in
\eqref{eq:jordan-block-reduction} has modulus at most $\|H\|$,
\cref{lem:generic-jordan-block-transition} applies to every block $J_{u,a}$
and $J_{\cs,b}$. For each unstable block $J_{u,a}$ with \(\Re(\lambda_{u,a})<0\), 
\cref{lem:generic-jordan-block-transition} gives a rate $\nu_{u,a}>0$. Since
there are finitely many unstable blocks, the bounds hold with
$\nu:=\min_{1\le a\le p}\nu_{u,a}$. That is, for each
unstable block $J_{u,a}$, there is a constant $D_{u,a}^-\ge1$ such that
\[
\left\|\left({\prod}_{\ell=j}^{k-1}(I-\alpha_\ell J_{u,a})\right)^{-1}\right\|
\le D_{u,a}^-\exp(-\nu t_{k,j}),
\qquad k\ge j,
\] 
Choose $0<\kappa<\nu$. For each
center-stable block $J_{\cs,b}$ with \(\Re(\lambda_{\cs,b})\ge0\), \cref{lem:generic-jordan-block-transition} yields a constant $D_{\cs,b}^+\ge1$ such that
\[
\left\|{\prod}_{\ell=j}^{k-1}(I-\alpha_\ell J_{\cs,b})\right\|
\le D_{\cs,b}^+\exp(\kappa t_{k,j}),
\qquad k\ge j.
\]

\smallskip
\noindent\textbf{Step 4: Back to spectral subspaces.}
Define $D_H:=\max\{\max_{1\le a\le p} D_{u,a}^-,
\max_{1\le b\le q} D_{\cs,b}^+\}$ and 
\[
C_H:=D_H\max\{1,\ \|Q\|\|Q^{-1}\|\|Z_u\|\|Z_u^{-1}\|,
\ \|Q\|\|Q^{-1}\|\|Z_{\cs}\|\|Z_{\cs}^{-1}\|\}.
\]
From the diagonal formulation \eqref{eq:spectral-block-transition} and the definitions of
$\Phi_u^{-}(k,j)$ and $\Phi_{\cs}^{+}(k,j)$, we have
\[
Q^{-1}\Phi_u^{-}(k,j)Q=\diag\bigl((\widehat\Phi_u^+(k,j))^{-1},0\bigr)
\qquad \text{and} \qquad
Q^{-1}\Phi_{\cs}^{+}(k,j)Q=\diag(0,\widehat\Phi_{\cs}^+(k,j)).
\]
By \eqref{eq:jordan-transition-blocks} and the definition of $C_H$,
$\|\Phi_u^{-}(k,j)\|\le C_H\exp(-\nu t_{k,j})$ and
$\|\Phi_{\cs}^{+}(k,j)\|\le C_H\exp(\kappa t_{k,j})$ for all $k\ge j$.
\end{proof}

\subsection{\texorpdfstring{Proof of \cref{lem:global_psi_property}}{Proof of global Lipschitzness of truncated residuals}}
\label{subsec:proof-global-psi-property}
\begin{proof}
If $y,w\notin B_{2\delta}(0)$, then both truncated residuals are zero.
Suppose first that $y,w\in B_{2\delta}(0)$. Since $\widetilde R_k(0)=0$, we
have $\|\widetilde R_k(z)\|\le2\delta\ell_\delta$ on $B_{2\delta}(0)$, and
the triangle inequality gives
\[
\begin{aligned}
\|R_k^\delta(y)-R_k^\delta(w)\|
&\le \|\chi(y)(\widetilde R_k(y)-\widetilde R_k(w))\|
     +|\chi(y)-\chi(w)|\,\|\widetilde R_k(w)\|\\
&\le \ell_\delta\|y-w\|+(c_\chi/\delta)\|y-w\|\,(2\delta\ell_\delta).
\end{aligned}
\]
If $y\in B_{2\delta}(0)$ and $w\notin B_{2\delta}(0)$, then $\chi(w)=0$ and
\[
\|R_k^\delta(y)-R_k^\delta(w)\|
\le |\chi(y)-\chi(w)|\,\|\widetilde R_k(y)\|
\le 2c_\chi\ell_\delta\|y-w\|.
\]
The remaining case follows by symmetry.
\end{proof}

\subsection{\texorpdfstring{Proof of \cref{lem:well_define}}{Proof of well-definedness}}
\label{subsec: proof of well_define}
\begin{proof}
Fix $\omega_\xi\in\cE_{\rm reg}^*$. Since $\widetilde b^k=(I-S_{k+1})b_k$,
\cref{assumption:stochastic-oracle} implies that the tails
\[
\widetilde B_k:={\sum}_{i=k}^{\infty}\alpha_i\widetilde b^i \quad \text{satisfy}\quad C_{\widetilde b}:=\sup_{k\in\N}\|\widetilde B_k\|<\infty.
\] 
By the growth bounds in \cref{lem:growth_rates_nonsym}, i.e., for every $k\geq j$
\begin{equation}\label{eq:LP-transition-bounds}
\|\Phi_{\cs}^{+}(k,j)\|\le C_H\exp(\kappa t_{k,j})
\qquad \text{and} \qquad
\|\Phi_u^{-}(k,j)\|\le C_H\exp(-\nu t_{k,j}),
\end{equation}
the initial term satisfies
\[
\sup_{k\in\N}\; \exp(-\theta t_k)\|\Phi_{\cs}^{+}(k,0)\zeta\|
\le C_H\|\zeta\|\qquad \text{and} \qquad \kappa<\theta.
\]
Consider the center-stable convolution. Since
$\alpha_j\widetilde b^j=\widetilde B_j-\widetilde B_{j+1}$, summation by parts gives
\[
{\sum}_{j=0}^{k-1}\alpha_j\Phi_{\cs}^{+}(k,j+1)\widetilde b^j
=\Phi_{\cs}^{+}(k,1)\widetilde B_0-\Phi_{\cs}^{+}(k,k)\widetilde B_k 
+{\sum}_{j=1}^{k-1}
\bigl(\Phi_{\cs}^{+}(k,j+1)-\Phi_{\cs}^{+}(k,j)\bigr)\widetilde B_j.
\]
Since $
\Phi_{\cs}^{+}(k,j+1)-\Phi_{\cs}^{+}(k,j)
=\alpha_j\Phi_{\cs}^{+}(k,j+1)HP_{\cs}$, 
the triangle inequality and \eqref{eq:LP-transition-bounds} yield
\[
\sup_{k\in\N}\;\exp(-\theta t_k)
\Bigl\|{\sum}_{j=0}^{k-1}\alpha_j\Phi_{\cs}^{+}(k,j+1)\widetilde b^j\Bigr\|
\le
C_{\widetilde b}C_H\left(2+\frac{\|HP_{\cs}\|}{\kappa}\right).
\]
For the unstable tail convolution, summation by parts gives
\[
{\sum}_{j=k}^{\infty}\alpha_j\Phi_u^{-}(j+1,k)\widetilde b^j
=\Phi_u^{-}(k+1,k)\widetilde B_k
+{\sum}_{j=k+1}^{\infty}
\bigl(\Phi_u^{-}(j+1,k)-\Phi_u^{-}(j,k)\bigr)\widetilde B_j.
\]
Note also that $\Phi_u^{-}(j+1,k)-\Phi_u^{-}(j,k)=\alpha_j\Phi_u^{-}(j+1,k)HP_u$. 
Using \eqref{eq:LP-transition-bounds} and
${\sum}_{j=k+1}^{\infty}\alpha_j\exp(-\nu t_{j+1,k})\le1/\nu$, we obtain
\[
\sup_{k\in\N}\exp(-\theta t_k)
\Bigl\|{\sum}_{j=k}^{\infty}\alpha_j\Phi_u^{-}(j+1,k)\widetilde b^j\Bigr\|
\le
C_{\widetilde b}C_H\left(1+\frac{\|HP_u\|}{\nu}\right).
\]
Thus, the affine terms in \eqref{eq:LP} are bounded in $\|\cdot\|_\theta$.

It remains to bound the nonlinear convolutions. If $\by\in\cY_\theta$, then $
\|R_j^\delta(y^j)\|\le L_\delta\|y^j\|
\le L_\delta\|\by\|_\theta\exp(\theta t_j)$. 
By \eqref{eq:LP-transition-bounds}, we can bound
\begin{align*}
\exp(-\theta t_k)
\Bigl\|{\sum}_{j=0}^{k-1}\alpha_j\Phi_{\cs}^{+}(k,j+1)R_j^\delta(y^j)\Bigr\| 
&\le C_HL_\delta\|\by\|_\theta
{\sum}_{j=0}^{k-1}\alpha_j
\exp\bigl(-\theta t_k+\kappa t_{k,j+1}+\theta t_j\bigr)\\
&\le C_HL_\delta\|\by\|_\theta
{\sum}_{j=0}^{k-1}\alpha_j\exp\bigl(-(\theta-\kappa)t_{k,j}\bigr)
\le \frac{C_HL_\delta}{\theta-\kappa}\|\by\|_\theta.
\end{align*}
The same argument on the unstable tail gives
\[
\begin{aligned}
&\quad \exp(-\theta t_k)
\Bigl\|{\sum}_{j=k}^{\infty}\alpha_j\Phi_u^{-}(j+1,k)R_j^\delta(y^j)\Bigr\|\\
&\le C_HL_\delta\|\by\|_\theta
{\sum}_{j=k}^{\infty}\alpha_j
\exp\!\bigl(-\nu t_{j+1,k}+\theta t_{j,k}\bigr)
\le \frac{C_HL_\delta}{\nu-\theta}\|\by\|_\theta.
\end{aligned}
\]
The two nonlinear convolution sequences are bounded in $\|\cdot\|_\theta$. Together with the estimates above, this proves $\cT_\zeta\by\in\cY_\theta$.
\end{proof}

\subsection{\texorpdfstring{Proof of \cref{lem:LP_contraction}}{Proof of contraction}}
\label{subsec: proof of LP_contraction}
\begin{proof}
Fix $\by,\bw\in\cY_\theta$. Since $R_j^\delta$ is $L_\delta$-Lipschitz and
$\|y^j-w^j\|\le \exp(\theta t_j)\|\by-\bw\|_\theta$, we have
\[
\|R_j^\delta(y^j)-R_j^\delta(w^j)\|
\le L_\delta\exp(\theta t_j)\|\by-\bw\|_\theta.
\]
The affine terms in \eqref{eq:LP} do not depend on $\by$ and therefore cancel.
Thus, for the center-stable component,
\[
(\cT_\zeta\by-\cT_\zeta\bw)_{\cs}^k
=-{\sum}_{j=0}^{k-1}\alpha_j
\Phi_{\cs}^{+}(k,j+1)
\bigl(R_j^\delta(y^j)-R_j^\delta(w^j)\bigr).
\]
By the bound $\|\Phi_{\cs}^{+}(k,j)\|\le C_H\exp(\kappa t_{k,j})$ (see \eqref{eq:LP-transition-bounds}), we obtain
\[
\begin{aligned}
\exp(-\theta t_k)\|(\cT_\zeta\by-\cT_\zeta\bw)_{\cs}^k\| &\le C_HL_\delta\|\by-\bw\|_\theta
{\sum}_{j=0}^{k-1}\alpha_j
\exp\bigl(-\theta t_k+\kappa t_{k,j+1}+\theta t_j\bigr)\\
&\le C_HL_\delta\|\by-\bw\|_\theta
{\sum}_{j=0}^{k-1}\alpha_j
\exp\bigl(-(\theta-\kappa)t_{k,j}\bigr).
\end{aligned}
\]
By summing over the intervals
$[t_{k,j+1},t_{k,j}]$ and using $t_{k,j}=t_{k,j+1}+\alpha_j$, we have
\[
{\sum}_{j=0}^{k-1}\alpha_j\exp\bigl(-(\theta-\kappa)t_{k,j}\bigr)
\le \int_{0}^{t_k}\exp\bigl(-(\theta-\kappa)s\bigr)\,ds
\le \frac{1}{\theta-\kappa}.
\]
Hence, it holds that
\begin{equation}\label{eq:lem contraction 1}
\exp(-\theta t_k)\|(\cT_\zeta\by-\cT_\zeta\bw)_{\cs}^k\|
\le \frac{C_HL_\delta}{\theta-\kappa}\|\by-\bw\|_\theta.
\end{equation}
For the unstable component, the cancellation of the affine terms gives
\[
(\cT_\zeta\by-\cT_\zeta\bw)_u^k
={\sum}_{j=k}^{\infty}\alpha_j
\Phi_u^{-}(j+1,k)
\bigl(R_j^\delta(y^j)-R_j^\delta(w^j)\bigr).
\]
Based on $\|\Phi_u^{-}(j+1,k)\|\le C_H\exp(-\nu t_{j+1,k})$ (see \eqref{eq:LP-transition-bounds}), we obtain 
\begin{align*}
\exp(-\theta t_k)\|(\cT_\zeta\by-\cT_\zeta\bw)_u^k\| 
&\le C_HL_\delta\|\by-\bw\|_\theta
{\sum}_{j=k}^{\infty}\alpha_j
\exp\bigl(-\theta t_k-\nu t_{j+1,k}+\theta t_j\bigr)\\
&=C_HL_\delta\|\by-\bw\|_\theta
{\sum}_{j=k}^{\infty}\alpha_j
\exp\bigl(-(\nu-\theta)t_{j,k}-\nu\alpha_j\bigr).
\end{align*}
Since $\nu>\nu-\theta>0$, we have $
\exp\bigl(-(\nu-\theta)t_{j,k}-\nu\alpha_j\bigr)
\leq
\exp\bigl(-(\nu-\theta)t_{j+1,k}\bigr)$. 
Therefore,
\[
{\sum}_{j=k}^{\infty}\alpha_j
\exp\bigl(-\theta t_k-\nu t_{j+1,k}+\theta t_j\bigr)
\le
{\sum}_{j=k}^{\infty}\alpha_j\exp\bigl(-(\nu-\theta)t_{j+1,k}\bigr)
\le \int_{0}^{\infty}\exp\bigl(-(\nu-\theta)s\bigr)\,ds
=\frac{1}{\nu-\theta}.
\]
The second inequality follows from the interval comparison over
$[t_{j,k},t_{j+1,k}]$.
Hence,
\begin{equation}\label{eq:lem contraction 2}
\exp(-\theta t_k)\|(\cT_\zeta\by-\cT_\zeta\bw)_u^k\|
\le \frac{C_HL_\delta}{\nu-\theta}\|\by-\bw\|_\theta.
\end{equation}
Combining \eqref{eq:lem contraction 1}--\eqref{eq:lem contraction 2} and using the triangle inequality, we obtain
\begin{equation}\label{eq:lem contraction 3}
\exp(-\theta t_k)\|(\cT_\zeta\by-\cT_\zeta\bw)^k\|
\le
C_HL_\delta\left(\frac{1}{\theta-\kappa}
+\frac{1}{\nu-\theta}\right)\|\by-\bw\|_\theta.
\end{equation}
Taking the supremum of \eqref{eq:lem contraction 3} over $k\in\N$ gives \eqref{eq:def q}.
\end{proof}

\subsection{\texorpdfstring{Proof of \cref{claim:Lipschitz-h}}{Proof of Lipschitz continuity of the graph map}}
\label{subsec: proof of Lipschitz h}
\begin{proof}
Fix $\zeta_1,\zeta_2\in E_{\cs}$, and let $\by_{\zeta_i}$ be the fixed point of $\cT_{\zeta_i}$. Then
\[
\by_{\zeta_1}-\by_{\zeta_2}
=\bigl[\cT_{\zeta_1}(\by_{\zeta_1})-\cT_{\zeta_1}(\by_{\zeta_2})\bigr]
+\bigl[\cT_{\zeta_1}(\by_{\zeta_2})-\cT_{\zeta_2}(\by_{\zeta_2})\bigr].
\]
Using \cref{lem:LP_contraction} and rearranging,
\[
(1-q)\|\by_{\zeta_1}-\by_{\zeta_2}\|_\theta
\le \|\cT_{\zeta_1}(\by_{\zeta_2})-\cT_{\zeta_2}(\by_{\zeta_2})\|_\theta.
\]
The two operators differ only in the initial center-stable term, so \eqref{eq:LP-transition-bounds} gives
\[
\|\cT_{\zeta_1}(\by_{\zeta_2})-\cT_{\zeta_2}(\by_{\zeta_2})\|_\theta
\le \sup_{k\in\N}\exp(-\theta t_k)C_H\exp(\kappa t_k)\|\zeta_1-\zeta_2\|
\le C_H\|\zeta_1-\zeta_2\|.
\] 
Since $h(\zeta)=P_u y_\zeta^0$, we have $\|h(\zeta_1)-h(\zeta_2)\|
\le \|P_u\|\,\|\by_{\zeta_1}-\by_{\zeta_2}\|_\theta
\le \frac{C_H\|P_u\|}{1-q}\|\zeta_1-\zeta_2\|$.
\end{proof}

\bibliographystyle{siam}
\bibliography{references}
\end{document}

%% file: config.tex
\usepackage[applemac]{inputenc} 		% allow utf-8 input
\usepackage[T1]{fontenc}    		% use 8-bit T1 fonts
\usepackage[colorlinks]{hyperref}      % hyperlinks
\usepackage{url}            			% simple URL typesetting
\usepackage{amsthm}
\usepackage{booktabs}       		% professional-quality tables
\usepackage{amsfonts}       		% blackboard math symbols
\usepackage{amsmath}
\usepackage{aliascnt}

\usepackage{nicefrac}       		% compact symbols for 1/2, etc.
\usepackage{microtype}      		% microtypography
\usepackage{lipsum}				% Can be removed after putting your text content
\usepackage[square,numbers]{natbib}
\usepackage{mathtools}
\usepackage{algorithm}
\usepackage{algorithmicx}
\usepackage{algpseudocode}
\makeatletter
\providecommand*{\theHALG@line}{}
\renewcommand*{\theHALG@line}{\thealgorithm.\arabic{ALG@line}}
\makeatother
\usepackage{graphicx}
\usepackage{indentfirst,latexsym,bm}
\usepackage{amsmath}
\usepackage{subfigure}
\usepackage{amssymb}
\usepackage{xcolor}
\usepackage{comment}
\usepackage{enumitem}
\usepackage{bbm}
\usepackage{tikz}
\usepackage{mdframed}
\usepackage{nicematrix}
\usepackage{bbding}
\usepackage{pifont}

\usetikzlibrary{arrows.meta,positioning}
\usepackage{pgfplots}
\pgfplotsset{compat=newest}
\usepgfplotslibrary{fillbetween}
\usetikzlibrary{shapes,decorations}
\usetikzlibrary{fit}

\colorlet{color1}{blue}
\colorlet{color2}{red!50!black}

\definecolor{ivory}{RGB}{218,215,203}

\definecolor{cuhkp}{RGB}{98,56,105} 	% purple dark
\definecolor{cuhkpl}{RGB}{152,24,147} 	% purple light
\definecolor{cuhkb}{RGB}{219,160,1} 	% ocher
\definecolor{cuhkbd}{RGB}{178,129,0} 	% ocher dark
\definecolor{cuhkr}{RGB}{88,35,155}  	% magenta-red
\definecolor{blackp}{RGB}{0,0,0} 
\definecolor{redp}{RGB}{255,0,0}
\definecolor{orangep}{RGB}{255,128,0}
\definecolor{brownp}{RGB}{128,77,0}
\definecolor{yellowp}{RGB}{255,230,0}
\definecolor{greenp}{RGB}{128,230,0}
\definecolor{bluep}{RGB}{0,128,255}
\definecolor{purplep}{RGB}{152,24,147}
\definecolor{pinkp}{RGB}{230,0,128}
\definecolor{lavender}{rgb}{0.9, 0.9, 0.98}

\usepackage{hyperref}[6.83]

\hypersetup{
    colorlinks=true, 
    linkcolor=blue,  
    citecolor=blue, 
    urlcolor=magenta 
}

\RequirePackage[capitalize,nameinlink]{cleveref}

\crefformat{equation}{\textup{#2(#1)#3}}
\crefrangeformat{equation}{\textup{#3(#1)#4--#5(#2)#6}}
\crefmultiformat{equation}{\textup{#2(#1)#3}}{ and \textup{#2(#1)#3}}
{, \textup{#2(#1)#3}}{, and \textup{#2(#1)#3}}
\crefrangemultiformat{equation}{\textup{#3(#1)#4--#5(#2)#6}}%
{ and \textup{#3(#1)#4--#5(#2)#6}}{, \textup{#3(#1)#4--#5(#2)#6}}{, and \textup{#3(#1)#4--#5(#2)#6}}

\Crefformat{equation}{#2Equation~\textup{(#1)}#3}
\Crefrangeformat{equation}{Equations~\textup{#3(#1)#4--#5(#2)#6}}
\Crefmultiformat{equation}{Equations~\textup{#2(#1)#3}}{ and \textup{#2(#1)#3}}
{, \textup{#2(#1)#3}}{, and \textup{#2(#1)#3}}
\Crefrangemultiformat{equation}{Equations~\textup{#3(#1)#4--#5(#2)#6}}%
{ and \textup{#3(#1)#4--#5(#2)#6}}{, \textup{#3(#1)#4--#5(#2)#6}}{, and \textup{#3(#1)#4--#5(#2)#6}}

\crefdefaultlabelformat{#2\textup{#1}#3}

\theoremstyle{plain}
\newtheorem{theorem}{Theorem}[section]

\newaliascnt{lemma}{theorem}
\newtheorem{lemma}[lemma]{Lemma}
\aliascntresetthe{lemma}

\newaliascnt{corollary}{theorem}
\newtheorem{corollary}[corollary]{Corollary}
\aliascntresetthe{corollary}

\newaliascnt{proposition}{theorem}
\newtheorem{proposition}[proposition]{Proposition}
\aliascntresetthe{proposition}

\newaliascnt{claim}{theorem}
\newtheorem{claim}[claim]{Claim}
\aliascntresetthe{claim}

\theoremstyle{remark}
\newaliascnt{remark}{theorem}
\newtheorem{remark}[remark]{Remark}
\aliascntresetthe{remark}

\theoremstyle{definition}
\newaliascnt{assumption}{theorem}
\newtheorem{assumption}[assumption]{Assumption}
\aliascntresetthe{assumption}

\newaliascnt{definition}{theorem}
\newtheorem{definition}[definition]{Definition}
\aliascntresetthe{definition}

\newaliascnt{example}{theorem}
\newtheorem{example}[example]{Example}
\aliascntresetthe{example}

\crefname{theorem}{Theorem}{Theorems}
\Crefname{theorem}{Theorem}{Theorems}

\crefname{lemma}{Lemma}{Lemmas}
\Crefname{lemma}{Lemma}{Lemmas}

\crefname{corollary}{Corollary}{Corollaries}
\Crefname{corollary}{Corollary}{Corollaries}

\crefname{proposition}{Proposition}{Propositions}
\Crefname{proposition}{Proposition}{Propositions}

\crefname{claim}{Claim}{Claims}
\Crefname{claim}{Claim}{Claims}

\crefname{remark}{Remark}{Remarks}
\Crefname{remark}{Remark}{Remarks}

\crefname{assumption}{Assumption}{Assumptions}
\Crefname{assumption}{Assumption}{Assumptions}

\crefname{definition}{Definition}{Definitions}
\Crefname{definition}{Definition}{Definitions}

\crefname{example}{Example}{Examples}
\Crefname{example}{Example}{Examples}

\DeclareMathOperator*{\argmin}{argmin}

\newcommand{\cs}{\mathrm{cs}}

\newcommand{\cC}{\mathcal{C}}
\newcommand{\cE}{\mathcal{E}}
\newcommand{\cZ}{\mathcal{Z}}
\newcommand{\cY}{\mathcal{Y}}

\newcommand{\cA}{\mathcal{A}}
\newcommand{\cF}{\mathcal{F}}
\newcommand{\cO}{\mathcal{O}}

\newcommand{\cM}{\mathcal{M}}

\newcommand{\R}{\mathbb{R}}

\newcommand{\sL}{{\sf L}}

\newcommand{\bw}{\mathbf{w}}
\newcommand{\by}{\mathbf{y}}

\newcommand{\cU}{\mathcal{U}}

\newcommand{\cT}{\mathcal{T}}
\newcommand{\cX}{\mathcal{X}}
\newcommand{\D}{\mathbb{D}}
\newcommand{\N}{\mathbb{N}}

\newcommand{\Rd}{\mathbb{R}^d}

\newcommand{\Prob}{\mathbb{P}}
\newcommand{\Exp}{\mathbb{E}}
\newcommand{\vp}{\varphi}

\newcommand{\dist}{\mathrm{dist}}

\newcommand{\cS}{\mathcal{S}}
\newcommand{\cG}{\mathcal{G}}

\newcommand{\iprod}[2]{\left\langle #1, #2 \right\rangle}

\newcommand{\be}{\begin{equation}}
\newcommand{\ee}{\end{equation}}

\newcommand\prox[1]{\mathrm{prox}_{#1}}

\newcommand\diag{\mathrm{diag}}

\newcommand{\xmarkt}{\color{black}\ding{55}}
\newcommand{\cmarkt}{\color{bluep}\ding{51}}